\documentclass[reqno,11pt]{amsart}
\usepackage[english]{babel}
\usepackage[dvipsnames,prologue]{xcolor} 
\usepackage{calc,graphicx,amsfonts,amsthm,amscd,amsmath,amssymb,enumerate,dsfont,mathrsfs,paralist,bbm}
\usepackage{verbatim}
\usepackage{caption}
\usepackage{microtype}
\usepackage{subcaption}
\usepackage{pdfsync}
\usepackage{stmaryrd}
\usepackage{etoolbox}
\usepackage[numeric,initials,nobysame]{amsrefs}
\usepackage{marginnote}
\usepackage{hyperref}
\usepackage{textcase}
\usepackage[capitalize]{cleveref}
\usepackage{autonum}
\AtBeginEnvironment{biblist}{\catcode`\#=12 }

\usepackage{lmodern}
\usepackage{subcaption}
\usepackage[utf8]{inputenc}
\usepackage{chngcntr}\usepackage{apptools}
\usepackage{mleftright}

\renewcommand{\left}{\mleft}
\renewcommand{\right}{\mright}
\AtAppendix{\counterwithin{Alemma}{section}}

\usepackage{tikz}
\usepackage{floatrow}
\usepackage{xcolor}
\usetikzlibrary{arrows}
\usetikzlibrary[patterns]
\usetikzlibrary{decorations.pathreplacing}
\usetikzlibrary{patterns.meta}
\usepackage{tkz-euclide}
\definecolor{ffqqqq}{rgb}{1,0,0}
\definecolor{qqffqq}{rgb}{0,1,0}
\definecolor{ffffff}{rgb}{1,1,1}
\definecolor{ttqqqq}{rgb}{0.07,0.07,0.07}
\colorlet{ColorGray}{gray!30}
\usetikzlibrary{shapes.misc}
\usetikzlibrary{shapes.geometric}
\usetikzlibrary{arrows.meta}
\usetikzlibrary{decorations}
\usetikzlibrary{spy}

\tikzset{cross/.style={cross out, draw=black, minimum size=2*(#1-\pgflinewidth), inner sep=0pt, outer sep=0pt}, cross/.default={1pt}}

\usepackage{booktabs}
\reversemarginpar
\newlength\fullwidth
\numberwithin{equation}{section}

\DeclareMathSymbol{\leqslant}{\mathalpha}{AMSa}{"36} 
\DeclareMathSymbol{\geqslant}{\mathalpha}{AMSa}{"3E} 
\DeclareMathSymbol{\eset}{\mathalpha}{AMSb}{"3F}     

\def\1{\ifmmode {1\hskip -3pt \rm{I}} \else {\hbox {$1\hskip -3pt \rm{I}$}}\fi}

\newcommand{\var}{\operatorname{Var}}

\newcommand{\tc}{\thinspace |\thinspace}

\renewcommand{\L}{\Lambda}

\newcommand{\g}{\gamma}

\newcommand{\z}{\zeta}

\newtheorem{theorem}{Theorem}[section]
\newtheorem*{theorem*}{Theorem}
\newtheorem{lemma}[theorem]{Lemma}

\newtheorem{proposition}[theorem]{Proposition}
\newtheorem{corollary}[theorem]{Corollary}
\newtheorem{remark}[theorem]{Remark}

\newtheorem{definition}[theorem]{Definition}
\newtheorem{maintheorem}{Theorem}
\newtheorem*{question*}{Question}

\newtheorem*{remark*}{Remark}
\newtheorem*{idefinition*}{Definition}
\newtheorem{example}{Example}

\newcommand{\N}{\mathbb N}

\newcommand{\Z}{{\ensuremath{\mathbb Z}} }
\newcommand{\R}{{\ensuremath{\mathbb R}} }

\let\f=\varphi \let\g=\gamma

\let\y=\upsilon \let\x=\xi \let\z=\zeta
     \let\L=\Lambda

\newcommand{\qmed}{q_{\mathrm{med}}}

\newcommand{\supp}[1]{\mathrm{Supp}(#1)}

\title{The multicolour East model}
\author[Y. Couzini\'e]{Yannick~Couzini\'e}
\address{Dipartimento di Matematica e Fisica, Universit\`a Roma
  Tre}\email{yannick.couzinie@uniroma3.it}

\begin{document}
\begin{abstract}
        We consider the \emph{multicolour East model}, a model of glass forming
        liquids closely related to the East model on $\Z^d$. The state space
        ${(G\cup \{\star\})}^{\Z^d}$ consists of $|G|\le 2^d$ different
        \emph{vacancy types} and the \emph{neutral state} $\star$. To each
        $h\in G$ we associate unique facilitation mechanisms ${\{c_x^{h}\}}_{x\in
        \Z^d}$ that correspond to rotated versions of the East model
        constraints. If $c_x^{h}$ is satisfied, the state on $x$ can transition
        from $h$ to $\star$ with rate $p\in (0,1)$ or vice versa with rate
        $q_h\in (0,1)$, where generally $q_h\neq q_{h'}$ if $h'\neq h$.
        Notably, vertices in the state $h$ cannot transition directly to
        $h'\neq h$ and neighbouring $h'$-vacancies do not contribute in
        satisfying $c_x^{h}$. Thus, there is a novel blocking mechanism between
        vacancies of differing type. We find sufficient conditions on the model geometry to have a positive spectral gap and prove
        that with $|G|=2^d$ the model is not ergodic. For $d=2$ we prove
        that the model with $|G|\le 3$ has positive spectral gap and
        we find sufficient conditions on the transition rates for the spectral
        gap to be given in the leading order by the spectral gap of the East
        model on $\Z^2$ with parameter $q_{\min}=\min_{h\in G}q_h$ in the limit
        $q_{\min}\rightarrow 0$. In particular, we prove this when there are
        $h\in G$ with $q_h\gg q_{\min}$ by explicitly constructing mechanisms on
        which the frequent vacancy types cooperate to facilitate the East
        movement of the least frequent vacancies.
\end{abstract}
\maketitle

\section{Introduction}
In~\cite{garrahan2003coarse}, Chandler and Garrahan introduce a coarse-grained
model inspired by the complex dynamics of glass-forming liquids. It is best
described as a generalization of the East process, so let us first recall the
dynamics of the East process on $\Z^d$ (see e.g.\
\cites{cancrini2008kcm,faggionato2012east,mareche2019exponential}). The East
process is an interacting particle system with state space ${\{0,1\}}^{\Z^d}$
and single parameter $q\in (0,1)$.
Each vertex $x\in \Z^d$, with rate one and independently across $\Z^d$, is
resampled from $\{0,1\}$ according to the $\mathrm{Ber}(p)$-measure, $p=1-q$,
iff in the current configuration there is at least one vacancies (i.e.\ a state
``$0$'') among the neighbours $y$ of $x$ of the form $y=x-\mathbf{e}$ where
$\mathbf{e}$ are canonical base vectors of $\Z^d$ the set of which we denote by
$\mathcal{B}$. We say that the vacancies have north and east as their
propagation directions.

The model by Chandler and Garrahan is an interacting particle system on $\Z^d$
with state space ${(G\cup \{\star\})}^{\Z^d}$, $|G|=2^d$, and two parameters
$q,\xi\in (0,1)$, that is informally described as follows. To each element of
$G$ we associate a unique set of propagation directions corresponding to one of
the $2^d$ possible rotations of the East propagation directions. On each vertex
there are two Poisson clocks: one that gives diffusive rings with rate $\xi$
and one that gives directed rings with rate $1-\xi$. Both of these rings come
with their own facilitation mechanisms. A diffusive ring on $x\in \Z^d$ is
legal if there is a neighbour $y$ of $x$ in the state $h$ such that $x-y$
corresponds to one of the associated propagation directions of $h$, we say that
$x$ is $h$-facilitated. On a legal diffusive ring, if the state of $x$ is in
$G$, there is a transition to $\star$ with rate $p:=1-q$ and if the state of
$x$ is $\star$ there is a transition to any of the states in $G$ with rate
$q/2^d$ respectively. Legality for directed rings is a bit more tricky. At a
directed ring on $x$, if $x$ is in the neutral state $\star$, it can transition
to $h\in G'$ with rate $q/2^h$, where $G'$ is the set of all vacancy types $h$
such that $x$ is $h$-facilitated. If $x$ is in the state $h$ it can transition
to the neutral state $\star$ with rate $p$ iff it is $h$-facilitated. Thus,
while in the diffusive rings it suffices to be facilitated at all to be able to
transition from a state in $G$ to $\star$ and vice versa, directed rings
require the transitions and the facilitation to be by the same vacancy type, so
that the various vacancy types block each other.

In this paper we consider the limit case $\xi=0$ where there are
no diffusive rings. This model behaves like multiple rotated versions of the
East model evolving at the same time with a shared ``$1$'' state, represented
by the \emph{neutral state} $\star$. The original model is not ergodic (see
\cref{thm:ergodicity}(A)) so we consider the model with only a subset of all
possible $2^{d}$ rotations and allow for varying transition rates $q_h$ for the
various states $h\in G$. We dub this model the \emph{multicolour East model}
(MCEM). The MCEM is reversible with respect to the product measure $\mu$
that locally assigns a state with its corresponding probability ${\{q_h\}}_{h\in
G}$ or $p:=1-\sum_{h\in G} q_h$ for the neutral state, so that the transition
rates correspond to the equilibrium densities of the respective states.

In \cref{thm:ergodicity} we give sufficient conditions on the geometry of the
vacancy types so that the MCEM on $\Z^d$ has positive spectral gap, which in
particular implies that the two-dimensional MCEM with $|G|\le 3$ has positive
spectral gap and with $|G|=4$ is not ergodic, thus fully classifying the
ergodicity landscape in two dimensions. In \cref{thm:abc_relaxation} we then
give sufficient conditions on the equilibrium densities and the geometry for the
spectral gap of the two-dimensional MCEM to be given by the two-dimensional East
model spectral gap in the leading order in the limit $q_{\min}:=\min_{h\in
G}q_h\rightarrow 0$. In particular, we prove this for cases where there are one
or two vacancy types with much larger equilibrium densities than $q_{\min}$.
This result might be surprising at first, as one could expect the East model
dynamics of the least frequent vacancies to be blocked by the more frequent
ones, leading to a spectral gap that is given by these blocking dynamics. In
fact, we prove that the frequent vacancies cooperate in a way to facilitate the
two-dimensional East movement of the least frequent vacancies, so that the
blocking is negligible.

As far as the author is aware this is the first time this model is treated in
mathematical literature, but the physical motivation and dynamics bear close
resemblance to those of kinetically constrained models (KCM)
\cites{cancrini2008kcm,garrahan2011kinetically}. In fact, this paper lines up
well with current research on KCM on $\Z^d$ which looks at ergodicity and in
particular ergodicity breaking
transitions~\cites{cancrini2008kcm,kordzakhia2006ergodicity,shapira2020kinetically}
and at the spectral gap
behaviour~\cites{martinelli2019towards,hartarsky2021universality}.
\subsection{Construction}%
\label{sec:construction}
We start by constructing the state space and together with the associated
propagation directions.

\begin{definition}[Vacancy types and their constraints]
         The set of vacancy types is a finite set $\mathcal{V}$ of cardinality
         $2^d$. We identify $\mathcal{V}$ with the hypercube
         $H_d:={\{0,1\}}^d\subset \Z^d$ and refer to the vacancy type
         corresponding to the vertex $h\in H_d$ as \emph{the vacancy of type
         $h$ or the $h$-vacancy}. We say that $\mathbf{v}$ is a
         \emph{propagation direction} for the $h$-vacancy, and write
         $\mathbf{v}\in \mathcal{P}(h)$, if $\|\mathbf{v}\|=1$ and $h+\mathbf{v}\in H_d$. For $h\in
         H_d$ we say that $x\prec^{(h)} y$ if $x\cdot \mathbf{v} \le y \cdot
         \mathbf{v}$ for every $\mathbf{v} \in \mathcal{P}(h)$.
        
         Given $G\subset H_d$, which we identify with a collection of vacancy
         types in $\mathcal V$, we define the vertex state space
         $\mathcal{S}(G)$ as the union of $G$ together with the \emph{neutral
         state} denoted by $\star$. For $\omega\in {\mathcal{S}(G)}^{\Z^d}$,
         $h\in G$ and $x\in \Z^d$ the constraint $c_x^{h}(\omega)$ is given by
         \begin{equation}
                 {c}_x^{h}(\omega)
                 = \begin{cases}
                         1 &
                         \text{if}\ \exists
                         \, \mathbf{v}\in \mathcal{P}(h):\ \omega_{x-\mathbf{v}}\text{ is a $h$-vacancy},\\
                         0 & \text{otherwise.}
                 \end{cases}
         \end{equation}
    See \cref{fig:states} for an illustration of $H_2$ and $H_3$ with
    the associated propagation directions for each vacancy type.
\end{definition} 
\begin{remark}
        If $G=\{(0,0,\ldots,0)\}$, we can identify
        $\star$ with $1$ and $(0,0,\ldots, 0)$ with $0$ to recover the state
        space of the $d$-dimensional East model with the corresponding
        constraints on $\Z^d$.
\end{remark}
\noindent\textbf{Notation warning:} In the sequel, for a given $\omega\in
{\mathcal{S}(G)}^{\Z^d}$ and $h\in G$, we will often write $\omega_x=h$ meaning that $\omega_x$ is a vacancy of type $h$.

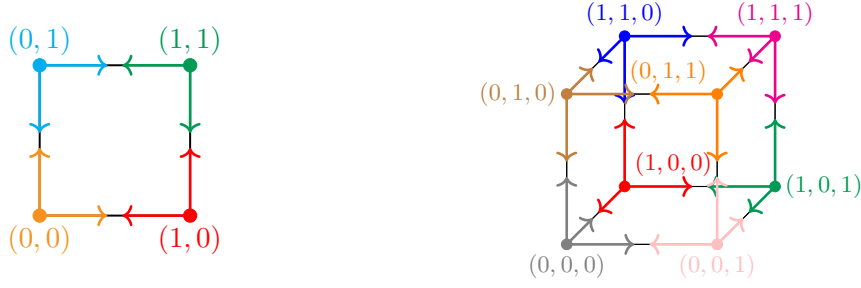
\begin{figure}
        \begin{center}
            \begin{subfigure}[c]{0.425\textwidth}
                \begin{tikzpicture}[]
                        \draw[line width=0.75pt] (0,0) -- ++(0,-2) -- ++(2,0)
                                -- ++(0,2) -- cycle;

                        \node[anchor=north, BurntOrange] at (0,-2) {$(0,0)$};
                        \draw[line width=1pt, ->, BurntOrange] (0,-2) -- ++(0,0.9);
                        \draw[line width=1pt, ->, BurntOrange] (0,-2) -- ++(0.9,0);
                        \filldraw[BurntOrange](0,-2)circle[radius=2.5pt] {};

                        \node[anchor=south, Cyan] at (0,0) {$(0,1)$};
                        \draw[line width=1pt, ->, Cyan] (0,0) -- ++(0,-0.9);
                        \draw[line width=1pt, ->, Cyan] (0,0) -- ++(0.9,0);
                        \filldraw[Cyan](0,0)circle[radius=2.5pt] {};

                        \node[anchor=south, ForestGreen] at (2,0) {$(1,1)$};
                        \draw[line width=1pt, ->, ForestGreen] (2,0) -- ++(0,-0.9);
                        \draw[line width=1pt, ->, ForestGreen] (2,0) -- ++(-0.9,0);
                        \filldraw[ForestGreen](2,0)circle[radius=2.5pt] {};

                        \node[anchor=north, red] at (2,-2) {$(1,0)$};
                        \draw[line width=1pt, ->, red] (2,-2) -- ++(0,0.9);
                        \draw[line width=1pt, ->, red] (2,-2) -- ++(-0.9,0);
                        \filldraw[red](2,-2)circle[radius=2.5pt] {};
                \end{tikzpicture}
            \end{subfigure}    
            \begin{subfigure}[c]{0.425\textwidth}
                        \begin{tikzpicture}[]
                                \newcommand{\Side}{2}
                                \coordinate (O) at (0,0,0);
                                \coordinate (A) at (0,\Side,0);
                                \coordinate (B) at (0,\Side,\Side);
                                \coordinate (C) at (0,0,\Side);
                                \coordinate (D) at (\Side,0,0);
                                \coordinate (E) at (\Side,\Side,0);
                                \coordinate (F) at (\Side,\Side,\Side);
                                \coordinate (G) at (\Side,0,\Side);

                                \foreach \x in {0,\Side} {
                                        \foreach \y in {0,\Side} {
                                                \foreach \z in {0,\Side} {
                                                        \filldraw[black!50] (\x,\y,\z)
                                                                circle (1pt);
                                                }
                                        }
                                }

                                \draw[black] (O) -- (C) -- (G) -- (D) -- cycle;
                                \draw[black] (O) -- (A) -- (E) -- (D) -- cycle;
                                \draw[black] (O) -- (A) -- (B) -- (C) -- cycle;
                                \draw[black] (D) -- (E) -- (F) -- (G) -- cycle;
                                \draw[black] (C) -- (B) -- (F) -- (G) -- cycle;
                                \draw[black] (A) -- (B) -- (F) -- (E) -- cycle;

                                \node[anchor=south, blue] at (A) {\footnotesize{$(1,1,0)$}};
                                \filldraw[blue] (A) circle (2pt);
                                \draw [black, ->,line width=1pt, blue] (A) -- (0.9,\Side,0) ;
                                \draw [black, ->,line width=1pt, blue] (A) -- (0,-0.9+\Side,0) ;
                                \draw [black, ->,line width=1pt, blue] (A) -- (0,\Side,+0.9) ;

                                \node[anchor=east, brown] at (B) {\footnotesize{$(0,1,0)$}};
                                \filldraw[brown] (B) circle (2pt);
                                \draw [black, ->,line width=1pt, brown] (B) -- (0.9,\Side,\Side) ;
                                \draw [black, ->,line width=1pt, brown] (B) -- (0,-0.9+\Side,\Side) ;
                                \draw [black, ->,line width=1pt, brown] (B) -- (0,\Side,\Side-0.9) ;

                                \node[anchor=north, gray] at (C) {\footnotesize{$(0,0,0)$}};
                                \filldraw[gray] (C) circle (2pt);
                                \draw [black, ->,line width=1pt, gray] (C) -- (0.9,0,\Side) ;
                                \draw [black, ->,line width=1pt, gray] (C) -- (0,0.9,\Side) ;
                                \draw [black, ->,line width=1pt, gray] (C) -- (0,0,\Side-0.9) ;

                                \node[anchor=west, ForestGreen] at (D) {\footnotesize{$(1,0,1)$}};
                                \filldraw[ForestGreen] (D) circle (2pt);
                                \draw [black, ->,line width=1pt,ForestGreen] (D) -- (\Side-0.9,0,0) ;
                                \draw [black, ->,line width=1pt,ForestGreen] (D) -- (\Side,+0.9,0) ;
                                \draw [black, ->,line width=1pt,ForestGreen] (D) -- (\Side,0,+0.9) ;

                                \node[anchor=south, magenta] at (E) {\footnotesize{$(1,1,1)$}};
                                \filldraw[magenta] (E) circle (2pt);
                                \draw [black, ->,line width=1pt, magenta] (E) -- (\Side-0.9,\Side,0) ;
                                \draw [black, ->,line width=1pt, magenta] (E) -- (\Side,\Side-0.9,0) ;
                                \draw [black, ->,line width=1pt, magenta] (E) -- (\Side,\Side,0.9) ;

                                \node[anchor=south east, orange] at (F) {\footnotesize{$(0,1,1)$}};
                                \filldraw[orange] (F) circle (2pt);
                                \draw [black, ->,line width=1pt,orange] (F) -- (\Side-0.9,\Side,\Side) ;
                                \draw [black, ->,line width=1pt,orange] (F) -- (\Side,\Side-0.9,\Side) ;
                                \draw [black, ->,line width=1pt,orange] (F) -- (\Side,\Side,\Side-0.9) ;

                                \node[anchor=north, pink] at (G) {\footnotesize{$(0,0,1)$}};
                                \filldraw[pink] (G) circle (2pt);
                                \draw [black, ->,line width=1pt, pink] (G) -- (\Side-0.9,0,\Side) ;
                                \draw [black, ->,line width=1pt, pink] (G) -- (\Side,+0.9,\Side) ;
                                \draw [black, ->,line width=1pt, pink] (G) -- (\Side,0,\Side-0.9) ;

                                \node[anchor=south west, red] at (O) {\footnotesize{$(1,0,0)$}};
                                \filldraw[red] (O) circle (2pt);
                                \draw [black, ->,line width=1pt, red] (O) -- (0.9,0,0) ;
                                \draw [black, ->,line width=1pt, red] (O) -- (0,0.9,0) ;
                                \draw [black, ->,line width=1pt, red] (O) -- (0,0,0.9) ;
                        \end{tikzpicture}
            \end{subfigure}    
            \caption{$H_d$ for $d=2$ (left) and $d=3$ (right) together with the
            vacancy types as coloured corners and the propagation directions
            (arrows) of
            the corners, the length of the propagation directions is less than half the
            actual length for rendering reasons.}\label{fig:states} 
       \end{center} 
\end{figure}

For $G\subset H_d$ we call vectors $\mathbf{q}=\{q_h\colon h\in G\}$ with
$q_h>0$ for $h\in G$, and $\sum_{h\in G} q_h < 1$, \emph{valid parameter
sets} and write $p=1-\sum_{h\in G} q_h$. Given a valid parameter set
$\mathbf{q}$ let $\nu$ denote the probability measure on $\mathcal{S}(G)$ that
assigns probability $p$ to the state $\star$ and $q_h$ to $h$ for all $h\in G$.
For any $\Lambda\subset \Z^d$ define the state space
$\Omega_{\Lambda}={\mathcal{S}(G)}^{\Lambda}$ and the measure
$\mu_{\Lambda}:=\otimes_{x\in \Lambda}\nu$, where we recall the notational
convention that we leave away $\Lambda$ if $\Lambda=\Z^d$. We also omit
the dependence on $\mathbf{q}$ and $G$ in the notation of $p$, $\nu$
and $\Omega_{\Lambda}$ since they will be clear from context.

For subsets $V\subset \Lambda\subset \Z^d$ and configurations $\omega\in
\Omega_{V}$, $\omega'\in \Omega_{\Lambda\setminus V}$
in $\Omega_{\Lambda}$ we write $\omega\cdot \omega'\in \Omega_{\Lambda}$ for
the state given by $\omega$ on $V$ and $\omega'$ on $\Lambda\setminus V$. We
say a function $f:\Omega\rightarrow \R$ is \emph{local} if the value
$f(\omega)$ only depends on the state of finitely many vertices.
\begin{definition}[The $G$-MCEM process]
        Given a subset $G\subset H_d$ and a valid parameter set $\mathbf{q}$ we
        define the continuous time $G$-MCEM process on $\Z^d$ via the infinitesimal
        generator, which we define through its action on local functions\footnote{See
                \cites{kordzakhia2006ergodicity,liggett2010continuous} on how to
                construct a continuous time Markov process starting from the
                action of the generator on local functions}
        $f:\Omega\rightarrow \R$, as
        \begin{equation}
                \mathcal{L}f(\omega)
                =\sum_{h\in G}\sum_{x\in \Z^d}
                {c}_x^{h}(\omega)
                [\mathds{1}_{\omega_x=\star}q_h+\mathds{1}_{\omega_x=h} p]
                \nabla^{(h)}_x f(\omega),
        \end{equation}
        where
        \begin{equation}
                \nabla^{(h)}_x f(\omega):=
                \begin{cases}
                        f(h\cdot \omega_{\Z^d\setminus \{x\}})-f(\omega)
                        &\colon \text{if $\omega_x=\star$,}\\
                        f(\star\cdot \omega_{\Z^d\setminus \{x\}})-f(\omega)
                        &\colon \text{if $\omega_x=h$,}\\
                        0
                        &\colon \text{else.}
                \end{cases}
        \end{equation}
        We write $\omega(t)$ for the state at time $t$ and $\mathds{E}_{\eta}$
        and $\mathds{P}_{\eta}$ for the corresponding expectation and law for
        the process started at $\eta\in \Omega$.
\end{definition}
\begin{remark}
        It might be surprising that the sum of the rates $q_h+p$ is strictly
        smaller than $1$. In fact, here the missing rate $1-q_h-p$
        is hidden in $\nabla_x^{(h)} f(\omega)=0$ if $\omega_x\not\in
        \{h,\star\}$, thus we could have added a term
        $\mathds{1}_{\omega_x\not\in \{h, \star\}} (1-q_h-p)$ for the
        transition in which nothing happens.
\end{remark}
\begin{remark}
        Notice that in the $G$-MCEM process a state can transition from $\star$
        to $h$ iff there is a vector $\mathbf{v}\in \mathcal{P}(h)$ such that
        $x-\mathbf{v}$ has an $h$-vacancy justifying the name \emph{propagation
        direction} for $\mathbf{v}$. In particular, an $h$-vacancy at $x$ can
        only influence those vertices $y$ such that $x\prec^{(h)} y$. Further,
        there is no transition from one vacancy type to another. The process,
        in order to change the state of a vertex from one vacancy type to
        another, first has to go through the neutral state $\star$ (justifying
        its name). In particular, when $|G|\ge 2$ an $h$-vacancy can be blocked
        by a cluster of nearby vacancies of type in $G\setminus \{h\}$. This
        blocking interaction is the main hurdle in bounding the spectral gap.
\end{remark}
The generated process is reversible with respect to $\mu$, indeed for
$x\in\Z^d$ and $\omega'\in \Omega_{\Z^d\setminus\{x\}}$
\begin{align}
        \sum_{\omega\in \Omega_x}
        &\mu_x(\omega) f(\omega\cdot \omega') [\mathds{1}_{\omega=\star}
        q_h+\mathds{1}_{\omega=h} p]
        \nabla_x^{(h)}g(\omega\cdot\omega'))\\
        &= p q_h (f(\star\cdot \omega')-f(h\cdot \omega'))(g(h\cdot
        \omega')-g(\star\cdot \omega')),
\end{align}
and thus, since $c_x^h$ does not depend on the state of $x$ we have
\begin{align}
        \mu(f\mathcal{L} g) = \mu(g\mathcal{L} f),
\end{align}
from which reversibility follows since $f$,$g$ were arbitrary. The associated
Dirichlet form is then
\begin{equation}
        \mathcal{D}(f) := \mu(-f\mathcal{L}f)
        = \sum_{h\in G}\sum_{x\in \Z^d}
        p q_h\mu\bigg[c_x^{h}\mathds{1}_{\omega_x\in \{\star,
                                h\}}
                        {(f(\star\cdot \omega)-f(h\cdot \omega))}^2\bigg]
                        \label{eqn:dirich_mcem}
\end{equation}
and we define the spectral gap as
\begin{equation}\label{eqn:variational_chara_mcem}
        \gamma(G;\mathbf{q})=\gamma(G)
        := \inf_{\substack{f\in \mathrm{Dom}(\mathcal{L})\\
                        f\neq \mathrm{const}}}
        \frac{\mathcal{D}(f)}{\var(f)}\;.
\end{equation}
Using~\cite{liggett1985interacting}*{Section~IV, Theorem~4.13} the MCEM
process is ergodic with stationary measure $\mu$, if $0$ is a simple eigenvalue
of $\mathcal{L}$ and thus in particular if the spectral gap is positive.
\begin{remark}\label{rem:dirichlet_subsets}
        We will sometimes write $\mathcal{L}_{\Lambda}$ and
        $\mathcal{D}_{\Lambda}$ in which the sum $\sum_{x\in \Z^d}$ is replaced
        with a sum over $\sum_{x\in \Lambda}$ and the measure $\mu$ with the
        measure $\mu_{\Lambda}$. These are functions on $\Omega_{\Lambda^c}$
        and thus either appear with a configuration
        $\omega\in\Omega_{\Lambda^c}$ or inside an average w.r.t.\
        $\mu_{\Lambda^c}$.
\end{remark}

\subsubsection{Graphical construction}%
\label{sub:iem_graphical_construction}
An alternative to the construction via the infinitesimal generator is
via a graphical construction. Put a marked Poisson process on each vertex in $
\Z^d$. The $k$-th ring at the vertex $x\in \Z^d$ occurs at time
$t_{x,k}$ and for each ring we have the mark $U_{x,k}\sim \mu$
so that $U_{x,k}\in \mathcal{S}(G)$ and ${\{U_{x,k}\}}_{x,k}$ is an i.i.d.\ family.
Consider a starting state $\omega(0)\in \Omega$ and denote by
$\omega(t)$ the state at time $t\in \R_+$. With $t_{x,k}-$ an infinitesimally
smaller time than $t_{x,k}$, the graphically constructed process
evolves as follows:
\begin{enumerate}[(i)]
        \item At $t_{x,k}$ we say that we have a $U_{x,k}$-legal ring if any of
                the following conditions is satisfied
                \begin{enumerate}[(a)]
                        \item $U_{x,k}=\star$ and there is an $h\in G$ such
                                that $\omega_x(t_{x,k}-)=h$ and
                                $c_{x}^{h}(\omega(t_{x,k}-))=1$, or
                        \item $U_{x,k}\neq \star$, $\omega_x(t_{x,k}-)=
                                \star$ and
                                $c_{x}^{U_{x,k}}(\omega(t_{x,k}-))=1$, or
                        \item $U_{x,k}=\omega_x(t_{x,k}-)$ and
                                $c_x^{U_{x,k}}(\omega(t_{x,k}-))=1$
                                (i.e.\ nothing changes).
                \end{enumerate}
        \item If $t_{x,k}$ is an $U_{x,k}$-legal ring, we set
                $\omega_x(t_{x,k})$ equal to 
                $U_{x,k}$.
\end{enumerate}
Showing that this construction is well defined on $\Z^d$ and
leads to the same process as the one constructed above through the
infinitesimal generator is analogous to the proof presented
in~\cite{kordzakhia2006ergodicity} for the North-East model.

\section{Results}%
\label{sec:results}
The first result gives sufficient conditions for ergodicity of the $G$-MCEM.
Recall for this that any $G\subset H_d$ inherits the graph structure of $\Z^d$.
\begin{maintheorem}\label{thm:ergodicity}%
        Consider all the following $G$-MCEM with an arbitrary valid parameter set
        $\mathbf{q}$.
        \begin{itemize}
                \item[(A)] If $G=H_d$ then the $G$-MCEM process is
                        not ergodic.
                \item[(B)] Suppose
                        $G\subsetneq H_d$ is such that either condition
                        holds:
                        \begin{enumerate}[(B.i)]
                                \item there is a canonical base vector
                                        $\mathbf{e}\in \mathcal{B}$
                                        of $\Z^d$
                                        such that for any two
                                        $h,h'\in G$ we have $h\cdot \mathbf{e}
                                        = h'\cdot \mathbf{e}$.
                                \item there is a superset $G'\subsetneq H_d$
                                        of $G$ such that $G'$ is isomorphic to
                                        a star-graph.
                        \end{enumerate}
                        Then the $G$-MCEM process has a positive spectral gap.
        \end{itemize}
\end{maintheorem}
\begin{example}
        Any $G\subset H_3$ that is a subset of a single face satisfies (B.i)
        and any $G\subset H_2$ with $|G|<4$ satisfies (B.ii). In particular
        note that this gives complete information about ergodicity in $d=2$ but
        leaves gaps for $d\ge 3$.
\end{example}

For $d=2$ we even find sufficient conditions on the geometry of $G$ and the
parameter set so that the limiting behaviour of the spectral gap is given by
the East model. Given a valid parameter set $\mathbf{q}$ we define $q_{\min}=\min_{h\in
G}q_h$, $q_{\max}=\max_{h\in G} q_h$ and if $|G|=3$ we write $\qmed$ for the
$q_h\in G$ that is not in $\{q_{\max},q_{\min}\}$. We further define
$\theta_{h}=\theta_{q_{h}}:=|\log_2(q_h)|$ and write $\gamma_2=\gamma_2(q)$ for
the spectral gap of the two-dimensional East model with vacancy density $q$ and
$\gamma(G;\mathbf{q})$ the spectral gap of the $G$-MCEM with parameter set
$\mathbf{q}$.
\begin{maintheorem}\label{thm:abc_relaxation}
        Fix $\Delta>0$ and 
        consider a $G$-MCEM on $\Z^2$ with $|G|\in \{2,3\}$ and a valid parameter set
        $\mathbf{q}$ such that $p>\Delta$. Then,
        \begin{equation}\label{eqn:abc_relaxation_limit}
                \lim_{q_{\min}\rightarrow
                0}\frac{\gamma(G;\mathbf{q})}{\gamma_2(q_{\min})}
                =1
        \end{equation}
        in the following cases.
        \begin{itemize}
        \item Any $2$-subset $G\subset H_2$ and either one of the following conditions holds:
        \begin{enumerate}[(2.i)]
                \item $\lim_{q_{\min}\rightarrow
                0}q_{\max} \theta_{q_{\min}}^3
                        = 0$,
                \item $\lim_{q_{\min}\rightarrow
                0}q_{\max}\theta_{q_{\min}}^3
                        /\log_2(\theta_{q_{\min}}) =\infty$.
        \end{enumerate}
        \item Any $3$-subset $G\subset H_3$ and either one of the following
                conditions holds:
        \begin{enumerate}[(3.i)]
                \item $\lim_{q_{\min}\rightarrow
                0}q_{\max} \theta_{q_{\min}}^3
                        = 0$,
                \item $\lim_{q_{\min}\rightarrow
                0}q_{\max}\theta_{\qmed}^3
                        /\log_2(\theta_{q_{\min}})=\infty$ and
                        $\lim_{q_{\min}\rightarrow 0}q_{\mathrm{med}}\theta_{q_{\min}}^6=
                        0$,
                \item $G$ is such that the vacancies associated to
                        $q_{\mathrm{med}}$ and $q_{\max}$ share a propagation direction and
                        $\liminf_{q_{\min}\rightarrow
                0}q_{\mathrm{med}}>0$.
        \end{enumerate}
        \end{itemize}
\end{maintheorem}
\begin{remark}
        The cases are ordered from the easiest to the hardest regime. The cases
        $(2.i)$ and $(3.i)$ are the easiest since in these cases even the
        highest density $q_{\max}$ is relatively low so that most vacancies in
        equilibrium are surrounded by large neutral state patches. Thus for
        these cases it is natural to conjecture that the spectral gap of the
        $G$-MCEM should be given by the two-dimensional East model spectral
        gap. This also includes the case by \cite{garrahan2003coarse} in which
        all vacancy type had the same density.

        The next harder case is if there is one vacancy type that is frequent
        in equilibrium, i.e.\ case $(2.ii)$ and $(3.ii)$. The
        conclusion of \cref{thm:abc_relaxation} still presents itself as a
        natural conjecture if we consider that any vacancy of the frequent type
        will see large patches of either neutral vertices or its own vacancy
        type. Thus, any vacancy of the frequent type that blocks the infrequent
        vacancies is likely to be removable by close vacancies of the same type
        allowing the infrequent vacancies to evolve according to their
        respective two-dimensional East model dynamics.

        The hardest case is $(3.iii)$, when two vacancy types are frequent. In this
        case the frequent vacancy types might block each other and we only manage to
        find configurations that remove the blocking frequent vacancies if they share a
        propagation direction.
\end{remark}
\begin{remark}
        It is possible to relax the requirement that $p>\Delta$ at the cost of
        an additional factor of $1/p$ or $1/p^2$ in $\gamma(G;\mathbf{q})$
        which represents the average waiting time for a vertex to get back into
        the neutral state. If $p\rightarrow 0$ then vertices rarely reach the
        neutral state and there can be no transition from one vacancy type to
        another explaining the extra cost in the spectral gap. As we have no
        tight bounds on the $p$ dependency, i.e.\ whether it should be $1/p$ or
        $1/p^2$ or even worse, we limit the discussion to the case $p>\Delta$.
\end{remark}

\section{Key tools}\label{sec:key_tools}
We recall past results together with smaller Lemmas that enter the proofs of
\cref{thm:ergodicity} and \cref{thm:abc_relaxation}. If $G$ and $\mathbf{q}$
are not explicitly stated then they, and correspondingly the state space $\Omega$,
local equilibrium $\nu$ and particle density $p$, are arbitrarily fixed.
\subsection{A constrained Poincar\'{e} inequality for product measures}%
\label{sub:exterior}
Define the support $\supp{A}$ of an event $\mathcal{A}\subset\Omega$ as the set
of vertices the event depends on.
\begin{definition}[Exterior condition]
        Given an increasing and exhausting collection of subsets
        ${\{V_n\}}_{n\in \Z}$ of $\Z^d$ (i.e.\ $V_n\subset V_{n+1}$ for all $n$
        and $\cup_n V_n=\Z^d$), let the exterior of $x\in V_n$ be the set
        $\mathrm{Ext}_x:= \cup_{j=n}^{\infty} V_{j+1}\setminus V_j$. We then
        say that the family of events \emph{${\{\mathcal{A}_x\}}_{x\in \Z^d}$
        satisfies the exterior condition} w.r.t.\ ${\{V_n\}}_{n\in \Z}$ if
        $\supp{\mathcal{A}_x}\subset \mathrm{Ext}_x$ for all $x\in
        \Z^d$.
\end{definition}

Let ${\{\mathcal{A}^{(i)}_x\}}_{x\in \Z^d}, i\in [k]:=\{1,\ldots, k\}$ and
write $\mathrm{Supp}(\mathcal{A}_x^{(I)}) = \bigcup_{i\in I}
\mathrm{Supp}(\mathcal{A}_x^{(i)})$ for nonempty subsets $I\subset
[k]$.
\begin{theorem}[Exterior condition theorem,
        {\cite{martinelli2019towards}*{Theorem~2}}]\label{thm:exterior_thm}
        Assume that
        \begin{equation}\label{eqn:ext_cond_thm_assumption_original}
                (2^k-1)\sup_{z\in \Z^2}
                \sum_{\substack{J\subset [k]\\J\neq \emptyset}}
                \sum_{\substack{x\in \Z^2\\ \{x\}\cup
                \mathrm{Supp}(\mathcal{A}_x^{(J)})\ni z}}
                \mu\left(\prod_{i\in
                J}(1-\mathds{1}_{\mathcal{A}_x^{(j)}})\right)< 1/4.
        \end{equation}
        Suppose in addition that there exists an exhausting and increasing
        family ${\{V_n\}}_{n\in \Z}$ of subsets of $\Z^d$ such that, for any
        $i\in [k]$, the family ${\{\mathcal{A}_x^{(i)}\}}_{x\in \Z^d}$
        satisfies the exterior condition w.r.t.\ ${\{V_n\}}_{n\in \Z}$. Then, for any
        local function $f:\Omega\rightarrow \R$ we have
        \begin{equation}\label{eqn:ext_cond_thm_dirichlet_rhs_original}
                \var(f) \le 4 \sum_{x}\mu\left(\left[\prod_{j=1}^k
                \mathds{1}_{\mathcal{A}_x^{(j)}}\right]\var_x(f)\right)\;.
        \end{equation}
        In particular, the same conclusion holds if, instead of
        \cref{eqn:ext_cond_thm_assumption_original} we have that
        \begin{equation}\label{eqn:ext_cond_indep_assumptions}
                \lim_{q_{\min}\rightarrow 0}\max_{j\in [k]}\left[\sup_{x\in \Z^2}
                        |\mathrm{Supp}(\mathcal{A}_x^{(j)})|
                \sup_{x\in \Z^2}
                \mu\left(1-\mathds{1}_{\mathcal{A}_x^{(j)}}\right)\right]=0.
        \end{equation}
\end{theorem}
\begin{proof}
        \cref{eqn:ext_cond_indep_assumptions} implies
        \cref{eqn:ext_cond_thm_assumption_original} and the proof how
        \cref{eqn:ext_cond_thm_assumption_original} implies
        \cref{eqn:ext_cond_thm_dirichlet_rhs_original} is
        in~\cite{martinelli2019towards}. Note
        that~\cite{martinelli2019towards} made the statement with KCM in
        mind, but the proof only uses that $\mu$ is a product measure so
        applies equally to MCEM.\@
\end{proof}

\subsection{Monotonicity in \texorpdfstring{$G$}{G} of the spectral gap}%
\label{sub:projection}
Naturally one conjectures that the more vacancy types are added to the $G$-MCEM
the lower the spectral gap should be as the model gets progressively more
jammed through the interaction of the various vacancy types. Indeed, the next
result shows this is the case.
\begin{lemma}\label{lemma:monotonicity_in_g}
        For any $G'\subset G\subset H_d$ and valid parameter set $\mathbf{q}$
        for the $G$-MCEM we have
        \begin{equation}
                \gamma(G, \mathbf{q})
                \le \gamma(G',\mathbf{q}')
        \end{equation}
        with $\mathbf{q}'=\{q_h\colon h\in G'\}$ and in particular
        \begin{equation}
                \gamma(G, \mathbf{q})
                \le \gamma_d(q_{\min}).
        \end{equation}
\end{lemma}
\begin{proof}
        Let $G'\subset G\subset H_d$ and fix a parameter set $\mathbf{q}$ for
        the $G$-MCEM. Recall that $\mathcal{S}(G)=G\cup \{\star\}$.  Define the
        projection $\varphi$ on $\mathcal{S}(G)$ to $\mathcal{S}(G')$ that maps
        $G'$ onto itself and $\mathcal{S}(G)\setminus G'$ to $\star$. We then
        have, through the variational characterisation of the spectral gap
        \cref{eqn:variational_chara_mcem}, that
        \begin{equation}
                \gamma(G,\mathbf{q})
                = \inf_{\substack{
                f\in \mathrm{Dom}(\mathcal{L}^{(G,\mathbf{q})})\\f\neq \mathrm{const}}}
                \frac{\mathcal{D}(f)}{\var(f)}
                \le
                \inf_{\substack{ g\in \mathrm{Dom}(\mathcal{L}^{(G',\mathbf{q})})
                \\g\neq \mathrm{const}}}
                \frac{\mathcal{D}(g\circ \varphi)}
                {\var(g\circ \varphi)}\;,
        \end{equation}
        where (exceptionally) we write $\mathcal{L}^{(G,\mathbf{q})}$ for the
        generator of the $G$-MCEM to make the $G$-dependence explicit in this
        proof. Write $\nu'$ for the measure on $\mathcal{S}(G')$ that assigns
        probability $q_h$ to $h\in G'$ and $p':=1-\sum_{h\in G'}q_h$ to $\star$
        and let $\mu'$ be the product measure of $\nu'$. Since $\mu(g\circ
        \varphi)=\mu'(g)$ and ${(g\circ\varphi)(\cdot)}^2=g^2\circ
        \varphi(\cdot)$ we get $\var(g\circ\varphi) = \var_{\mu'}(g)$. For the
        Dirichlet form we get (recall \cref{eqn:dirich_mcem})
        \begin{align}
                \mathcal{D}(g\circ\varphi)
                &=\sum_{h\in G}\sum_{x\in \Z^2}
                \mu\mleft[c_x^h q_h p {(\nabla_x^{(h)}(g\circ\varphi))}^2\mright]\\
                &\le\sum_{h\in G'}\sum_{x\in \Z^2}
                \mu'\mleft[c_x^h q_h p' {(\nabla_x^{(h)}(g))}^2\mright]
        \end{align}
        where we used that the constraints $c_x^{h}$ only check whether a
        qualified neighbour is $h$ or not, and thus is the same for the
        $G$-MCEM and the $G'$-MCEM if $h\in G'$. Further $\nabla_x^{(h)}(g\circ
        \varphi)(\omega)=0$ if $\omega_x\not\in G'$ and so we could replace
        $\mu$ with $\mu'$. The r.h.s.\ is equal to the Dirichlet form of the
        $G'$-MCEM with parameter set $\mathbf{q}$ so we get the first part of
        the claim.

        The second part follows analogously by mapping the $h$ with the lowest
        equilibrium density to $0$ and all the other states to $1$ thus
        recovering the spectral gap $\gamma_d(q_{\min})$ of the East model with
        vacancy density $q_{\min}$.
\end{proof}

\subsection{Variance as transition terms and the path method}
Given the valid parameter set $\mathbf{q}$, recall the measure $\nu$ on
$\mathcal{S}(G)=\{\star\}\cup G$
that assigns probability $q_h$ to $h\in G$ and $p=1-\sum_{h\in G}q_h$ to $\star$.
\begin{lemma}[Variance as transition terms]\label{lemma:var_as_trans}
        For any function $f:\mathcal{S}(G)\rightarrow \R$ we find
        \begin{equation}\label{eqn:var_as_trans}
                \var_{\nu}(f)(\omega) \le 2\sum_{h\in G}q_h{(\nabla^{(h)}(f))}^2(\omega).
        \end{equation}
\end{lemma}
\begin{proof}
        Writing $p=q_{\star}$ in this proof we have
        \begin{align}
                \frac12 \sum_{\omega,\omega'\in S(G)} q_{\omega}
                q_{\omega'} {(f(\omega)-f(\omega'))}^2
                &=\sum_{\omega\in S(G)} q_\omega {f(\omega)}^2
                - {\left(\sum_{\omega\in S(G)} q_{\omega}
                                f(\omega)\right)}^2\\
                &= \var_{\nu}(f).
        \end{align}
        We recover the left hand side of the claim by applying
        Cauchy-Schwarz, giving
        \begin{equation}
                {(f(\omega)-f(\omega'))}^2
                \le 2 \left({(f(\omega)-f(\star))}^2+{(f(\star)-f(\omega'))}^2
                \right)
        \end{equation}
        and thus the claim.
\end{proof}
We say that a family of configurations ${\{(\omega^{(i)})\}}_{i\in [n]}$ is a
\emph{legal path} if each transition from $\omega^{(i)}$ to $\omega^{(i+1)}$ is
legal for the $G$-MCEM. Recall also that $x \prec^{(h)} y$ for $h\in H_d$ if
$x\cdot \mathbf{v}\le y\cdot \mathbf{v}$ for any $\mathbf{v}\in
\mathcal{P}(h)$.

Our second tool, the path method, is a well known trick in estimating the
spectral gap see for example~\cite{cancrini2008kcm}*{Proposition~6.6}
or~\cite{gine2006lectures} for uses in other contexts. Recall for this the
notation of $\mathcal{D}_{\Lambda}$ introduced in \cref{rem:dirichlet_subsets}
where the sum over all vertices in $\Z^d$ is replaced by the sum over
$\Lambda\subset \Z^d$ and the equilibrium measure $\mu$ by $\mu_{\Lambda}$.
\begin{lemma}[The path method]\label{lemma:path_method}
        Let $\omega,\eta\in \Omega$ and let $\Gamma=(\omega^{(1)}, \ldots,
        \omega^{(n)})$ be a legal path such that $\omega^{(1)}=\omega$ and
        $\omega^{(n)}=\eta$ and let $\Lambda\subset \Z^d$ consist of those
        vertices $x$ such that $\omega^{(i)}_x\neq \omega_x^{(i+1)}$ for some
        $i\in [n]$. Then, for any $f:\Omega\rightarrow \R$
        \begin{equation}
                \mu_{\Lambda}(\omega){(f(\omega)-f(\eta))}^2 \le
                \frac{n}{\min(\mathbf{q}, p)} \max_{i\in
                [n]}\frac{\mu_{\Lambda}(\omega)}{\mu_{\Lambda}(\omega^{(i)})}
                \mathcal{D}_{\Lambda}(f)(\omega).
        \end{equation}
\end{lemma}
\begin{proof}
        Write $f(\omega)-f(\eta)=\sum_{i\in [n-1]}
        f(\omega^{(i)})-f(\omega^{(i+1)})$ as a telescopic sum and use
        Cauchy-Schwarz to get
        \begin{align}
                \mu_{\Lambda}(\omega){(f(\omega)-f(\eta))}^2
                &\le n \max_{i\in
                [n]}\frac{\mu_{\Lambda}(\omega)}{\mu_{\Lambda}(\omega^{(i)})}\sum_{i\in [n-1]}
                \mu_{\Lambda}(\omega^{(i)}){(f(\omega^{(i)})-f(\omega^{(i+1)}))}^2\\
                &\le 
                \frac{n}{\min(\mathbf{q}, p)} \max_{i\in
                [n]}\frac{\mu_{\Lambda}(\omega)}{\mu_{\Lambda}(\omega^{(i)})}
                \mathcal{D}_{\Lambda}(f)(\omega),
        \end{align}
        where in the last inequality we used that for $\omega^{(i)}\rightarrow
        \omega^{(i+1)}$ to be a legal transition there is exactly one $x$ such
        that $\omega^{(i)}_x\neq \omega^{(i+1)}_x$.
\end{proof}
In the proofs of part $(B)$ and $(C)$ of \cref{thm:ergodicity} we do not
explicitly mention the length of the involved paths as the important thing is
that they are finite not how they scale. In \cref{thm:abc_relaxation} instead
it is of crucial importance to know the order of magnitude.

\section{Proof of \texorpdfstring{\cref{thm:ergodicity}}{Theorem 1}}
Part (A) of \cref{thm:ergodicity} can be proven without any of the tools just
introduced, while the two subparts of part (B) require some more involved
construction.
\subsection{Proof of \texorpdfstring{\cref{thm:ergodicity}}{Theorem 1}(A)}
If $G=H_d$ say that $\omega\in \Omega_{H_d}$ is a blocked state if
$\omega_{\mathbf{1}-h}=h$ for each $h\in H_d$. By construction, there is no
legal transition from a blocked state to a non-blocked state since to transition the
$h$-vacancy at $\mathbf{1}-h$ to $\star$ you need another $h$-vacancy inside
$H_d$ but every vertex in $H_d$ is already occupied by a different
vacancy type. Let $\mathcal{A}$ be the event that $\omega\restriction_{H_d}$ is
a blocked state. Then $\mathds{1}_{\mathcal{A}}$ is not a constant
function but $\mathcal{D}(\mathds{1}_{\mathcal{A}})=0$ while
$\mu(\mathcal{A})>0$ so that $\var_{\mu}(\mathds{1}_{\mathcal{A}})>0$.
Thus, $0$ is not a simple eigenvalue of the generator and we get the
claim by~\cite{liggett1985interacting}*{Section~IV, Theorem~4.13}.
\qed

\subsection{Vacancies with a common direction: Proof of
        \texorpdfstring{\cref{thm:ergodicity}}{Theorem 1}(B.i)}%
\label{sec:common_dir}
By \cref{lemma:monotonicity_in_g}, we assume w.l.o.g.\ that
$G=\{h_{\mathbf{j}}\colon \mathbf{j}\in {\{0,1\}}^{d-1}\simeq H_{d-1}\}$ where
$h_{\mathbf{j}}=(\mathbf{j}_1,\ldots, \mathbf{j}_{d-1}, 0)\in H_{d-1}\otimes
\{0\}\subset H_d$. For any
$\mathbf{j}\in H_{d-1}$ we have $\mathbf{e}_d\in \mathcal{P}(h_{\mathbf{j}})$ and
for $i\in [d-1]$ we have ${(-1)}^{\mathbf{j}_i}\mathbf{e}_i\in
\mathcal{P}(h_{\mathbf{j}})$.

We start by identifying a configuration on $H_d$ that allows us to remove any
vacancies in $H_d+k\mathbf{e}_d$ for $k\ge 2$ and for which we can apply the
exterior condition theorem, \cref{thm:exterior_thm}. Then we use the path
method to conclude.

Recall from the construction of the MCEM that each vacancy type $h\in G$ is
associated to a corner $x\in H_d$ of the hypercube. We say that a configuration
$\omega\in \Omega$ is $H_d$-good if $\omega_{x}$ for $x\in H_d$ is either in
the state of its associated vacancy type or in the neutral state if there is no
associated vacancy type, i.e.\ if $\omega_h =h$ for every $h\in G$ and
$\omega_x=\star$ for $x\in H_d\setminus \{G\}$ (see
the top left of \cref{fig:ergodicity_proof_bi} for the $d=3$ case). By the above assumption on
$G$ this means that if $\omega$ is $H_d$-good, then any vertex $v\in H_d$ with
$v\cdot \mathbf{e}_d=1$ is in the neutral state, i.e.\ $\omega_{H_{d-1}\otimes
\{1\}}\equiv \star$.

Given an $H_d$-good $\omega$ and any vacancy type $h\in G$ there is a legal path
starting from $\omega$ and ending in a state $\eta$ with $\eta_{x}= h$ for
$x\in H_{d-1}\otimes \{1\}$ and $\eta_x=\omega_x$ otherwise. Indeed, assume
$h=(0,0,\ldots, 0)$, then we can put $h$ on $\mathbf{e}_d=h+\mathbf{e}_d$ since
$\mathbf{e}_d\in \mathcal{P}(h)$. Subsequently, we can put $h$ on any
$\mathbf{e}_d+\mathbf{e}_i$ for $i\neq d$ since $\mathcal{P}(h)$ consists of
all positive propagation directions. Iterate this procedure adding another
$\mathbf{e}_j$ with $j\neq i,d$ and so on until all of
$H_{d-1}\otimes\{1\}$ is in state $h$. By construction this is
possible for any $h\in G$.

Then, there is a legal path starting from $\omega$ ending in a state $\eta$
such that $\eta_{H_{d-1}\otimes\{2\}}\equiv \star$. Indeed, this is a
consequence of $\mathbf{e}_d$ being a propagation direction of any vacancy type
$h$ and the fact that we can bring $h$ to any vertex in $H_{d-1}\otimes\{1\}$
as discussed in the previous paragraph. By reversibility, this implies that we
can construct a legal path that puts $H_{d-1}\otimes\{2\}$ into any state.

For any $k\in \N, k\ge 2$ we can iterate this argument to find a legal path
from $\omega$ to $\sigma$ where $\sigma_x=\star$ if $x\in \cup_{j\in
[2,k]} (H_{d-1}\otimes \{j\})$ and $\sigma_x=\omega_x$ otherwise.
By reversibility we can thus find a legal path to \emph{any}
$\sigma$ that agrees with $\omega$ outside of $\cup_{j\in
[2,k]} (H_{d-1}\otimes \{j\})$.
\begin{figure}[ht]
        \centering
        \begin{tikzpicture}[scale=0.8]
                \newcommand{\Side}{2}
                \coordinate (O) at (0,0,0);
                \coordinate (A) at (0,\Side,0);
                \coordinate (B) at (0,\Side,\Side);
                \coordinate (C) at (0,0,\Side);
                \coordinate (D) at (\Side,0,0);
                \coordinate (E) at (\Side,\Side,0);
                \coordinate (F) at (\Side,\Side,\Side);
                \coordinate (G) at (\Side,0,\Side);

                \draw[black!50] (2*\Side,0,0) -- (2*\Side,\Side,0) --
                        (2*\Side,\Side,\Side) -- (2*\Side,0,\Side) -- cycle;
                \foreach \y in {0,\Side} {
                        \foreach \z in {0,\Side} {
                                \draw[black!50] (\Side,\y,\z) -- (2*\Side,\y,\z);
                                \filldraw[black] (\Side,\y,\z) circle (2pt);
                                \filldraw[black] (2*\Side,\y,\z) circle (2pt);
                                \node[black,anchor=west] at (\Side,\y,\z) {\large{$\star$}};
                        }
                }

                \draw[black] (O) -- (C) -- (G) -- (D) -- cycle;
                \draw[black] (O) -- (A) -- (E) -- (D) -- cycle;
                \draw[black] (O) -- (A) -- (B) -- (C) -- cycle;
                \draw[black] (D) -- (E) -- (F) -- (G) -- cycle;
                \draw[black] (C) -- (B) -- (F) -- (G) -- cycle;
                \draw[black] (A) -- (B) -- (F) -- (E) -- cycle;

                \node[anchor=south, blue] at (A) {\footnotesize{$(1,1,0)$}};
                \filldraw[blue] (A) circle (2pt);
                \draw [black, ->,line width=1pt, blue] (A) -- (0.9,\Side,0) ;
                \draw [black, ->,line width=1pt, blue] (A) -- (0,-0.9+\Side,0) ;
                \draw [black, ->,line width=1pt, blue] (A) -- (0,\Side,+0.9) ;

                \node[anchor=east, brown] at (B) {\footnotesize{$(0,1,0)$}};
                \filldraw[brown] (B) circle (2pt);
                \draw [black, ->,line width=1pt, brown] (B) -- (0.9,\Side,\Side) ;
                \draw [black, ->,line width=1pt, brown] (B) -- (0,-0.9+\Side,\Side) ;
                \draw [black, ->,line width=1pt, brown] (B) -- (0,\Side,\Side-0.9) ;

                \node[anchor=north, gray] at (C) {\footnotesize{$(0,0,0)$}};
                \filldraw[gray] (C) circle (2pt);
                \draw [black, ->,line width=1pt, gray] (C) -- (0.9,0,\Side) ;
                \draw [black, ->,line width=1pt, gray] (C) -- (0,0.9,\Side) ;
                \draw [black, ->,line width=1pt, gray] (C) -- (0,0,\Side-0.9) ;

                \node[anchor=south west, red] at (O) {\footnotesize{$(1,0,0)$}};
                \filldraw[red] (O) circle (2pt);
                \draw [black, ->,line width=1pt, red] (O) -- (0.9,0,0) ;
                \draw [black, ->,line width=1pt, red] (O) -- (0,0.9,0) ;
                \draw [black, ->,line width=1pt, red] (O) -- (0,0,0.9) ;
                \draw[black!50, ->, line width=1pt] (0,\Side/2,\Side/2) -- ++(0.9,0,0);
                \node[anchor=north] at (0.1,\Side/2,\Side/2) {$\mathbf{e}_3$};

                \filldraw[red] (2*\Side,\Side,\Side) circle (2pt);

                \draw[->, thick] (2.5*\Side, 0.5*\Side,0.5*\Side) to[bend left] (3.5*\Side,0.5*\Side,0.5*\Side);

                \begin{scope}[shift={(4*\Side,0,0)}]
                        \coordinate (O) at (0,0,0);
                        \coordinate (A) at (0,\Side,0);
                        \coordinate (B) at (0,\Side,\Side);
                        \coordinate (C) at (0,0,\Side);
                        \coordinate (D) at (\Side,0,0);
                        \coordinate (E) at (\Side,\Side,0);
                        \coordinate (F) at (\Side,\Side,\Side);
                        \coordinate (G) at (\Side,0,\Side);

                        \draw[black!50] (2*\Side,0,0) -- (2*\Side,\Side,0) --
                                (2*\Side,\Side,\Side) -- (2*\Side,0,\Side) -- cycle;
                        \node[black,anchor=west] at (\Side,0,0) {\large{$\star$}};
                        \node[black,anchor=west] at (\Side,\Side,0) {\large{$\star$}};
                        \node[black,anchor=west] at (\Side,0,\Side) {\large{$\star$}};

                        \draw[black] (O) -- (C) -- (G) -- (D) -- cycle;
                        \draw[black] (O) -- (A) -- (E) -- (D) -- cycle;
                        \draw[black] (O) -- (A) -- (B) -- (C) -- cycle;
                        \draw[black] (D) -- (E) -- (F) -- (G) -- cycle;
                        \draw[black] (C) -- (B) -- (F) -- (G) -- cycle;
                        \draw[black] (A) -- (B) -- (F) -- (E) -- cycle;
                        \draw[red, line width=1.25pt,->] (O) -- (\Side, 0, 0) -- (\Side, \Side, 0)
                                -- (\Side, \Side, \Side);

                        \foreach \y in {0,\Side} {
                                \foreach \z in {0,\Side} {
                                        \draw[black!50] (\Side,\y,\z) -- (2*\Side,\y,\z);
                                        \filldraw[black] (\Side,\y,\z) circle (2pt);
                                        \filldraw[black] (2*\Side,\y,\z) circle (2pt);
                                }
                        }

                        \filldraw[blue] (A) circle (2pt);

                        \filldraw[brown] (B) circle (2pt);

                        \filldraw[gray] (C) circle (2pt);

                        \filldraw[red] (O) circle (2pt);

                        \filldraw[red] (2*\Side,\Side,\Side) circle (2pt);
                        \filldraw[red] (\Side, \Side, \Side) circle (2pt);
                \end{scope}

                \draw[->, thick, densely dotted] (5*\Side, -0.5*\Side,0.5*\Side) to[bend left] (5*\Side,-1*\Side,0.5*\Side);

                \begin{scope}[shift={(4*\Side,-2.5*\Side,0)}]
                        \coordinate (O) at (0,0,0);
                        \coordinate (A) at (0,\Side,0);
                        \coordinate (B) at (0,\Side,\Side);
                        \coordinate (C) at (0,0,\Side);
                        \coordinate (D) at (\Side,0,0);
                        \coordinate (E) at (\Side,\Side,0);
                        \coordinate (F) at (\Side,\Side,\Side);
                        \coordinate (G) at (\Side,0,\Side);

                        \draw[black!50] (2*\Side,0,0) -- (2*\Side,\Side,0) --
                                (2*\Side,\Side,\Side) -- (2*\Side,0,\Side) -- cycle;
                        \foreach \y in {0,\Side} {
                                \foreach \z in {0,\Side} {
                                        \draw[black!50] (\Side,\y,\z) -- (2*\Side,\y,\z);
                                        \filldraw[black] (\Side,\y,\z) circle (2pt);
                                        \filldraw[black] (2*\Side,\y,\z) circle (2pt);
                                        \node[black,anchor=west] at (\Side,\y,\z) {\large{$\star$}};
                                        \node[black,anchor=west] at (2*\Side,\y,\z) {\large{$\star$}};
                                }
                        }

                        \draw[black] (O) -- (C) -- (G) -- (D) -- cycle;
                        \draw[black] (O) -- (A) -- (E) -- (D) -- cycle;
                        \draw[black] (O) -- (A) -- (B) -- (C) -- cycle;
                        \draw[black] (D) -- (E) -- (F) -- (G) -- cycle;
                        \draw[black] (C) -- (B) -- (F) -- (G) -- cycle;
                        \draw[black] (A) -- (B) -- (F) -- (E) -- cycle;

                        \filldraw[blue] (A) circle (2pt);

                        \filldraw[brown] (B) circle (2pt);

                        \filldraw[gray] (C) circle (2pt);

                        \filldraw[red] (O) circle (2pt);
                \end{scope}

                \draw[->, thick, densely dotted] (3.5*\Side, -2*\Side,0.5*\Side) to[bend left] (2.25*\Side,-2*\Side,0.5*\Side);

                \begin{scope}[shift={(0,-2.5*\Side,0)}, xscale=0.4]
                        \coordinate (O) at (0,0,0);
                        \coordinate (A) at (0,\Side,0);
                        \coordinate (B) at (0,\Side,\Side);
                        \coordinate (C) at (0,0,\Side);
                        \coordinate (D) at (\Side,0,0);
                        \coordinate (E) at (\Side,\Side,0);
                        \coordinate (F) at (\Side,\Side,\Side);
                        \coordinate (G) at (\Side,0,\Side);

                        \foreach \x in {1,...,4}{
                                \foreach \y in {0,\Side} {
                                        \foreach \z in {0,\Side} {
                                                \draw[black!50] (\x*\Side,0,0) -- (\x*\Side,\Side,0) --
                                                        (\x*\Side,\Side,\Side) -- (\x*\Side,0,\Side) -- cycle;
                                                \draw[black!50] (\x*\Side,\y,\z) -- (2*\Side,\y,\z);
                                        }
                                }
                        }

                        \foreach \y in {0,\Side} {
                                \foreach \z in {0,\Side} {
                                        \draw[black!50, dashed] (4*\Side, \y, \z) --
                                                (4.9*\Side, \y,\z);
                                }
                        }

                        \foreach \x in {1,...,4}{
                                \foreach \y in {0,\Side} {
                                        \foreach \z in {0,\Side} {
                                                \filldraw[black] (\x*\Side,\y,\z) circle (2pt);
                                                \node[black,anchor=west] at (\x*\Side,\y,\z) {\large{$\star$}};
                                        }
                                }
                        }

                        \draw[black] (O) -- (C) -- (G) -- (D) -- cycle;
                        \draw[black] (O) -- (A) -- (E) -- (D) -- cycle;
                        \draw[black] (O) -- (A) -- (B) -- (C) -- cycle;
                        \draw[black] (D) -- (E) -- (F) -- (G) -- cycle;
                        \draw[black] (C) -- (B) -- (F) -- (G) -- cycle;
                        \draw[black] (A) -- (B) -- (F) -- (E) -- cycle;

                        \filldraw[blue] (A) circle (2pt);

                        \filldraw[brown] (B) circle (2pt);

                        \filldraw[gray] (C) circle (2pt);

                        \filldraw[red] (O) circle (2pt);
                \end{scope}
        \end{tikzpicture}
        \caption{\label{fig:ergodicity_proof_bi} Path from proof of part (B.i)
                for $d=3$. The first image (top left) shows the part on $H_d$
                of a $H_d$-good configuration with the propagation directions
                of the involved vacancies. To remove the
                $(0,0,1)$-vacancy on $(0,1,2)$ we use the red path
                from the second image. Iterating this procedure to put $\star$
                on all the black vertices (of initially arbitrary state) at
                $(\cdot, \cdot, 2)$ (third image). 
                This procedure iterates to any $(\cdot, \cdot, k)$ for $k\ge 2$
                (fourth picture).}
\end{figure}
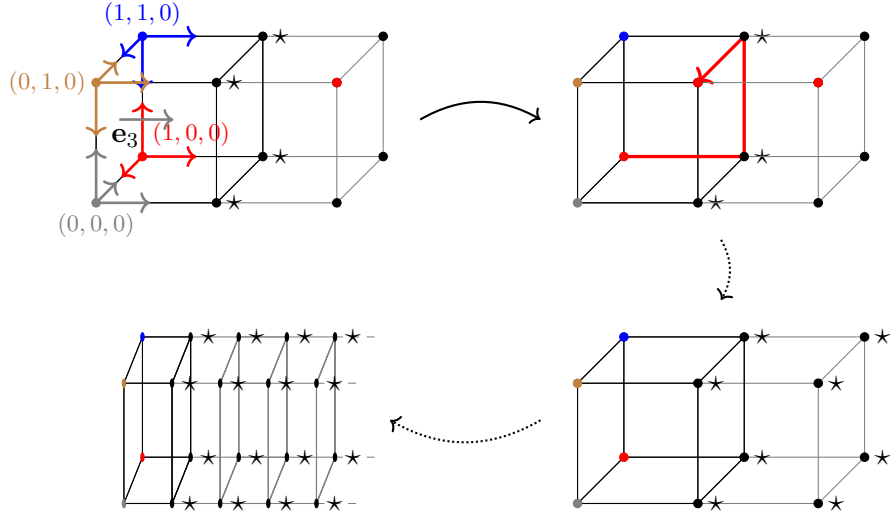

For $x\in \Z^d$ we say that $\omega$ is $(H_d+x)$-good if $\eta$ given by
$\eta_y=\omega_{y+x}$ is $H_d$ good.
Let $V_n=\{x\in \Z^d\colon x\cdot \mathbf{e}_d \ge -n \}$ for $n\in \Z$ so that
${\{V_n\}}_{n\in \Z}$ is an increasing and exhausting family of subsets of
$\Z^d$. With $\mathcal{A}_{x,j} := \{\omega\colon \omega \ \text{is}\
(H_d+x-(j+1)\mathbf{e}_d)\text{-good}\}$ we find that the family
$\mathcal{A}_x'(N)=\cup_{j\in [N]} \mathcal{A}_{x,j}$ for $N\ge 1$ satisfies the exterior
condition with respect to ${\{V_n\}}_{n\in \Z}$. The support of
$\mathcal{A}_x'(N)$ increases linearly in $N$ but the equilibrium failure
probability decreases exponentially in $N$. Thus, we can choose $N$ large
enough for \cref{eqn:ext_cond_thm_assumption_original} to hold and with
\cref{lemma:var_as_trans} we find a constant $C(\mathbf{q})$ such that
\begin{equation}
        \var(f)
        \le 4 \sum_{x\in
        \Z^d}\mu\left(\mathds{1}_{\mathcal{A}_x'(N)}\var_x(f)\right)
        \le C(\mathbf{q}) \sum_{x\in \Z^d}\sum_{j\in
        [N]}\sum_{h\in G}\sum_{\omega\in\mathcal{A}_{x,j}}\mu(\omega)
        {(\nabla_x^{(h)}f(\omega))}^2.\label{eqn:ergodicity_resolve_sum}
\end{equation}
Fix some $x\in\Z^d$, $j\in [N]$, $h\in G$ and $\omega\in \mathcal{A}_{x,j}$ and
assume w.l.o.g.\ that $\omega_x=h$. By the
above observations and translation invariance of the dynamics we find a legal
path $(\sigma^{(1)}, \ldots, \sigma^{(m)})$ with $m=O(N^2 2^{2d})$,
$\sigma^{(1)}=\omega$ and $\sigma^{(m)}=\sigma$ where $\sigma$ is the state
given by $\sigma_{x}=\star$ and $\sigma_{\Z^d\setminus
\{x\}}=\omega_{\Z^d\setminus \{x\}}$. Using the path method gives
\begin{equation}
        \mu(\omega)
        {(\nabla_x^{(h)}f(\omega))}^2
        \le C(\mathbf{q},m) \mu_{\Z^d\setminus\L_j(x)}(\omega)
        \mathcal{D}_{\Lambda_j(x)}(f)(\omega)
\end{equation}
where $\Lambda_j(x)$ is the smallest box containing both the support of
$\mathcal{A}_{x,j}$ and the origin. Using that $\Lambda_j(x)$ is finite for any
$x$ and $j$ we get
\begin{align}
        \var(f) 
        &\le C(\mathbf{q},m) \sum_{x\in \Z^d}\sum_{j\in
        [N]}\sum_{h\in G}\sum_{\omega\in\mathcal{A}_{x,j}}
        \mu_{\Z^d\setminus\L_j(x)}(\omega)
        \mathcal{D}_{\Lambda_j(x)}(f)(\omega)\\
        &\le C(\mathbf{q},m) \sum_{x\in \Z^d}\sum_{j\in
        [N]}\mu(-f\mathcal{L}_{\Lambda_j(x)}
        f)\\
        &\le C(\mathbf{q},m)\mathcal{D}(f),
\end{align}
where the $C(\mathbf{q},m)$ may be different from line to line (as is the
convention for constants throughout this paper). By the variational
characterisation of the spectral gap we thus have
\begin{equation}
        \gamma(G, \mathbf{q}) > 1/C(\mathbf{q},m).
\end{equation}
which is the claim.\qed

\subsection{G as a star graph: Proof of
        \texorpdfstring{\cref{thm:ergodicity}}{Theorem 1}(B.ii)}%
\label{sec:star_graph}
By \cref{lemma:monotonicity_in_g}, assume w.l.o.g.\ that $G=\{h_c,
h_1,\ldots, h_d\}$ where $h_c=0$ is the central vertex of $G$ and
$h_i=\mathbf{e}_i$, $i\in [d]$. We have
$\mathcal{P}(h_c)=\{\mathbf{e}_1,\ldots, \mathbf{e}_d\}$ and
$\mathcal{P}(h_i)=\{\mathbf{e}_1,\ldots, \mathbf{e}_{i-1},
-\mathbf{e}_i,\mathbf{e}_{i+1}, \ldots, \mathbf{e}_d\}$ so that the direction
$-\mathbf{e}_i$ is unique to $h_i$.

For $x\in \Z^d$ and $N_1,\ldots, N_d\in\N$ we say that the set
$x+\bigotimes_{i\in [d]}\{0,\ldots,N_i\}$ is a \emph{box with side lengths
$(N_1,\ldots,N_d)$ and origin $x$}. We call $x+(N_1,\ldots, N_d)$ the top right
corner of $B$. 
Let $\Lambda$ be the equilateral box of side
length $2$ and origin at $0$, i.e.\ $\Lambda={\{0,1,2\}}^d$. We call the vertex $x\in \Lambda$ a corner of
$\Lambda$ if $x_i\in \{0,2\}$ for all $i\in [d]$ and write $F_i=\{x\in
\Lambda\colon x_i=0\}$. For a configuration $\omega\in \Omega$ we say
that $\Lambda$ is good if $\omega_{2h}=h$ for every $h\in G$ and
$\omega_x=\star$ for $x\in \cup_{i}F_i\setminus \{2h\colon h\in G\}$.
Analogously define good boxes $\Lambda+x$ for any $x\in \Z^d$.
\begin{lemma}\label{lemma:move_good}
        If $\omega\in \Omega$ is such that $\Lambda$ is good and for each $i\in
        [d]$ there is a smallest $k_i\ge 2$ with
        $\omega_{\mathbf{v}+k_i\mathbf{e}_i}=h_i$, where we write
        $\mathbf{v}=\sum_{i\in [d]}\mathbf{e}_i$, then there is a legal path
        starting at $\omega$ and ending at a configuration $\sigma$ such that
        \begin{enumerate}[(i)]
                \item $\Lambda+\mathbf{v}$ is good in $\sigma$, and
                \item $\sigma_x=\star$ for any $x$ between $\Lambda$ and
                        $\mathbf{v}+k_i\mathbf{e}_i$, i.e.\ any $x\in
                        \cup_i\{\mathbf{v}+j\mathbf{e}_i\colon
                        3\le j\le k_i-1\}$ (if $k_i\ge 4$), and
                \item $\omega$ and $\sigma$ agree otherwise.
        \end{enumerate}
\end{lemma}
\begin{proof}
        We start by showing that there is a legal path that puts any state on
        $\Lambda\setminus \cup_i F_i$ and then we show how to use this to get
        that $\Lambda+\mathbf{v}$ is good. The steps are visualised for $d=2$
        in \cref{fig:ergodicity_proof_bii}.

        \textbf{Relax $\Lambda\setminus \cup_i F_i$:} Fix an $\omega$ as in
        the claim. Consider the vertex $h_c+\mathbf{v}=\mathbf{v}\in \Lambda$.
        For any $i\in [d]$ we have $\mathbf{v}-\mathbf{e}_i\in F_i\setminus
        \{2h\colon h\in G\}$ and thus $\omega_{\mathbf{v}-\mathbf{e}_i}=\star$.
        Now let $j\in [d]$, $j\neq i$. Since $-\mathbf{e}_j$ is the unique
        negative unit vector in $\mathcal{P}(h_j)$, there is a path\footnote{We
        use the term \emph{path} to mean paths on the lattice $\Z^d$ and the term \emph{legal path} to mean paths of
        configurations in $\Omega$ that are legal in the $G$-MCEM.} from $\mathbf{e}_j$
        to $\mathbf{v}-\mathbf{e}_i$ contained in $F_i\setminus \{2h\colon h\in G\}$
        consisting only of steps in $\mathcal{P}(h_j)$. Similar considerations apply to
        paths from the origin containing only steps in $\mathcal{P}(h_c)$. Thus,
        recalling that $\mathbf{e}_i\in \mathcal{P}(h_j)$, if $\Lambda$ is good there
        is a legal path that removes any non-$h_i$-vacancy from $\mathbf{v}$. Since $i$
        was arbitrary any vacancy type on $\mathbf{v}$ can be removed and by
        reversibility also any vacancy type can be put. Hence, we can also remove any
        non $h_i$-vacancy from $\mathbf{v}+\mathbf{e}_i$.

        Now use that $\mathbf{v}+\mathbf{e}_i-\mathbf{e}_j$ for $j\neq i$ is in
        $F_j$, if $d>2$ it is even in $F_j\setminus \{2h\colon h\in G\}$. If
        $v+\mathbf{e}_i-\mathbf{e}_j\neq \mathbf{e}_i$,
        there is a path contained in $F_j\setminus \{2h\colon h\in G\}$ from
        $\mathbf{e}_i$ to $\mathbf{v}+\mathbf{e}_i-\mathbf{e}_j$ consisting
        only of steps in $\mathcal{P}(h_i)$ so that there is a legal path that
        removes any vacancy from $\mathbf{v}+\mathbf{e}_i$. Since $i$ was
        arbitrary again we find a legal path that can put or remove any vacancy
        from $\mathbf{v}+\mathbf{e}$, for a canonical base vector
        $\mathbf{e}\in \mathcal{B}$. Analogously, it follows by induction in
        $n$ that we can remove or put any vacancy type on $x\in
        \Lambda\setminus \cup_i F_i$ with $\|x-\mathbf{v}\|_1=n$, where we just
        proved the base case $n=1$.
        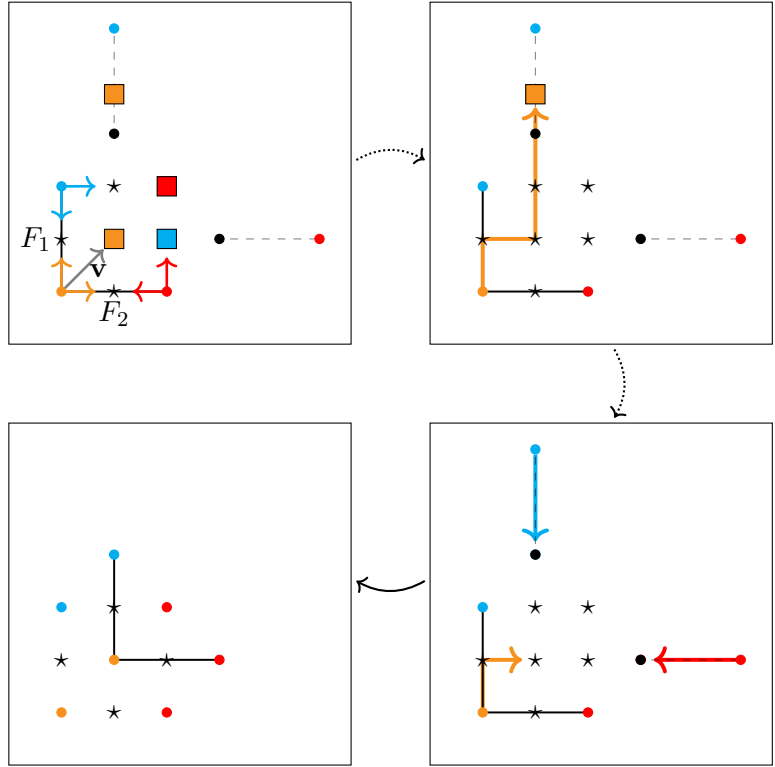
\begin{figure}[ht]
                \centering
                \begin{tikzpicture}[scale=0.7]

                        \draw[thin] (-1,-1) -- ++(0,6.5) -- ++(6.5,0) -- ++(0,-6.5) -- cycle;
                        \draw[line width=0.75pt] (0,2) -- ++(0,-2) -- ++(2,0);
                        \node[anchor=east] at (0,1) {$F_1$};
                        \node[anchor=north] at (1,0) {$F_2$};

                        \draw[line width=1pt, ->, black!50] (0,0) -- (0.8,0.8);
                        \node[anchor=north] at (0.7,0.7) {$\mathbf{v}$};
                        \draw[dashed, opacity=0.5] (1,2.9) -- (1,5);
                        \draw[dashed, opacity=0.5] (2.9,1) -- (4.9,1);

                        \draw[line width=1pt, shorten >=5pt, ->, Cyan] (0,2) -- ++(0,-0.9);
                        \draw[line width=1pt, shorten >=5pt, ->,Cyan ] (0,2) -- ++(0.9,0);
                        \filldraw[Cyan](0,2)circle[radius=2.5pt] {};
                        \draw[line width=1pt, shorten >=5pt, ->, BurntOrange] (0,0) -- ++(0,0.9);
                        \draw[line width=1pt, shorten >=5pt, ->, BurntOrange] (0,0) -- ++(0.9,0);
                        \filldraw[BurntOrange](0,0)circle[radius=2.5pt] {};
                        \draw[line width=1pt, shorten >=5pt, ->, red] (2,0) -- ++(0,0.9);
                        \draw[line width=1pt, shorten >=5pt, ->, red] (2,0) -- ++(-0.9,0);
                        \filldraw[red](2,0)circle[radius=2.5pt] {};

                        \filldraw[Cyan](1,5)circle[radius=2.5pt] {};
                        \node[draw, fill=BurntOrange, minimum size =1.01pt] (T)at (1,3.75){};
                        \filldraw[red](4.9,1)circle[radius=2.5pt] {};
                        \node[] at (0,1) {$\star$};
                        \node[] at (1,0) {$\star$};
                        \node[draw, fill=BurntOrange, minimum size =1.01pt] (T)at (1,1){};
                        \node[draw, fill=Cyan, minimum size =1.01pt] (T)at (2,1){};
                        \node[] at (1,2) {$\star$};
                        \node[draw, fill=red, minimum size =1.01pt] (T)at (2,2){};
                        \filldraw[black](3,1)circle[radius=2.5pt] {};
                        \filldraw[black](1,3)circle[radius=2.5pt] {};

                        \draw[->, thick, densely dotted] (5.6, 2.5) to[bend left] (6.9, 2.5);

                        \begin{scope}[shift={(8,0)}]

                                \draw[thin] (-1,-1) -- ++(0,6.5) -- ++(6.5,0) -- ++(0,-6.5) -- cycle;
                                \draw[line width=0.75pt] (0,2) -- ++(0,-2) -- ++(2,0);

                                \draw[BurntOrange, line width=1.5pt,shorten >=5pt, ->] (0,0) -- ++(0,1) -- ++(1,0) -- (1,3.75);

                                \draw[dashed, opacity=0.5] (1,2.9) -- (1,5);
                                \draw[dashed, opacity=0.5] (2.9,1) -- (4.9,1);

                                \filldraw[Cyan](0,2)circle[radius=2.5pt] {};
                                \filldraw[BurntOrange](0,0)circle[radius=2.5pt] {};
                                \filldraw[red](2,0)circle[radius=2.5pt] {};

                                \filldraw[Cyan](1,5)circle[radius=2.5pt] {};
                                \node[draw, fill=BurntOrange, minimum size =1.01pt] (T)at (1,3.75){};
                                \filldraw[red](4.9,1)circle[radius=2.5pt] {};
                                \filldraw[black](3,1)circle[radius=2.5pt] {};
                                \filldraw[black](1,3)circle[radius=2.5pt] {};
                                \node[] at (0,1) {$\star$};
                                \node[] at (1,0) {$\star$};
                                \node[] at (1,1) {$\star$};
                                \node[] at (2,1) {$\star$};
                                \node[] at (1,2) {$\star$};
                                \node[] at (2,2) {$\star$};

                        \end{scope}

                        \draw[->, densely dotted, thick] (10.5, -1.1) to[bend left] (10.5, -2.4);

                        \begin{scope}[shift={(8,-8)}]
                                \draw[thin] (-1,-1) -- ++(0,6.5) -- ++(6.5,0) -- ++(0,-6.5) -- cycle;
                                \draw[BurntOrange, line width=1.5pt,shorten >=5pt, ->] (0,0) -- ++(0,1) -- ++(1,0);
                                \draw[Cyan, line width=1.5pt,shorten >=5pt, ->] (1,4.9) -- (1,3);
                                \draw[red, line width=1.5pt,shorten >=5pt, ->] (4.9,1) -- ++(-1.9,0);
                                
                                \draw[line width=0.75pt] (0,2) -- ++(0,-2) -- ++(2,0);

                                \draw[dashed, opacity=0.5] (1,2.9) -- (1,5);
                                \draw[dashed, opacity=0.5] (2.9,1) -- (4.9,1);

                                \filldraw[Cyan](0,2)circle[radius=2.5pt] {};
                                \filldraw[BurntOrange](0,0)circle[radius=2.5pt] {};
                                \filldraw[red](2,0)circle[radius=2.5pt] {};

                                \filldraw[Cyan](1,5)circle[radius=2.5pt] {};
                                \filldraw[Cyan](1,3)circle[radius=2.5pt] {};
                                \filldraw[red](4.9,1)circle[radius=2.5pt] {};
                                \node[] at (1,1) {$\star$};
                                \node[] at (0,1) {$\star$};
                                \node[] at (1,0) {$\star$};
                                \node[] at (2,1) {$\star$};
                                \node[] at (1,2) {$\star$};
                                \node[] at (2,2) {$\star$};

                                \filldraw[red](3,1)circle[radius=2.5pt] {};
                                \filldraw[black](3,1)circle[radius=2.5pt] {};
                                \filldraw[black](1,3)circle[radius=2.5pt] {};
                        \end{scope}

                        \draw[->, thick] (6.9, -5.5) to[bend left] (5.6, -5.5);

                        \begin{scope}[shift={(0,-8)}]

                                \draw[thin] (-1,-1) -- ++(0,6.5) -- ++(6.5,0) -- ++(0,-6.5) -- cycle;
                                \draw[line width=0.75pt] (1,3) -- ++(0,-2) -- ++(2,0);

                                \filldraw[Cyan](0,2)circle[radius=2.5pt] {};
                                \filldraw[BurntOrange](0,0)circle[radius=2.5pt] {};
                                \filldraw[red](2,0)circle[radius=2.5pt] {};

                                \filldraw[Cyan](1,3)circle[radius=2.5pt] {};
                                \filldraw[BurntOrange](1,1)circle[radius=2.5pt] {};
                                \node[] at (1,0) {$\star$};
                                \node[] at (0,1) {$\star$};
                                \node[] at (2,1) {$\star$};
                                \node[] at (1,2) {$\star$};
                                \filldraw[red](2,2)circle[radius=2.5pt] {};

                                \filldraw[red](3,1)circle[radius=2.5pt] {};

                        \end{scope}

                \end{tikzpicture}
                \caption{Example configuration from \cref{lemma:move_good} in
                        the top-left. The black circles indicate the vertices
                        to which we need to bring the $(0,1)$-vacancy (blue) resp.\
                        $(1,0)$-vacancy (red). The squares indicate the vacancies
                        we replace with $\star$ in the course of the proof. From the first to the
                        second image any vacancy on $\Lambda\setminus \cup_i F_i$
                        is replaced with $\star$. For the third image we then
                        bring the necessary vacancies to $\Lambda+\mathbf{v}$ and
                        the final image shows how this gives a good
                        configuration on $\Lambda+\mathbf{v}$.}\label{fig:ergodicity_proof_bii}
        \end{figure}

        \textbf{Make $\Lambda+\mathbf{v}$ good:} For $i\in[d]$ we want to
        find a legal path which puts $h_i$ on $2\mathbf{e}_i+\mathbf{v}$ and,
        if $k_i\ge 4$, also puts the neutral state $\star$ on
        $\{\mathbf{v}+j\mathbf{e}_i\colon 3\le j \le k_i-1\}$. Then remove any
        other vacancies from $\cup_i (F_i+\mathbf{v})\setminus
        \{2h+\mathbf{v}\colon h\in G\}$.

        Let $i\in [d]$ and assume w.l.o.g.\ that $k_i>2$ since otherwise the
        $h_i$-vacancy is already at the correct position for
        $\Lambda+\mathbf{v}$ to be good. We already know that we can put any
        vacancy on $\mathbf{v}$ and since $\mathbf{e}_i\in \mathcal{P}(h)$ for
        every $h\in G\setminus\{h_i\}$ there is a legal path that removes any
        vacancy from $\{\mathbf{v}+j\mathbf{e}_i\colon j\in [1,k_i-1]\}$ (that
        are by assumption not $h_i$-vacancies since $k_i$ is the smallest integer such
        that $\omega_{\mathbf{v}+k_i\mathbf{e}_i}=h_i$). Then, use that
        $-\mathbf{e}_i\in \mathcal{P}(h_i)$ to bring the $h_i$ from
        $\mathbf{v}+k_i\mathbf{e}_i$ to $\mathbf{v}+2\mathbf{e}_i$ and put
        $\star$ in between $\mathbf{v}+2\mathbf{e}_i$ and
        $\mathbf{v}+k_i\mathbf{e}_i$.

        Since $i$ was arbitrary we can put $h_i$ on
        $\mathbf{v}+2\mathbf{e}_i$ for any $i$. Using again that we can put
        $\mathbf{v}$ into any state we can, in particular, put
        $h_c$ on $\mathbf{v}$. Thus, we get a legal path that puts $h$ on $\mathbf{v}+2h$ for each $h\in G$ and $\star$ on
        $\cup_i\{\mathbf{v}+j\colon j\in [3,k_i-1]\}$ where the final state
        still agrees with $\omega$ outside these vertices.

        Finally, we need to put the neutral state on $x \in(F_i+\mathbf{v})\setminus \Lambda$
        which is an analogous induction to the one for putting
        the neutral state on $\Lambda\setminus \cup_{i\in [d]}F_i$ detailed
        above.
\end{proof}
The Lemma tells us that we can move a good $\Lambda$ in the direction
$\mathbf{v}$, given enough non-central vacancies outside of $\Lambda$. To
satisfy the exterior condition this is too lose a condition as we cannot always
assume that we find these vacancies for each step. The next Lemma gives another
construction that does not require new vacancies after every step.
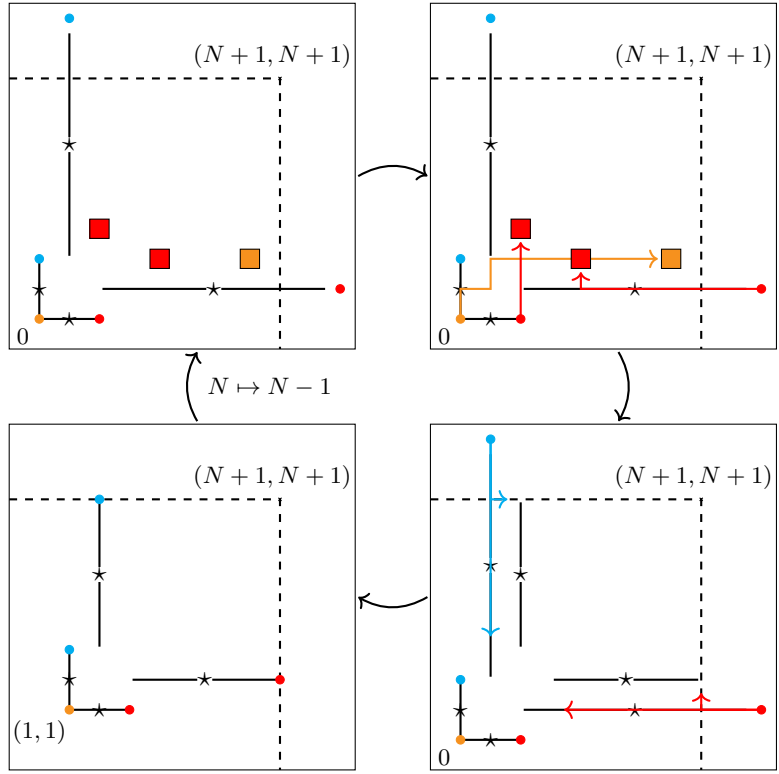
\begin{figure}[t]
        \centering
        \begin{tikzpicture}[scale=0.4]
                \draw[thin] (-1,-1) -- ++(11.5,0) -- ++(0,11.5) -- ++(-11.5,0) -- cycle;
                \draw[thick, black] (2,0) -- ++(-2,0) -- ++(0,2);

                \filldraw[red](10,1)circle[radius=4pt] {};
                \draw[thick] (2.1,1) -- (5.55,1);
                \node[] at (5.8,1) {$\star$};
                \draw[thick] (6.05,1) -- (9.5,1);

                \filldraw[Cyan](1,10)circle[radius=4pt] {};
                \draw[thick] (1,2.1) -- (1,5.55);
                \node[] at (1,5.8) {$\star$};
                \draw[thick] (1,6.05) -- (1,9.5);

                \draw(8,8) node[cross, black]{};
                \node[anchor=south] at (7.75,8) {\footnotesize $(N+1,N+1)$};

                \node[anchor=north east] at (0,0) {\footnotesize $0$};

                \draw[thick, dashed, black]  (-1,8) -- ++(9,0) -- ++(0,-9);

                \filldraw[Cyan](0,2)circle[radius=4pt] {};
                \node[] at (0,1) {$\star$};
                \filldraw[BurntOrange](0,0)circle[radius=4pt] {};
                \node[] at (1,0) {$\star$};
                \filldraw[red](2,0)circle[radius=4pt] {};



                \node[draw, fill=red, minimum size =1.01pt] (T)at (4,2){};
                \node[draw, fill=BurntOrange, minimum size =1.01pt] (T)at (7,2){};
                \node[draw, fill=red, minimum size =1.01pt] (T)at (2,3){};

                \draw[->, thick] (10.6, 4.75) to[bend left] (12.9, 4.75);
                \begin{scope}[shift={(14,0)}]
                        \draw[thin] (-1,-1) -- ++(11.5,0) -- ++(0,11.5) -- ++(-11.5,0) -- cycle;
                        \draw[thick, black] (2,0) -- ++(-2,0) -- ++(0,2);

                        \filldraw[red](10,1)circle[radius=4pt] {};
                        \draw[thick] (2.1,1) -- (5.55,1);
                        \node[] at (5.8,1) {$\star$};
                        \draw[thick] (6.05,1) -- (9.5,1);

                        \filldraw[Cyan](1,10)circle[radius=4pt] {};
                        \draw[thick] (1,2.1) -- (1,5.55);
                        \node[] at (1,5.8) {$\star$};
                        \draw[thick] (1,6.05) -- (1,9.5);

                        \draw(8,8) node[cross, black]{};
                        \node[anchor=south] at (7.75,8) {\footnotesize
                                $(N+1,N+1)$};

                        \node[anchor=north east] at (0,0) {\footnotesize $0$};

                        \draw[thick, dashed, black]  (-1,8) -- ++(9,0) -- ++(0,-9);

                        \filldraw[Cyan](0,2)circle[radius=4pt] {};
                        \node[] at (0,1) {$\star$};
                        \filldraw[BurntOrange](0,0)circle[radius=4pt] {};
                        \node[] at (1,0) {$\star$};
                        \filldraw[red](2,0)circle[radius=4pt] {};

                        \draw[BurntOrange, thick,shorten >=5pt, ->] (0,0) -- ++(0,1) -- ++(1,0) -- ++(0,1) -- (7,2);
                        \draw[red, thick,shorten >=5pt, ->] (10,1) -- ++(-6,0) -- ++(0,1);
                        \draw[red, thick,shorten >=5pt, ->] (2,0) -- (2,3);

                        \node[draw, fill=red, minimum size =1.01pt] (T)at (4,2){};
                        \node[draw, fill=BurntOrange, minimum size =1.01pt] (T)at (7,2){};
                        \node[draw, fill=red, minimum size =1.01pt] (T)at (2,3){};
                \end{scope}

                \begin{scope}[shift={(14,-14)}]
                        \draw[->, thick] (5.25, 12.9) to[bend left] (5.25, 10.6);
                        \draw[thin] (-1,-1) -- ++(11.5,0) -- ++(0,11.5) -- ++(-11.5,0) -- cycle;
                        \draw[thick, black] (2,0) -- ++(-2,0) -- ++(0,2);

                        \filldraw[red](10,1)circle[radius=4pt] {};
                        \draw[thick] (2.1,1) -- (5.55,1);
                        \node[] at (5.8,1) {$\star$};
                        \draw[thick] (6.05,1) -- (9.5,1);

                        \filldraw[Cyan](1,10)circle[radius=4pt] {};
                        \draw[thick] (1,2.1) -- (1,5.55);
                        \node[] at (1,5.8) {$\star$};
                        \draw[thick] (1,6.05) -- (1,9.5);

                        \draw[thick] (3.1,2) -- (5.25,2);
                        \node[] at (5.5,2) {$\star$};
                        \draw[thick] (5.75,2) -- (7.9,2);

                        \draw[thick] (2,3.1) -- (2,5.25);
                        \node[] at (2,5.5) {$\star$};
                        \draw[thick] (2,5.75) -- (2,7.9);

                        \draw(8,8) node[cross, black]{};
                        \node[anchor=south] at (7.75,8) {\footnotesize $(N+1,N+1)$};

                        \node[anchor=north east] at (0,0) {\footnotesize $0$};

                        \draw[thick, dashed, black]  (-1,8) -- ++(9,0) -- ++(0,-9);

                        \filldraw[Cyan](0,2)circle[radius=4pt] {};
                        \node[] at (0,1) {$\star$};
                        \filldraw[BurntOrange](0,0)circle[radius=4pt] {};
                        \node[] at (1,0) {$\star$};
                        \filldraw[red](2,0)circle[radius=4pt] {};

                        \draw[red, thick,shorten >=5pt, ->] (10,1) -- (3,1);
                        \draw[red, thick,shorten >=5pt, ->] (10,1) -- (8,1) -- (8,2);

                        \draw[Cyan, thick,shorten >=5pt, ->] (1,10) -- (1,3);
                        \draw[Cyan, thick,shorten >=5pt, ->] (1,10) -- (1,8) -- (2,8);
                \end{scope}

                \begin{scope}[shift={(0,-14)}]
                        \draw[->, thick]  (5.25, 10.6) to[bend left](5.25, 12.9);
                        \node[anchor=west] at (5.25,11.75) {\footnotesize $N\mapsto N-1$};
                        \draw[->, thick] (12.9, 4.75) to[bend left] (10.6, 4.75) ;
                        \draw[thin] (-1,-1) -- ++(11.5,0) -- ++(0,11.5) -- ++(-11.5,0) -- cycle;
                        \draw[thick, black] (3,1) -- ++(-2,0) -- ++(0,2);

                        \draw[thick] (3.1,2) -- (5.25,2);
                        \node[] at (5.5,2) {$\star$};
                        \draw[thick] (5.75,2) -- (7.9,2);

                        \draw[thick] (2,3.1) -- (2,5.25);
                        \node[] at (2,5.5) {$\star$};
                        \draw[thick] (2,5.75) -- (2,7.9);

                        \draw(8,8) node[cross, black]{};
                        \node[anchor=south] at (7.75,8) {\footnotesize
                                $(N+1,N+1)$};

                        \node[anchor=north east] at (1.25,1) {\footnotesize $(1,1)$};

                        \draw[thick, dashed, black]  (-1,8) -- ++(9,0) -- ++(0,-9);

                        \filldraw[Cyan](1,3)circle[radius=4pt] {};
                        \node[] at (1,2) {$\star$};
                        \filldraw[BurntOrange](1,1)circle[radius=4pt] {};
                        \node[] at (2,1) {$\star$};
                        \filldraw[red](3,1)circle[radius=4pt] {};

                        \filldraw[Cyan](2,8)circle[radius=4pt] {};
                        \filldraw[red](8,2)circle[radius=4pt] {};
                \end{scope}
        \end{tikzpicture}
        \caption{The first image (top left) shows an example $\omega$ as in
                \cref{lemma:move_good_2} for two dimensions using the colour
                and shape code from \cref{fig:ergodicity_proof_bi}. The second
                image shows the paths used to get rid of any vacancies on
                $\{2\mathbf{v}+j\mathbf{e}_i\colon j\in [N-1]\}$. The third
                image then shows the paths to move the good box and
                put the $(0,1)$- resp. $(1,0)$-vacancy on the
                respective $\{2\mathbf{v}+(N-1)\mathbf{e}_i\}$. The fourth
                image then shows how the resulting state is the same as in
                \cref{lemma:move_good_2} translated by $\mathbf{v}$ so we can
                iterate the proof by setting $N\mapsto N-1$.\label{fig:move_good_2}}
\end{figure}
\begin{lemma}\label{lemma:move_good_2}
        Fix an $N\in 2\N$, $N\ge 4$. Let $\omega\in \Omega$ be such that
        $\Lambda$ is good and for each $i\in [d]$ there is a $k_i\in [N,3N/2]$
        with $\omega_{\mathbf{v}+k_i\mathbf{e}_i}=h_i$ and such that
        $\omega_y=\star$ for each $y\in \{\mathbf{v}+n\mathbf{e}_i\colon n\in
        [k_i-1]\}$. Then, there is a legal path starting at
        $\omega$ and ending at $\sigma$ such that $\Lambda+(N-2)\mathbf{v}$ is
        good and that agrees with $\omega$ otherwise.
\end{lemma}
\begin{proof}
        \cref{fig:move_good_2} illustrates a state $\omega$ as in the claim and
        the steps of the following proof. We start by clearing the line
        $\{2\mathbf{v}+n\mathbf{e}_i\colon n\in [N-1]\}$ of any vacancies and
        then move the good box by $\mathbf{v}$ so that we recover the initial
        situation and can iterate the argument.

        Fix an $i,j\in [d]$ s.t.\ $i\neq j$. We can
        bring the $h_i$-vacancy from $\mathbf{v}+k_i\mathbf{e}_i$ to
        $\mathbf{v}+n\mathbf{e}_i$ for any $n\in [N]$. Thus, we can remove any
        $h_i$-vacancy from $\{\mathbf{v}+n\mathbf{e}_i+\mathbf{e}_j\colon n\in
        [N]\}$. Since $\mathbf{v}+\mathbf{e}_j\in \Lambda\setminus \cup_i F_i$,
        we can put any $h$-vacancy on it. Thus,
        there is a legal path to  remove any vacancy from
        $\{\mathbf{v}+n\mathbf{e}_i+\mathbf{e}_j\colon n\in [N]\}$ using that
        $\mathbf{e}_i\in \mathcal{P}(h)$ for $h\neq h_i$. The chosen
        $j\neq i$ was arbitrary so that this works for any pair $i$, $j$.

        Analogously to the proof of \cref{lemma:move_good} this is the starting
        case for the proof that we can remove any vacancy from
        $\{\mathbf{v}+n\mathbf{e}_i+\sum_{i\in I} \mathbf{e}_i\colon n\in
        [N-1],I\subset [d]\setminus [i], |I|=m\}$ for any $m\le d-1$ by
        induction in $m$.

        In particular, we can remove any vacancy from $y\in
        2\mathbf{v}+(n-1)\mathbf{e}_i$ for $n\in [N]$ since
        $y=\mathbf{v}+n\mathbf{e}_i+\sum_{i\in [d]\setminus
        \{\mathbf{e}_i\}}\mathbf{e}_i$ is covered by the case $m=d-1$. The
        choice of $i\in [d]$ was arbitrary so that we can construct a legal
        path that ends in a state $\sigma^{(1)}$ on which
        $\cup_i \{2\mathbf{v}+n\mathbf{e}_i\colon n\in [N-2]\}$ is in the
        neutral state and $2\mathbf{v}+(N-1)\mathbf{e}_i$ has an $h_i$-vacancy.

        Further, by \cref{lemma:move_good} there is a path
        starting at $\sigma^{(1)}$ and ending in a state $\sigma^{(2)}$ in
        which $\Lambda+\mathbf{v}$ is good and which does not change
        $\cup_i\{2\mathbf{v}+n\mathbf{e}_i\colon n\in [N-1]\}$. The state
        $\sigma^{(2)}$ is now in the configuration of the claim for $N-1$ so
        that we can iterate the proof until we find a legal path
        that ends in a state with $\Lambda+(N-2)\mathbf{v}$ good.
\end{proof}
There is a useful property on boxes that
allows us to use both of the previous Lemmas.
\begin{definition}
        We say that a subset $R\subset \Z^d$ is \emph{colourful} in a
        configuration $\omega\in\Omega$ if for each vacancy type $h\in G$ on
        each straight line connecting two opposite boundaries of $R$ there is
        an $x$ such that $\omega_x=h$.
\end{definition}

For $N\in 2\N$ define the event $\mathcal{E}^{(N)}$ (see
\cref{fig:ergodicity_event}) as the set of configurations $\omega$ such that
        \begin{enumerate}[(E.i)]
                \item Let $B$ be the box with side lengths $(N,\ldots, N)$ and origin 
                        $-N\mathbf{v}$. The subset of $B$ given by the box with
                        side lengths $(N-1,\ldots, N-1)$ and origin
                        $-(N-1)\mathbf{v}$ \emph{excluding} the origin $0$ is
                        colourful.
                \item For each vertex $x\in B$ such that there is an $i\in[d]$
                                with $x\cdot\mathbf{e}_i=-N$ there is a $k\in
                                [N]$ and a good box $\Lambda$ such that
                                $x-k\mathbf{v}$ is the top right corner of
                                $\Lambda$.
                \item For each $i\in [d]$ the box with origin $-2N\mathbf{e}_i$
                        and side lengths given by $N$ in the $i$-direction and
                        $N/2$ otherwise is colourful.
        \end{enumerate}
        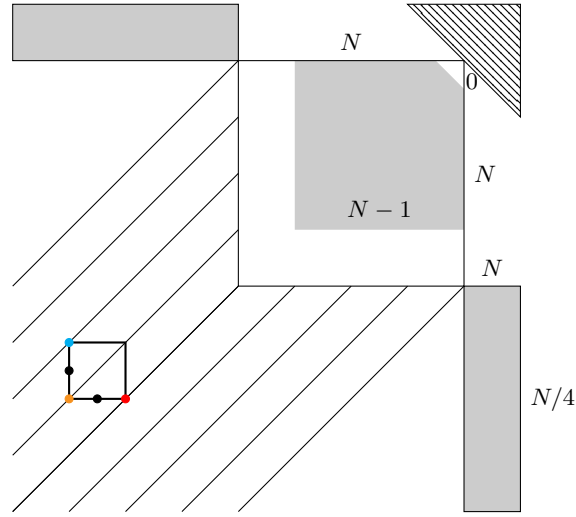
\begin{figure}[ht]
                \centering
                \begin{tikzpicture}[scale=0.75]
                        \node[anchor=north west] at (-0.15,-0.05) {\footnotesize $0$};
                        \node[anchor=south] at (-2,0) {\footnotesize $N$};
                        \node[anchor=west] at (0,-2) {\footnotesize $N$};
                        \node[anchor=south] at (0.5,-4) {\footnotesize $N$};
                        \node[anchor=south] at (-1.5,-3) {\footnotesize $N-1$};
                        \node[anchor=west] at (1,-6) {\footnotesize $N/4$};
                        \draw[] (0,0) rectangle  (-4,-4);
                        \fill[black,opacity=0.2] (-0.5,0) -- (0,-0.5) -- (0,-3)
                                -- (-3,-3) -- (-3,0) -- cycle ;
                        \draw[pattern=north west lines]  (-1,1) -- ++(2,-2)
                                -- ++(0,2) -- cycle;
                        \draw[fill=black!20] (-4,0) rectangle (-8,1);
                        \draw[fill=black!20] (0,-4) rectangle (1,-8);

                        \foreach \x in {0,-1,...,-4} {
                                \draw (-4,\x) -- ++(-4,-4);
                                \draw (\x,-4) -- ++(-4,-4);
                        }
                        \draw[black, thick] (-7,-6) -- ++(1,0) -- ++(0,1)
                                --++(-1,0) -- cycle;

                        \filldraw[Cyan](-7,-5)circle[radius=2pt] {};
                        \filldraw[black](-7,-5.5)circle[radius=2pt] {};
                        \filldraw[BurntOrange](-7,-6)circle[radius=2pt] {};
                        \filldraw[black](-6.5,-6)circle[radius=2pt] {};
                        \filldraw[red](-6,-6)circle[radius=2pt] {};
                \end{tikzpicture}
                \caption{Support of the event $\mathcal{E}^{(N)}$ in $d=2$ and
                        $N=4$. The gray sets are the boxes that we assume to be
                        colourful. The diagonal lines are the lines that we
                        assume intersect a good box $\Lambda$.}
                \label{fig:ergodicity_event}
        \end{figure}
        Let $\mathcal{E}^{(N)}_x$ be the correspondingly translated event for
        $x\in \Z^d$. On $\Z^d$ consider the $d-1$ dimensional hyperplane $U_0$
        perpendicular to $\mathbf{v}$ that goes through the origin. By
        construction we have $\mathrm{Supp}(\mathcal{E}^{(N)})\subset V^c_0$
        where $V_0:=\cup_{\ell=1}^{\infty}(U_0+\ell \mathbf{v})$. And the
        family ${\{\mathcal{E}^{(N)}_x\}}_{x\in \Z^d}$ satisfies the exterior
        condition w.r.t.\ ${\{V_n\}}_{n\in \Z}$ where $V_n=V_0-n\mathbf{v}$.
        The failure probability of $\mathcal{E}_x^{(N)}$ can be upper bounded
        by a series of simple union bounds to get a bound that is exponentially
        decreasing with $N$ while the support grows polynomially in $N$. 

        We can thus choose $N$ large enough (depending on $\mathbf{q}$) such
        that \cref{thm:exterior_thm} then gives
        \begin{equation}
                \var(f) \le 4\sum_{x\in \Z^d}
                \mu(\mathds{1}_{\mathcal{E}_x^{(N)}}\var_x(f)).
        \end{equation}

        The proof then concludes analogously to the proof of (B.i) once we have
        defined the paths that allow us to remove or put any vacancy type on
        $x$. W.l.o.g.\ consider only the case $x=0$ and fix $\omega\in
        \mathcal{E}^{(N)}$. Consider a $y$ in the box $B$ from (E.i) such that
        $y\cdot \mathbf{e}_i=-N$ for some $i\in [d]$. By (E.ii) there is a $k\in
        [N]$ such that $y-k\mathbf{v}$ is in a good box $\Lambda$. For any
        $j\in [d]$ we are guaranteed to hit one of the boxes from (E.iii) or
        the subset of $B$ given in (E.i) when going in the direction
        $\mathbf{e}_j$ from $y-k\mathbf{v}$. Since these sets are colourful,
        we can apply \cref{lemma:move_good} to propagate the good box. Since
        this works for arbitrary $k$, we can propagate it until it intersects
        $B$ and put the neutral state on $y$ in a legal path of finite length.
        These transitions are independent of the state of $y$, so we find a
        legal path starting at $\omega$ and ending in a state $\sigma$ where
        $\sigma_y=\star$ and $\sigma_z=\omega_z$ for $z\neq y$.

        The vector $y$ was arbitrary so that we can iterate this argument until
        the conditions to apply \cref{lemma:move_good_2} are satisfied. Notice
        that for this we need to propagate and leave the good box around $-N\mathbf{v}$
        which follows by the same argument. The statement then follows
        analogously to (B.i).
        \qed

\section{Spectral gap bounds for the two-dimensional MCEM: Proof
of \texorpdfstring{{\cref{thm:abc_relaxation}}}{Theorem 2}}%
\label{sec:abc_spectral_gap}
The upper bound in \cref{thm:abc_relaxation} follows by
\cref{lemma:monotonicity_in_g}. The steps to prove the corresponding lower
bound are analogous to the proof of \cref{thm:ergodicity}(B). The main
difference is that we have to be careful about the cost of our intermediate
steps. Before we were fine estimating the spectral gap by any positive
constant. Now we want to show that in the highest order the spectral gap is
given by the spectral gap of the East model on $\Z^d$ which we explicitly
recall here as it is central to the proof.
\begin{theorem}[{\cite{chleboun2016relaxation}*{Theorem~1}}]\label{thm:east_relax}
        As $q\rightarrow 0$ the spectral gap $\gamma_d(q)$ of the East model on
        $\Z^d$ with parameter $q$ is given as
        \begin{equation}
                \gamma_d(q) = 2^{-\frac{\theta_q^2}{2d}(1+o(1))},
        \end{equation}
        where $\theta_q:=|\log_2(q)|$. 
\end{theorem}

\subsection{Preliminary constructions}\label{sec:hgrid}
Note that by \cref{lemma:monotonicity_in_g} the cases (3.i) and
(3.ii) imply the cases (2.i) and (2.ii). Using this and symmetry
considerations, w.l.o.g.\ we can assume
in the following that $G=\{(1,1),(0,0),(0,1)\}$. We call the associated MCEM the
$ABC$-model and call $A=(0,0)$, $B=(1,1)$, $C=(0,1)$ and $D=(1,0)$. As noted
in the introduction, by \cref{lemma:monotonicity_in_g} we have
\begin{align}
        \lim_{q_{\min}\rightarrow 0}
        \frac{\gamma(G,\mathbf{q})}{\gamma_2(q_{\min})} \le 1
\end{align}
to prove \cref{thm:abc_relaxation} we thus need the corresponding lower bound.

Analogously to bounding the spectral gap from zero our strategy for finding
good lower bounds on the spectral gap relies on the exterior condition theorem,
\cref{thm:exterior_thm}.
Fix a $G\subset H_2$ and consider a family ${\{\mathcal{A}_x\}}_{x\in \Z^2}$ of
events that satisfies the requirements of the exterior condition theorem so
that with \cref{lemma:var_as_trans} we have
\begin{equation}\label{eqn:result_from_exterior}
        \var(f)\le 4\sum_{x\in \Z^2}\mu(\mathds{1}_{\mathcal{A}_x}\var_x(f))
        \le \frac{4}{p} \sum_{h\in G}\sum_{x\in \Z^2}
        \mu\left[\mathds{1}_{\mathcal{A}_x}pq_h {(\nabla_x^{(h)}f)}^2\right].
\end{equation}
Since $p q_h {(\nabla_x^{(h)}f)}^2 =
\var_x(f\mathds{1}_{\{\star,h\}})=:\var_x(f\tc \{\star, h\})$
we can treat the transition for each vacancy type separately. The main
difficulty in finding good lower bounds on the spectral gap is then to identify
events $\mathcal{A}_x$ that satisfy the exterior condition, have a low failing
probability and such that for each $h\in G$ we have
\begin{equation}\label{eqn:particle_poincare}
        \mu\left[ \mathds{1}_{\mathcal{A}_x} \var_x(f\mid \{\star, h\})\right]
        \le
        2^{\theta_{q_{\min}}^2(1+\varepsilon)/4}\mu
        \left[\mathcal{D}_{\Lambda_h}(f)\right],
\end{equation}
for some $\Lambda_h$ such that the overlap of the various $\Lambda_h$ for the
different $x$ (and thus the overcounting term) can be absorbed into the $\varepsilon$ in
$2^{\theta_{q_{\min}}^2(1+\varepsilon)/4}$ for $q_{\min}$ small enough. We do
this by defining events $\mathcal{A}_x^{(h)}$ for each $h\in G$ and setting
$\mathcal{A}_x= \cap_{h\in G} \mathcal{A}_x^{(h)}$. Each $\mathcal{A}_x^{(h)}$
is defined such that it allows the rewriting of the local variance with
indicator $\mathds{1}_{\mathcal{A}^{(h)}_x}$ to a Dirichlet form by using a
mixture of auxiliary models that behave like the standard one- or
two-dimensional East model and the path method. Let us start by outlining the
construction used in the proof of part (3.i) and (3.ii).
\subsubsection{Geometric construction}\label{sub:renorm_geometric}
Let us start with some deterministic constructions for $h\in G$.
\begin{definition}[$h$-paths]
        For $h\in H_2$ we say that $\Gamma=(x_1,\ldots,x_n)\subset \Z^2$ is an
        $h$-path if $x_i-x_{i+1}\in \mathcal{P}(h)$ for $i\in [n-1]$, i.e.\
        starting from $x_n$ we can reach $x_1$ staying on $\Gamma$ and only
        using steps in $\mathcal{P}(h)$.
\end{definition}
\begin{remark}
        Note that we want $x_i-x_{i+1}$ to be a propagation direction of
        $\mathcal{P}(h)$ instead of the more intuitive direction from $x_i$ to
        $x_{i+1}$ (i.e.\ $x_{i+1}-x_i$). Defining it this way we can find an
        $h$-path starting from some vertex $x\in \Z^d$ and ending in a vertex
        containing an $h$-vacancy which can then travel on the $h$-path back to
        $x$.
\end{remark}
We build the $h$-grid first for $B$-vacancies and then explain how to
generalise to $h\in \{A,C\}$. We do the construction incrementally by
starting with a base cell for $B$-vacancies.
\begin{definition}[$B$-Base cell $Q$]\label{def:base_cell}
        Let $\ell\in 8\N$. Define $D^{(1)}\subset \Z^2$ as the $B$-path starting at
        $\mathbf{e}_1+3\mathbf{e}_2$ that first does an $\mathbf{e}_1$-step,
        then zigzags north and east for $2$ steps respectively until $\ell$
        steps east have been made with the last step being a single one. Then
        define $D^{(2)}\subset \Z^2$ as the path starting again at
        $\mathbf{e}_1+3\mathbf{e}_2$ which starts with $4$ steps north, goes
        one step east and then zigzags $8$ steps north and one step east until
        $\ell$ steps north have been made with the last step $4$ long
        instead of $8$. Then, define $D^{(3)} =
        D^{(1)}+\ell/8\mathbf{e}_1+\ell\mathbf{e}_2$ and
        $D^{(4)}=D^{(2)}+\ell\mathbf{e}_1+\ell\mathbf{e}_2$, i.e.\ the paths
        $D^{(1)}$ resp.\ $D^{(2)}$ shifted to start at the end point of
        $D^{(2)}$ resp. $D^{(1)}$. We then define the \emph{$B$-base cell $Q$
        with side length $\ell$} as the set of vertices enclosed by and
        including the boundaries $D^{(i)}$ for $i\in [4]$. We refer to
        $D^{(i)}$ as the bottom, left, top and right boundary of $Q$ for
        $i=1,2,3,4$ respectively (see left side of \cref{fig:base_cell}).
\end{definition}
\begin{figure}[ht]
        \centering
        \begin{subfigure}[b]{0.475\textwidth}
        \centering
        \begin{tikzpicture}[scale=0.24]
                        \draw[step=1.0, opacity=0.15](-3.5,-2.5) grid (20.5,33.5);

                        \draw(-1,-2) node[anchor=north east] {\footnotesize $0$};
                        \draw(-1,-2) node[cross, black]{};

                        \draw[] (0,1)
                         -- ++(1,0) -- ++(0,2)
                         -- ++(2,0) -- ++(0,2)
                         -- ++(2,0) -- ++(0,2)
                         -- ++(2,0) -- ++(0,2)
                         -- ++(2,0) -- ++(0,2)
                         -- ++(2,0) -- ++(0,2)
                         -- ++(2,0) -- ++(0,2)
                         -- ++(2,0) -- ++(0,2)
                         -- ++(1,0)
                         -- ++(0,4) -- ++(1,0)
                         -- ++(0,8) -- ++(1,0)
                         -- ++(0,4)
                         -- ++(-1,0) -- ++(0,-2)
                         -- ++(-2,0) -- ++(0,-2)
                         -- ++(-2,0) -- ++(0,-2)
                         -- ++(-2,0) -- ++(0,-2)
                         -- ++(-2,0) -- ++(0,-2)
                         -- ++(-2,0) -- ++(0,-2)
                         -- ++(-2,0) -- ++(0,-2)
                         -- ++(-2,0) -- ++(0,-2)
                         -- ++(-1,0)
                         -- ++(0,-4) -- ++(-1,0)
                         -- ++(0,-8) -- ++(-1,0)
                         -- cycle;

                        \draw(9,17) node[opacity=0.5] {\footnotesize $Q$};
                        \draw(1,10) node[opacity=0.5, anchor=east] {\footnotesize $D^{(2)}$};
                        \draw(8.5,9) node[opacity=0.5, anchor=north] {\footnotesize $D^{(1)}$};
                        \draw(17,25) node[opacity=0.5, anchor=west] {\footnotesize $D^{(4)}$};
                        \draw(9.8,25) node[opacity=0.5, anchor=south] {\footnotesize $D^{(3)}$};


                        \draw[decorate,decoration={brace,amplitude=8pt},xshift=0pt,yshift=-3pt]
                                (16,1) -- (0,1) node [black,midway,yshift=-.5cm] 
                                {\footnotesize $\ell$};
        \end{tikzpicture}
        \end{subfigure}
        \hfill
        \begin{subfigure}[b]{0.475\textwidth}
        \centering
        \begin{tikzpicture}[scale=0.24]
                        \draw[step=1.0, opacity=0.15](-3.5,-2.5) grid (20.5,33.5);

                        \draw(-1,-2) node[anchor=north east] {\footnotesize $0$};
                        \draw(-1,-2) node[cross, black]{};

                        \draw[] (0,1)
                         -- ++(1,0) -- ++(0,2)
                         -- ++(2,0) -- ++(0,2)
                         -- ++(2,0) -- ++(0,2)
                         -- ++(2,0) -- ++(0,2)
                         -- ++(2,0) -- ++(0,2)
                         -- ++(2,0) -- ++(0,2)
                         -- ++(2,0) -- ++(0,2)
                         -- ++(2,0) -- ++(0,2)
                         -- ++(1,0)
                         -- ++(0,4) -- ++(1,0)
                         -- ++(0,8) -- ++(1,0)
                         -- ++(0,4)
                         -- ++(-1,0) -- ++(0,-2)
                         -- ++(-2,0) -- ++(0,-2)
                         -- ++(-2,0) -- ++(0,-2)
                         -- ++(-2,0) -- ++(0,-2)
                         -- ++(-2,0) -- ++(0,-2)
                         -- ++(-2,0) -- ++(0,-2)
                         -- ++(-2,0) -- ++(0,-2)
                         -- ++(-2,0) -- ++(0,-2)
                         -- ++(-1,0)
                         -- ++(0,-4) -- ++(-1,0)
                         -- ++(0,-8) -- ++(-1,0)
                         -- cycle;

                        \draw[red, line width=0.5mm] (3,5) -- ++(0,5) -- ++(2,0) -- ++(0,2) -- ++(1,0) --
                                ++(0,4) -- ++(2,0) --++(0,1) --++(1,0) -- ++(0,6);
                        \draw[blue, line width=0.5mm] (1,6) -- ++(1,0) -- ++(0,1) -- ++(1,0) -- ++(0,4)
                                -- ++(3,0) -- ++(0,5) -- ++(3,0) --++(0,1) --++(1,0)
                                -- ++(0,4) -- ++(6,0);
        \end{tikzpicture}
        \end{subfigure}
        \caption{\label{fig:base_cell} Left and right: $B$-Base cell
                $Q$ with side length $\ell=16$ (see \cref{def:base_cell}). Left: Notation as
                introduced in \cref{def:base_cell}. Right: Base cell
                $Q$ with cross as in \cref{def:renorm_crossing} with
                horizontal interior crossing in blue and vertical
                interior crossing in red. The colours only
                serve to better distinguish the horizontal from
                the vertical path.}
\end{figure}
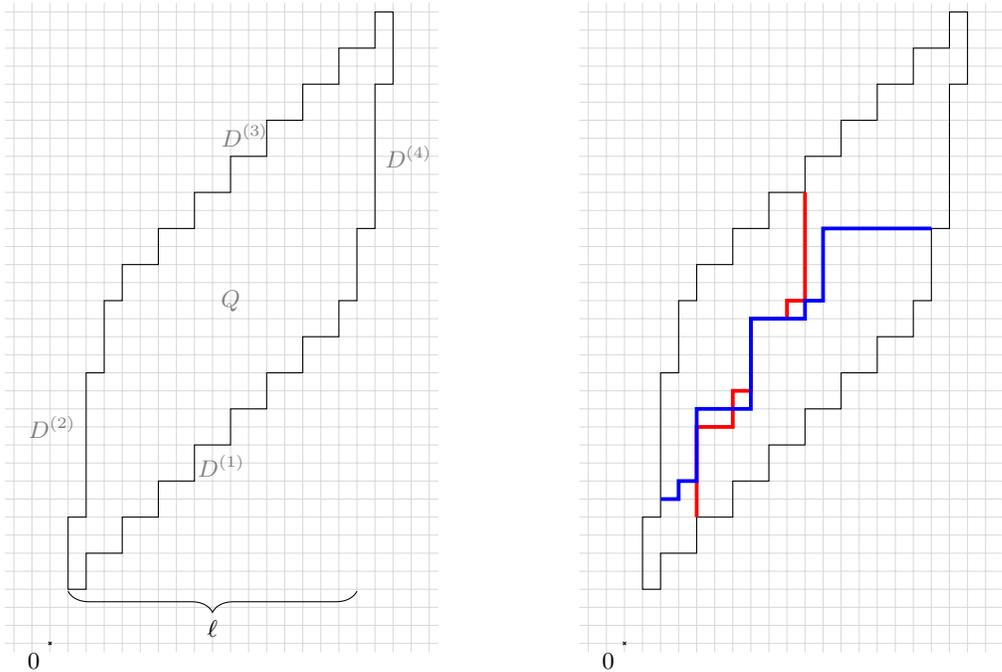
For the rest of this section fix a side length $\ell$. In this base cell we
define the notion of interior crossing paths in the horizontal and vertical
direction.
\begin{definition}[Interior $B$-crossings and cross]\label{def:renorm_crossing}
        Let $Q$ be the $B$-base cell.
        We say that a $B$-path $(x^{(1)},\ldots, x^{(n)})\subset Q$ is a
        \emph{vertical interior $B$-crossing for $Q$} if $x^{(1)}\in D^{(1)}$,
        $x^{(n)}\in
        D^{(3)}$ and $x^{(i)}\not\in\bigcup_{i\in [4]} D^{(i)}$ for $i\in
        [2,n-1]$. Similarly, we say that it is a \emph{horizontal interior
        $B$-crossing} if $x^{(1)}\in
        D^{(2)}$, $x^{(n)}\in D^{(4)}$ and $x^{(i)}\not\in \bigcup_{i\in [4]}
        D^{(i)}$ for $i\in [2,n-1]$ (see right side of \cref{fig:base_cell}).
        We call a pair $\mathcal{C}_0=(\mathcal{C}_0^{(v)}, \mathcal{C}_0^{(h)})$ of
        a vertical interior crossing and horizontal interior crossing of
        $Q$ a \emph{cross in $Q$}.
\end{definition}
We translate the cell $Q$ to construct larger square grids of cells.
\begin{figure}[ht]
        \begin{subfigure}[b]{0.475\textwidth}
        \centering
        \begin{tikzpicture}[scale=0.16]
                        \draw[step=1.0, opacity=0.1](-2.5,-3.5) grid (28.5,51.5);

                        \draw(-1,-1) node[anchor=north east] {\footnotesize $0$};
                        \draw(-1,-1) node[cross, black]{};

                        \foreach \i in {0, ..., 2}{
                                \foreach \j in {0, ..., 2}{
                                        \draw[]
                                                (8*\i+\j,8*\i+8*\j+2)
                                                 -- ++(1,0) -- ++(0,2)
                                                 -- ++(2,0) -- ++(0,2)
                                                 -- ++(2,0) -- ++(0,2)
                                                 -- ++(2,0) -- ++(0,2)
                                                 -- ++(1,0)
                                                 -- ++(0,4) -- ++(1,0)
                                                 -- ++(0,4)
                                                 -- ++(-1,0) -- ++(0,-2)
                                                 -- ++(-2,0) -- ++(0,-2)
                                                 -- ++(-2,0) -- ++(0,-2)
                                                 -- ++(-2,0) -- ++(0,-2)
                                                 -- ++(-1,0)
                                                 -- ++(0,-4) -- ++(-1,0)
                                                 -- cycle;
                                         \draw(8*\i+\j+4.5,8*\i+8*\j+9.5) node[opacity=0.5]
                                                {\tiny $Q_{\i,\j}$};
                                }
                        }

                        \foreach \i in {1}{
                                \foreach \j in {0, ..., 2}{
                                        \draw[fill, opacity=0.2] 
                                                (8*\i+\j,8*\i+8*\j+2)
                                                 -- ++(1,0) -- ++(0,2)
                                                 -- ++(2,0) -- ++(0,2)
                                                 -- ++(2,0) -- ++(0,2)
                                                 -- ++(2,0) -- ++(0,2)
                                                 -- ++(1,0)
                                                 -- ++(0,4) -- ++(1,0)
                                                 -- ++(0,4)
                                                 -- ++(-1,0) -- ++(0,-2)
                                                 -- ++(-2,0) -- ++(0,-2)
                                                 -- ++(-2,0) -- ++(0,-2)
                                                 -- ++(-2,0) -- ++(0,-2)
                                                 -- ++(-1,0)
                                                 -- ++(0,-4) -- ++(-1,0)
                                                 -- cycle;
                                }
                        }

                        \foreach \i in {0, ..., 2}{
                                \foreach \j in {2}{
                                        \draw[fill, opacity=0.2]
                                                (8*\i+\j,8*\i+8*\j+2)
                                                 -- ++(1,0) -- ++(0,2)
                                                 -- ++(2,0) -- ++(0,2)
                                                 -- ++(2,0) -- ++(0,2)
                                                 -- ++(2,0) -- ++(0,2)
                                                 -- ++(1,0)
                                                 -- ++(0,4) -- ++(1,0)
                                                 -- ++(0,4)
                                                 -- ++(-1,0) -- ++(0,-2)
                                                 -- ++(-2,0) -- ++(0,-2)
                                                 -- ++(-2,0) -- ++(0,-2)
                                                 -- ++(-2,0) -- ++(0,-2)
                                                 -- ++(-1,0)
                                                 -- ++(0,-4) -- ++(-1,0)
                                                 -- cycle;
                                }
                        }
        \end{tikzpicture}
        \end{subfigure}
        \hfill
        \begin{subfigure}[b]{0.475\textwidth}
        \centering
        \begin{tikzpicture}[scale=0.16]
                        \draw[step=1.0, opacity=0.1](-2.5,-3.5) grid (28.5,51.5);

                        \draw(-1,-1) node[anchor=north east] {\footnotesize $0$};
                        \draw(-1,-1) node[cross, black]{};

                        \foreach \i in {0, ..., 2}{
                                \foreach \j in {0, ..., 2}{
                                        \draw[]
                                                (8*\i+\j,8*\i+8*\j+2)
                                                 -- ++(1,0) -- ++(0,2)
                                                 -- ++(2,0) -- ++(0,2)
                                                 -- ++(2,0) -- ++(0,2)
                                                 -- ++(2,0) -- ++(0,2)
                                                 -- ++(1,0)
                                                 -- ++(0,4) -- ++(1,0)
                                                 -- ++(0,4)
                                                 -- ++(-1,0) -- ++(0,-2)
                                                 -- ++(-2,0) -- ++(0,-2)
                                                 -- ++(-2,0) -- ++(0,-2)
                                                 -- ++(-2,0) -- ++(0,-2)
                                                 -- ++(-1,0)
                                                 -- ++(0,-4) -- ++(-1,0)
                                                 -- cycle;
                                }
                        }

                        \draw[blue, line width=0.4mm] (0, 5) -- ++(2, 0) -- ++(0,4) -- ++(3,0)
                                -- ++(0,3) -- ++(2,0) -- ++(0,1) -- ++(3, 0) -- ++(0,2) -- ++(1,0)
                                -- ++(0,4) -- ++(4, 0) -- ++(0,4) -- ++(3,0)
                                -- ++(0,1) -- ++(1, 0) -- ++(0,3) -- ++(5,0);

                        \draw[blue, line width=0.4mm] (2, 15) -- ++(2,0)
                                -- ++(0,1) -- ++(1,0) -- ++(0,1) -- ++(1,0)
                                -- ++(0,3) -- ++(1,0) -- ++(0,1) -- ++(4,0)
                                -- ++(0,3) -- ++(2,0) -- ++(0,3) -- ++(2,0)
                                -- ++(0,2) -- ++(4,0) -- ++(0,3) -- ++(2,0)
                                -- ++(0,2) -- ++(2,0) -- ++(0,1) -- ++(2,0);

                        \draw[blue, line width=0.4mm] (2, 21) -- ++(2,0)
                                -- ++(0,3) -- ++(1,0) -- ++(0,1) -- ++(1,0)
                                -- ++(0,1) -- ++(1,0) -- ++(0,3) -- ++(5,0)
                                -- ++(0,2) -- ++(2,0) -- ++(0,2) -- ++(1,0)
                                -- ++(0,4) -- ++(5,0) -- ++(0,3) -- ++(2,0)
                                -- ++(0,3) -- ++(4,0);

                        \draw[red, line width=0.4mm] (4,6) -- ++(0,5) -- ++(1,0)
                                -- ++(0,8) -- ++(1,0) -- ++(0,6) -- ++(1,0) -- ++(0,1)
                                -- ++(1,0) -- ++(0,4);

                        \draw[red, line width=0.4mm] (10,12) -- ++(0,4) -- ++(2,0) -- ++(0,2)
                                -- ++(1,0) -- ++(0,9) -- ++(3,0) -- ++(0,11);
                        
                        \draw[red, line width=0.4mm] (18,20) -- ++(0,4) -- ++(1,0)
                                -- ++(0,8) -- ++(1,0) -- ++(0,10);

                        \filldraw[](5,12) circle[radius=10pt];
                        \filldraw[](6,20) circle[radius=10pt];
                        \filldraw[](8,29) circle[radius=10pt];
                        \filldraw[](13,19) circle[radius=10pt];
                        \filldraw[](16,29) circle[radius=10pt];
                        \filldraw[](19,27) circle[radius=10pt];
                        \filldraw[](20,32) circle[radius=10pt];
                        \filldraw[](16,37) circle[radius=10pt];
                        \filldraw[](20,40) circle[radius=10pt];
        \end{tikzpicture}
        \end{subfigure}
        \caption{\label{fig:asym_grid} Left: $Q_{i,j}$ for $i\in [0, 2]$ and
                $j\in [0,2]$, side length $\ell=8$. The vertical strip $Q^{(v)}_1$ and
                the horizontal strip $Q^{(h)}_2$ are shaded in gray. Right: A grid
                $\mathcal{C}$ with $N=2$. The hard vertical interior crossings
                are red and hard horizontal interior crossings are blue and the
                intersections points $X(\mathcal{C})$ are black.}
\end{figure}
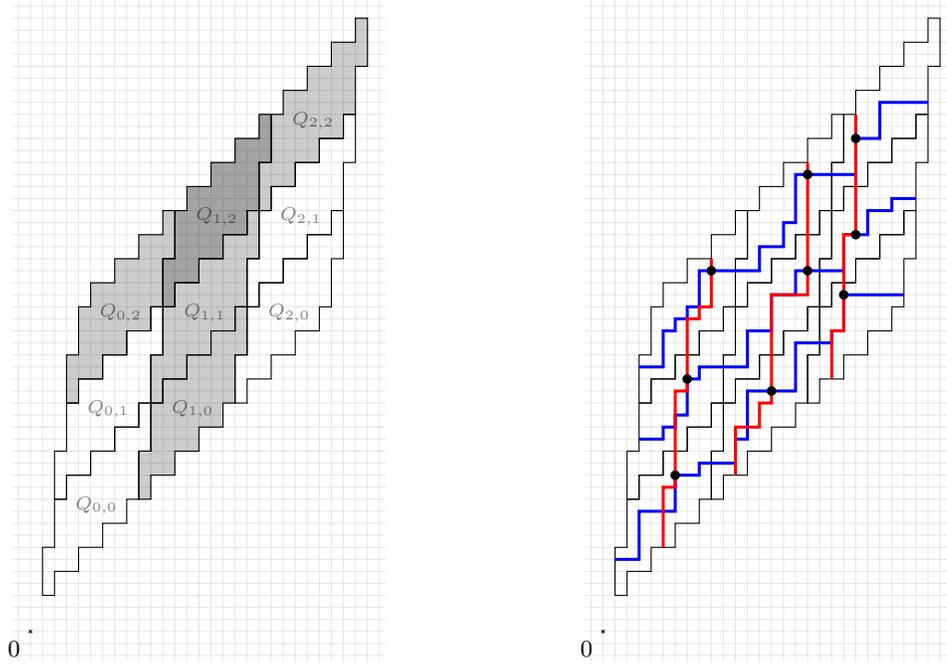
\begin{definition}[$Q_{i,j}$]
        Let $\mathbf{b}_1 =
        \ell(\mathbf{e}_1+\mathbf{e}_2)$ and $\mathbf{b}_2 =
        (\ell/8)\mathbf{e}_1 + \ell \mathbf{e}_2$. For $i,j\in \Z$ we then let
        $Q_{i,j}=Q_0+i \mathbf{b}_1+ j \mathbf{b}_2$. Given a \emph{square side
        length $N\in \N$} we define the \emph{rectangle of grids
        $\mathcal{Q}^{(B)}$} as
        \begin{equation}
                \mathcal{Q}^{(B)} =
                \bigcup_{(i,j)\in{[0,N]}^2}
                Q_{i,j}\;.
        \end{equation}
\end{definition}
\begin{remark}
        Notice that $Q_{0,0}=Q$ and that neighbouring cells share a boundary.
\end{remark}
In what follows consider the square side length $N\in
\N$ fixed. On sets of neighbouring cells we introduce a notion of hard
interior $B$-crossing, as opposed to the local one which only dealt with paths in
one cell.
\begin{definition}[$B$-strips and hard interior $B$-crossing]\label{def:hard_int_cros}
        For $i\in [0, N]$ we call the set of cells
        \begin{equation}
                Q^{(v)}_i=\bigcup_{j\in [0, N]} Q_{i,j}
        \end{equation}
        the \emph{$i$-th vertical $B$-strip} and for $j\in [0,N]$ we
        define the \emph{$j$-th horizontal $B$-strip} as
        \begin{equation}
                Q^{(h)}_j=\bigcup_{i\in [0, N]} Q_{i,j}\;.
        \end{equation}
        A $B$-path $\Gamma\subset Q_i$ is a \emph{hard vertical interior
        $B$-crossing of $Q_i$} if $\Gamma\cap Q_{i,j}$ is a vertical
        interior $B$-crossing of $Q_{i,j}$ for any $j\in [0,N]$. Analogously
        for \emph{hard horizontal interior $B$-crossings} (see
        \cref{fig:asym_grid}).
\end{definition}
The set of hard interior crossings induce a grid $\mathcal{C}$.
\begin{definition}[$B$-grids]\label{def:grid}
        For $i\in [0,N]$ let $\mathcal{C}_i^{(v)}$ be a hard vertical interior
        crossing for the $i$-th vertical $B$-strip and for $j\in [0,N]$ let
        $\mathcal{C}_j^{(h)}$be a hard horizontal interior crossing of the
        $j$-th horizontal strip. We call $\mathcal{C}=(\mathcal{C}^{(v)},
        \mathcal{C}^{(h)})$ a \emph{$B$-grid of $\mathcal{Q}^{(B)}$} where
        $\mathcal{C}^{(v/h)} ={\{\mathcal{C}_i^{(v/h)}\}}_{i\in [0,N]}$.
        Given a grid $\mathcal{C}$ of $\mathcal{Q}^{(B)}$ we call
        $\mathcal{C}_{i,j}=(\mathcal{C}_{i}^{(v)}, \mathcal{C}_{j}^{(h)})$ the
        cross induced in $Q_{i,j}$.
\end{definition}
The intersection points of the induced crosses in each $Q_{i,j}$ form a set
that is isomorphic to an equilateral box in $\Z^2$.
\begin{definition}[Intersection points associated to grid]%
        \label{def:intersect_points}
        Given a $B$-grid $\mathcal{C}$ of $\mathcal{Q}^{(B)}$ we denote by
        $x_{i,j}$ the highest point in $\mathcal{C}_i^{(v)}\cap
        \mathcal{C}_j^{(h)}$ in the $\prec^{(B)}$-partial order\footnote{The
                $\prec^{(B)}$-partial order corresponds to the usual order
                where $x\prec^{(B)}y$ if $x_i\le y_i$ for all $i\in [d]$, we
                write $\prec^{(B)}$ to make it easier to generalise to $A$- and
                $C$-grids.} and call it an
        intersection point of $\mathcal{C}$. We write $X(\mathcal{C})$ for the set of
        intersection points. We call $x_{i,j}$ and $x_{i',j'}$ \emph{neighbours
        in $X(\mathcal{C})$} if $(i,j)$ and $(i',j')$ are neighbours in
        ${[0,N]}^2$. Analogously we call $x_{i',j'}$ an \emph{oriented
        neighbour of $x_{i,j}$ in $X(\mathcal{C})$} if $x_{i',j'}$ and
        $x_{i,j}$ are neighbours in $X(\mathcal{C})$ such that
        $(i',j')\prec^{(B)} (i,j)$. We call $x_{i+1,j}$ (if it exists)
        the \emph{east neighbour of $x_{i,j}$ in $X(\mathcal{C})$} and
        $x_{i,j+1}$ (if it exists) the \emph{north neighbour of
        $x_{i,j}$ in $X(\mathcal{C})$} and analogously for the south
        and west neighbours.
\end{definition}
\begin{remark}
        The $\prec^{(B)}$-ordering is only partial but since we look at intersection
        points of $B$-paths there is always a unique highest point on
        $\mathcal{C}_i^{(v)}\cap \mathcal{C}_j^{(h)}$, and since $\mathcal{C}$
        induces a cross in each $Q_{i,j}$, $x_{i,j}$ is well defined for any
        $i,j\in [0,N]$.
\end{remark}
The $A$-base cell is defined analogously by exchanging the role of
$\mathbf{e}_1$ with $-\mathbf{e}_2$ and $\mathbf{e}_2$ with $-\mathbf{e}_1$ and
for the $C$-base cell exchange $\mathbf{e}_1$ with $-\mathbf{e}_1$. Do the
analogous exchanges in the following definitions for $h$-crossings and
$h$-grids. When changing the base vectors like this the horizontal $A$-crossing
would cross the base cell vertically so change the names appropriately.

Further, this construction can be translated to be based at any $x\in \Z^2$ by
replacing the origin in the definitions with $x$. We will denote this as an
explicit argument so $Q_{i,j}(x):=Q_{i,j}+x$. Since by translation
invariance we can apply the results for the origin to any $x\in \Z^2$ this
notation is rarely used.

Consider the set $V_0$ given by the points $x\in \Z^2$ such that
$-x_1+x_2\le 0$ (i.e.\ the set that is `below' the main diagonal going through
the origin) and for $n\in \Z$ let $V_n=V_0+(-n,n)$, then ${\{V_n\}}_{n\in
\Z}$ is an increasing and exhausting set of $\Z^2$. The following Lemma is the
principal reason to construct the $h$-grid as we did.
\begin{lemma}\label{lemma:grid_exterior}
        Let $\mathcal{A}_x$ be an event with support in $\cup_{h\in [A,B,C]}
        \mathcal{Q}^{(h)}_x$, then the family ${\{\mathcal{A}\}}_{x\in \Z^2}$
        satisfies the exterior condition w.r.t.\ ${\{V_n\}}_{n\in \Z}$.
\end{lemma}
\begin{proof}
        Follows from the construction of the grids.
\end{proof}

\subsubsection{Crossing probabilities and grid relaxation}%
\label{sub:renorm_probability}
Let $\mathbf{q}$ be a parameter set for the $ABC$-model and set $\ell=\lceil
\theta_B^{3/2}\rceil$, $N=2^{\lceil \theta_B/2+\log_2(\theta_B)\rceil}$ as the
parameters for any base cells and grids. The goal for this section is to define
an event so that we can use the exterior condition theorem,
\cref{thm:exterior_thm}.

We say that a set $\Lambda$ is \emph{$B$-traversable} if it does not contain
$A$ or $C$ vacancies, and we define correspondingly $A$- and
$C$-traversability. The event for which we want to apply the exterior condition
theorem will require the existence of an appropriately traversable grid
$\mathcal{C}$ for each vacancy type so let us upper bound the probability of
not finding $B$-traversable $B$-crossings as a first step.
\begin{lemma}\label{lemma:good_strip}
        Let $\mathcal{A}$ be the event of finding a $B$-traversable hard
        interior $B$-crossing in a strip $Q$. If $\mathbf{q}$ is such that $q_A+q_C\rightarrow
        0$ as $q_B\rightarrow 0$ then we find a constant $C>0$ so that
        \begin{equation}
                \mu(\mathcal{A}^c) \le 2^{-\theta_B^{3/2}(1+o(1)}
        \end{equation}
        for $q_B$ small enough.
\end{lemma}
\begin{proof}
        We follow the arguments from~\cite{martinelli2020diffusive} to apply a
        Peierls-type argument. We will deal with the vertical case first, the
        horizontal one being analogous. Consider a vertical strip $Q^{(v)}_i$ with left
        boundary $D^{(2)}$ and right boundary $D^{(4)}$. Define on it the dual
        graph $Q^*_i$ as the faces of $Q^{(v)}_i$, i.e.\ the graph given by
        \begin{equation}
                Q^*_i=\{x^*\in Q^{(v)}_i+1/2(\pm\mathbf{e}_1\pm\mathbf{e}_2)\colon \|
                \{x\in Q^{(v)}_i\colon \|x^*-x\|_1=1\}\|= 4\}, 
        \end{equation}
        with neighbourhood relations induced by $\Z^2+1/2
        (\mathbf{e}_1+\mathbf{e}_2)$. We define the left boundary $D^{(2,*)}$ as the set of
        $x^*\in Q^*_i$ for which there exists an $x\in D^{(2)}$ such that
        $\|x^*-x\|_1=1$ and analogously for the right boundary $D^{(4,*)}$ with $D^{(4)}$.
        Say that the horizontal directed edge $(x^*,x^*+\mathbf{e}_1)$ in
        $Q^*_i$ is \emph{closed} in a configuration $\omega\in \Omega$ if
        $x^*+1/2(\mathbf{e}_1+\mathbf{e}_2)$ (north-east corner) is not $B$-traversable, i.e.\ has
        an $A$- or $C$-vacancy and \emph{open} otherwise. Similarly for the
        vertical edge $(x^*,x^*+\mathbf{e}_2)$ with vertex
        $x^*+1/2(\mathbf{e}_1-\mathbf{e}_2)$ (south-east corner). For convenience call
        all other directed edges in $Q_i^*$ \emph{closed}. We call a dual path in
        $Q^*_i$ connecting $D^{(2,*)}$ to $D^{(4,*)}$ \emph{closed} iff all
        its edges are closed.

        For $\omega\in \Omega$ consider the set $W$ of vertices in
        $Q^{(v)}_i\setminus (D^{(2)}\cup D^{(4)})$ that are reachable by a
        $B$-traversable $B$-path (recall: up-right path) starting at
        $D^{(1)}\setminus (D^{(2)}\cup D^{(4)})$ and let the \emph{contour} be
        the set of faces $x^*$ that have a vertex inside and a vertex outside
        of $\{W\cup D^{(1)}\setminus(D^{(2)}\cup D^{(4)})\}$ incident to them.
        Not finding a $B$-traversable hard interior $B$-crossing on $Q^{(v)}_i$ then,
        by construction, implies that the contour is a closed dual path in
        $Q_i^*$ connecting $D^{(2,*)}$ to $D^{(4,*)}$ (see \cref{fig:blocking_paths}).

        For a fixed $\omega\in \Omega$ let $\Gamma=\Gamma(\omega)$ be a closed
        non-backtracking dual path connecting the left to the right boundary
        and $n_n$, $n_e$, $n_s$, $n_w$ be the amount of north, east, south and
        west steps in it respectively. $\Gamma$ being closed then implies the
        existence of at least $(n_e+n_s)/2$ $A$- or $C$-vacancies, only half
        since if an east step follows a south step they have the same
        associated vertex. Further note that by construction of $Q^{(v)}_i$ every
        eighth step north an additional step east or south has to be made to
        reach the right boundary while any step west immediately implies
        another step east. So, $\Gamma$ being closed implies the existence
        $\Theta(|\Gamma|)$ $A$- or $C$-vacancies\footnote{Note, we are not
                saying that there are only $\Theta(|\Gamma|)$ vacancies, but
        that the directly implied amount is of this order}. Let $\Pi_{x^*}$ be the set of
        dual paths starting at $x^*\in D^{(2,*)}$ and ending at $D^{(4,*)}$.
        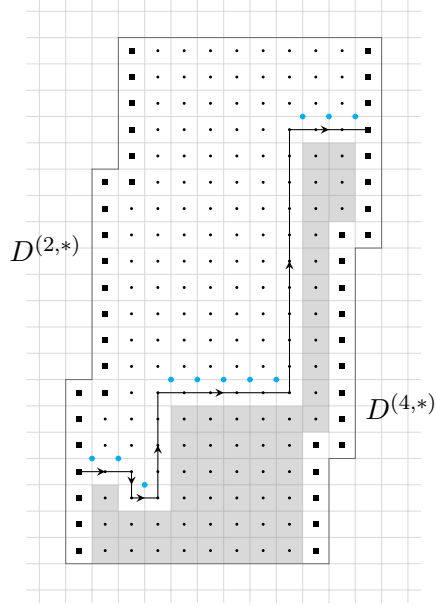
\begin{figure}
                \centering
        \begin{tikzpicture}[scale=0.35]
                        \tikzset{
                          on each segment/.style={
                            decorate,
                            decoration={
                              show path construction,
                              moveto code={},
                              lineto code={
                                \path [#1]
                                (\tikzinputsegmentfirst) -- (\tikzinputsegmentlast);
                              },
                              curveto code={
                                \path [#1] (\tikzinputsegmentfirst)
                                .. controls
                                (\tikzinputsegmentsupporta) and (\tikzinputsegmentsupportb)
                                ..
                                (\tikzinputsegmentlast);
                              },
                              closepath code={
                                \path [#1]
                                (\tikzinputsegmentfirst) -- (\tikzinputsegmentlast);
                              },
                            },
                          },
                          mid arrow/.style={postaction={decorate,decoration={
                                markings,
                                mark=at position .5 with {\arrow[#1]{stealth}}
                              }}},
                        }
                        \draw[step=1, opacity=0.15](-0.5,-1.5) grid (14.5,21.5);

                        \draw[opacity=0.5]
                                (1,0) -- ++(10,0) --
                                ++(0,4) -- ++(1,0) -- ++(0,8) -- ++(1,0) --
                                ++(0,8) -- ++(-10,0) -- ++(0,-5) -- ++(-1,0) --
                                ++(0,-8) -- ++(-1,0) -- ++(0,-7);

                        \fill[opacity=0.3, gray] (2,0) -- ++(0,3) -- ++(1,0) --
                                ++(0,-1) -- ++(2,0) -- ++(0,4) -- ++(5,0) -- ++(0,10)
                                -- ++(2,0) -- ++(0,-3) -- ++(-1,0) -- ++(0,-8) --
                                ++(-1,0) -- ++(0,-5) -- cycle;

                        \foreach \x in {0,...,6}{
                                \fill[black] (1+0.4,\x+0.4) rectangle
                                        (1+0.6,\x+0.6);
                        }
                        \foreach \x in {0,...,8}{
                                \fill[black] (2+0.4,6+\x+0.4) rectangle
                                        (2+0.6,6+\x+0.6);
                        }

                        \foreach \x in {0,...,5}{
                                \fill[black] (3+0.4,14+\x+0.4) rectangle
                                        (3+0.6,14+\x+0.6);
                        }

                        \foreach \x in {0,...,4}{
                                \fill[black] (10+0.4,\x+0.4) rectangle
                                        (10+0.6,\x+0.6);
                        }
                        \foreach \x in {0,...,8}{
                                \fill[black] (11+0.4,4+\x+0.4) rectangle
                                        (11+0.6,4+\x+0.6);
                        }

                        \foreach \x in {0,...,7}{
                                \fill[black] (12+0.4,12+\x+0.4) rectangle
                                        (12+0.6,12+\x+0.6);
                        }

                        \foreach \horizontal in {0,...,9}{
                                \foreach \one in {0,...,6}{
                                        \filldraw[black](\horizontal+1+0.5, \one+0.5) circle[radius=1pt];
                                }
                                \foreach \two in {0,...,7}{
                                        \filldraw[black](\horizontal+2+0.5, \two+7+0.5) circle[radius=1pt];

                                }
                                \foreach \three in {0,...,4}{
                                        \filldraw[black](\horizontal+3+0.5, \three+15+0.5) circle[radius=1pt];
                                }

                        }

                        \foreach \x in {0,1}{
                                \foreach \y in {0,...,2}{
                                        \filldraw[black](11+\x+0.5, \y+8*\x+4+0.5) circle[radius=1pt];
                                }
                        }

                        \path[draw=black, postaction={on each segment={mid arrow=black}}] (1.5,3.5) -- ++(2,0) -- ++(0,-1) -- ++(1,0) -- ++(0,4)
                                -- ++(5,0) -- ++(0,10) -- ++(3,0);
                        \filldraw[Cyan](2,4) circle[radius=2.5pt];
                        \filldraw[Cyan](3,4) circle[radius=2.5pt];
                        \filldraw[Cyan](4,3) circle[radius=2.5pt];
                        \filldraw[Cyan](5,7) circle[radius=2.5pt];
                        \filldraw[Cyan](6,7) circle[radius=2.5pt];
                        \filldraw[Cyan](7,7) circle[radius=2.5pt];
                        \filldraw[Cyan](8,7) circle[radius=2.5pt];
                        \filldraw[Cyan](9,7) circle[radius=2.5pt];
                        \filldraw[Cyan](10,17) circle[radius=2.5pt];
                        \filldraw[Cyan](11,17) circle[radius=2.5pt];
                        \filldraw[Cyan](12,17) circle[radius=2.5pt];

                        \node[anchor=east] at (2,12) {$D^{(2,*)}$};
                        \node[anchor=west] at (12,6) {$D^{(4,*)}$};
        \end{tikzpicture}

                \caption{\label{fig:blocking_paths} Example for a closed dual path
                together with $W$ shaded in grey. The implied $A$- or
                $C$-vacancies are in blue.}
        \end{figure}
        We then have for some constants $\kappa, C$,
        \begin{align}
                \mu(\mathcal{A}^c)
                &\le \sum_{x^*\in D^{(2,*)}}\sum_{\Gamma\in \Pi_{x^*}}
                \mu(\Gamma\ \text{is closed})\\
                &\le \sum_{x^*\in D^{(2,*)}}\sum_{\Gamma\in \Pi_{x^*}}
                {(q_A+q_C)}^{\Theta(|\Gamma|)}\\
                &\le \sum_{x^*\in D^{(2,*)}}\sum_{k=\kappa\ell }^{\infty}
                3^k {(q_A+q_C)}^{\Theta(k)}\\
                &\le C N\ell 2^{-\ell}
        \end{align}
        where we chose $q_B$ small enough and use that $q_A+q_C\rightarrow 0$
        as $q_B\rightarrow 0$. The proof for horizontal strips is analogous and
        the claim follows.
\end{proof}
With this we can calculate the failing probability of finding a $B$-traversable
grid is a simple union bound.
\begin{corollary}\label{cor:grid_failing_1}
        Let $\mathcal{E}^{(B,1)}$ be the event that 
        \begin{itemize}
                \item there is a $B$-traversable $B$-grid,
                \item there is an intersection point $x_{i,j}$ in the above grid with
                        $i,j> N/2$ such that there exists $\mathbf{e}\in
                        \mathcal{B}=\{(0,1),(1,0)\}$ with $\omega_{x_{(i,j)+\mathbf{e}}}=B$,
        \end{itemize}
        Then, for parameter sets such that $(q_A+q_C)\rightarrow 0$ as
        $q_B\rightarrow 0$ we have
        \begin{equation}
                \lim_{q_B\rightarrow
                0}|\mathcal{Q}^{(B)}|\mu(1-\mathds{1}_{\mathcal{E}^{(B,1)}}) =0.
        \end{equation}
\end{corollary}
This gives us a $B$ vacancy on an intersection point and the necessary
$B$-traversable paths to bring it into $Q_{0,0}$. The intersection point
$x_{0,0}$ is still random though so we require another set of $B$-traversable
paths to bring the $B$-vacancy to a deterministic point. The following result
gives this with another set of simple estimates.
\begin{lemma}\label{lemma:grid_failing_2}
        Let $\mathcal{E}^{(B,2)}$ be the event that the boundary
        $D_{0,0}^{(1)}$ is $B$-traversable. Then we have for parameter sets
        such that $\ell^2(q_A+q_C)\rightarrow 0$ as $q_B\rightarrow 0$ that
        \begin{equation}
                \lim_{q_B\rightarrow 0}\mathrm{Supp}(\mathcal{E}^{(B,2)})
                \mu(1-\mathds{1}_{\mathcal{E}^{(B,2)}})=0.
        \end{equation}
\end{lemma}
To show that $\mathcal{E}^{(B)}=\mathcal{E}^{(B,1)}\cap \mathcal{E}^{(B,2)}$
allows us to find an inequality like \cref{eqn:particle_poincare} we need to
introduce another tool.
\begin{lemma}[Extending the variance]\label{lemma:extend_variance}
        Let $\mathcal{A}=\mathcal{A}_1\cap \mathcal{A}_2\cap \mathcal{A}_3$
        be an event on $\Omega$, let $V_i:=\mathrm{Supp}(\mathcal{A}_i)$ for
        $i\in [3]$ and assume that $V_i\cap V_j=\emptyset$ for any pair $i\neq
        j$. Then, for any $f\in L^2(\mu)$ and for the conditional variance
        $\var_x(f\tc \mathcal{A})=\mu_x(f^2\tc \mathcal{A})-{(\mu_x(f\tc
        \mathcal{A}))}^2$ we find
        \begin{equation}
                \mu(\mathds{1}_{\mathcal{A}}\var_{V_1}(f\tc \mathcal{A}))
                \le \mu(\mathds{1}_{\mathcal{A}}\var_{V}(f\tc \mathcal{A}))
        \end{equation}
        for $V=V_1\cup V_2$.
\end{lemma}
\begin{remark}
        The usual use case is that we have an event $\mathcal{A}$ with a large
        support that we split into two smaller events $\mathcal{A}_1$,
        $\mathcal{A}_2$ and the `rest' $\mathcal{A}_3$ which is why $V$ only
        contains $V_1$ and $V_2$.
\end{remark}
\begin{proof}
        Write $\mathcal{A}'=\mathcal{A}_1\cap \mathcal{A}_3$ and calculate
        directly
        \begin{align}
                \mu(\mathds{1}_{\mathcal{A}}\var_{V_1}(f\tc \mathcal{A}))
                &=
                \mu_{V_2}(\mathcal{A}_2)\mu_{V_2^c}\left[\mathds{1}_{\mathcal{A}'}
                        \mu_{V_2}(\mu_{V_1}(f^2\tc\mathcal{A})-{(\mu_{V_1}(f\tc
                        \mathcal{A}))}^2\tc \mathcal{A}_2)\right]\\
                &\le
                \mu_{V_2}(\mathcal{A}_2)\mu_{V_2^c}\left[\mathds{1}_{\mathcal{A}'}
                        \left(\mu_{V}(f^2\tc\mathcal{A})-{(\mu_{V}(f\tc
                        \mathcal{A}))}^2\right)\right]\\
                &= \mu\left[\mathds{1}_{\mathcal{A}}
                        \var_V(f\tc \mathcal{A})\right],
        \end{align}
        where in the first inequality we used Jensen's inequality and in the
        last equality we used that $\var_V(f\tc \mathcal{A})$ does not depend
        on spins in $V_2$ anymore.
\end{proof}
Any configuration in $\mathcal{E}^{(B)}$ potentially contains many conforming
$B$-grids so let us introduce a partial order on them. Let
$\Gamma=(x^{(1)},\ldots, x^{(n)})$ and $\Gamma'=(y^{(1)},\ldots, y^{(m)})$ be
two hard interior crossings of the same strip. If there is no crossing point
say that $\Gamma$ is smaller than $\Gamma'$ if $x^{(1)}\prec y^{(1)}$. If they
cross in a single point $x^{(i)}=y^{(j)}$ and $x^{(i+1)}\prec y^{(j+1)}$ then
we say that $\Gamma$ is smaller than $\Gamma'$.

This generalises to a partial order on any family of hard interior crossings of
the same strip with multiple crossing points if the above condition is
fulfilled after \emph{every} crossing point. Note that this is only a partial
order but there is a unique smallest crossing. For $\omega\in
\mathcal{E}^{(B)}$ we write $\mathcal{G}(\omega)$ for the $B$-grid with the
smallest crossings in each strip conforming to $\mathcal{E}^{(B)}$.

A final remark about notation: We will
write $\mu^{(h)}(\cdot):=\mu(\cdot \tc \{\star, h\})$ and
$\var^{(h)}(\cdot):=\var(\cdot\tc \{\star, h\})$ for the measure resp.\
variance conditioned to be in the state space $\{\star, h\}$. Recall further that
$\mathcal{Q}^{(B)}$ is the grid of base cells with parameters $N,\ell$ of which
the smallest vertex in the $\prec$-partial order (i.e.\ the closest vertex to
the origin) is $z_B:=\mathbf{e}_1+3\mathbf{e}_2$.
\begin{lemma}\label{lemma:grid_relaxation}
        Let $f\in L^2(\mu)$. For any $\varepsilon>0$ we find a $q(\varepsilon)>0$ such that
        \begin{equation}
                \mu_{\mathcal{Q}^{(B)}}
                (\mathds{1}_{\mathcal{E}^{(B)}}
                \var^{(B)}_{z_B}(f))
                \le 2^{\theta_B^2(1+\varepsilon)/4}
                \sum_{y\in \mathcal{Q}^{(B)}}
                \mu_{\mathcal{Q}^{(B)}}
                \left[c_y^{B}q_B p{(\nabla_y^{(B)}f)}^2\right],
        \end{equation}
        for $q_B<q(\varepsilon)$.
\end{lemma}
\begin{proof}
        For simplicity we write $\mu_{\mathcal{Q}^{(B)}}=\mu$ in this proof.
        There might be many intersection points $x_{i,j}\in
        X(\mathcal{G}(\omega))$ such that $\omega_{x_{(i,j)+\mathbf{e}}}=B$ for
        some $e\in \mathcal{B}$. Introduce the constraint
        $c^{(\mathcal{G})}_{x_{i,j}}$ that there is an $\mathbf{e}\in
        \mathcal{B}$ such that $\omega_{x_{(i,j)+\mathbf{e}}}=B$ and denote by
        $\xi(\omega)$ the vertex in $X(\mathcal{G}(\omega))$ with the highest
        coordinate in the lexicographic order such that
        $c^{(\mathcal{G})}_{\xi(\omega)}=1$. Since this uniquely identifies a grid and
        an intersection point we have
        \begin{equation}
                \mu\left[\mathds{1}_{\mathcal{E}^{(B)}}
                \var^{(B)}_{z_B}(f)\right]
                   = \sum_{\mathcal{C}\ \text{grid}}
                   \sum_{x\in X(\mathcal{C})} 
                   \mu\left[\mathds{1}_{\mathcal{G}=\mathcal{C},\xi=x,
                           \mathcal{E}^{(B,2)}}
                   \var^{(B)}_{z_B}(f)\right].
        \end{equation}
        Let us upper bound a generic summand
        $\mu\left[\mathds{1}_{\mathcal{G}=\mathcal{C},\xi=x,
        \mathcal{E}^{(B,2)}} \var^{(B)}_{z_B}(f)\right]$ and assume without
        loss of generality that $x=x_{N,N-1}$.
        Let $V\subset X(\mathcal{C})$ be a subset with $x_{0,0},x_{N,N-1}\in V$. The event
        $\mathcal{G}=\mathcal{C}$ on $V\cap z_B$ reduces to requiring
        $B$-traversability so that we can extend the variance
        (\cref{lemma:extend_variance})
        \begin{equation}
                \mu\left[\mathds{1}_{\mathcal{G}=\mathcal{C},\xi=x,
                \mathcal{E}^{(B,2)}} \var^{(B)}_{z_B}(f)\right]
                = \mu\left[\mathds{1}_{\mathcal{G}=\mathcal{C},\xi=x,
                \mathcal{E}^{(B,2)}} \var^{(B)}_{V\cup z_B}(f)\right].
                \label{eqn:grid_expand_events}
        \end{equation}
        We now want to consider two separated blocks. One in which we consider
        the relaxation on $V$ given that $\xi=x$ and the other were we consider
        $z_B$ given that $x_{0,0}$ has a $B$-vacancy. We can do this using the
        block relaxation lemma, that we explicitly state here again for
        completeness sake as we frequently cite it throughout this paper so
        that we state it in a more general form.
        Consider two sets $V_1,
        V_2\subset \Z^d$ together with some state space $\Omega^{(i)}$ and measures
        $\nu^{(i)}$ on $V_i$ for $i\in [2]$. Write $V=V_1\cup V_2$ and
        $\nu=\nu^{(1)}\otimes \nu^{(2)}$. Consider an event $\mathcal{A}$ on
        $\Omega^{(1)}$ such that $\nu^{(1)}(\mathcal{A})>0$.
        \begin{lemma}[Block relaxation
                Lemma,~{\cite{cancrini2008kcm}*{Proposition~4.4}}]\label{lemma:block_relax}
                        In the above situation we have for $f:\Omega^{(1)}\otimes
                        \Omega^{(2)}\rightarrow \R$
                        \begin{equation}
                                \var_{\nu}(f)
                                \le \frac{2}{\nu^{(1)}(\mathcal{A})}
                                \nu(\var_{\nu^{(1)}}(f)+
                                \mathds{1}_{\mathcal{A}}\var_{\nu^{(2)}}(f)).
                        \end{equation}
        \end{lemma}
        Using this we have
        \begin{align}
                \mu&\left[\mathds{1}_{\mathcal{G}=\mathcal{C},\xi=x,
                \mathcal{E}^{(B,2)}} \var^{(B)}_{V\cup z_B}(f)\right]
                 \\&\le\frac{2}{q_B}\mu\left[\mathds{1}_{\mathcal{G}=\mathcal{C},\xi=x,
                \mathcal{E}^{(B,2)}}\left(
                \mathds{1}_{\omega_{x_{0,0}}=B}\var^{(B)}_{z_B}(f) 
                +\var^{(B)}_{V}(f)\right) \right].\label{eqn:grid_two_summands}
        \end{align}
        Let us deal with both these terms separately and start with the first
        summand. For any $\omega\in \mathcal{E}^{(B,2)}\cap
        \{\mathcal{G}=\mathcal{C}\}$ there is a unique shortest $B$-traversable
        $B$-path $\Gamma(\omega)\subset D^{(1)}_{0,0}\cup \mathcal{C}$ from
        $z_B$ to $x_{0,0}-\mathbf{e}$ for some $\mathbf{e}\in \mathcal{B}$. As
        before, the event $\mathcal{E}^{(B,2)}\cap \{\mathcal{G}=\mathcal{C}\}$
        on $\Gamma$ simplifies to $\Gamma$ being $B$-traversable. Thus, we can
        extend the variance again to get
        \begin{equation}
                \mu\left[\mathds{1}_{\mathcal{G}=\mathcal{C},\xi=x,\mathcal{E}^{(B,2)},
                        \omega_{x_{0,0}=B}}\var^{(B)}_{z_B}(f)\right]
                         \le
                \mu\left[\mathds{1}_{\mathcal{G}=\mathcal{C},\xi=x,\mathcal{E}^{(B,2)},
                        \omega_{x_{0,0}=B}}\var^{(B)}_{\Gamma}(f)\right].
                        \label{eqn:grid_reinsert_1}
        \end{equation}
        Consider the auxiliary model on $\Gamma$ with equilibrium measure
        $\mu^{(B)}_{\Gamma}$ that is given by the
        one-dimensional East model where $B$-vacancies are the vacancy state
        and the neutral state is the particle state and note that with
        $\omega_{x_{0,0}}=B$ this has ergodic boundary conditions. Since
        $|\Gamma|=O(\ell)$ we find a constant $\kappa>0$ with
        \cite{chleboun2016relaxation}*{Theorem~2} such that
        \begin{equation}
                \mathds{1}_{\omega_{x_{0,0}}=B} \var_{\Gamma}^{(B)}(f) \le
                2^{\kappa \theta_B\log_2(\theta_B)}\sum_{y\in
                \Gamma}\mu^{(B)}_{\Gamma}\left(c_x^{B} p
                q_B{(\nabla_y^{(B)}f)}^2\right),
        \end{equation}
        for $q_B$ small enough, where we used that the one-dimensional
        constraints on $\Gamma$ lower bound the two-dimensional constraints
        $c_x^B$ for $x\in \Gamma$ and that $p>\Delta$ to bound the $1/(q_B+p)$
        term coming from the conditional density of vacancies and particles in
        the East model. Inserting back into \cref{eqn:grid_reinsert_1} gives
        terms like
        \begin{equation}
                \sum_{y\in \Gamma}
                \mu\left[\mathds{1}_{\mathcal{G}=\mathcal{C},\xi=x,\mathcal{E}^{(B,2)}}
                c_y^B p q_B{(\nabla_y^{(B)} f)}^2\right]
                \le \sum_{y\in Q_{0,0}}
                \mu\left[\mathds{1}_{\mathcal{G}=\mathcal{C},\xi=x,\mathcal{E}^{(B,2)}}
                c_y^B p q_B{(\nabla_y^{(B)} f)}^2\right].
        \end{equation}
        Contrary to $\Gamma$, $Q_{0,0}$ is not dependent on the specific
        $\mathcal{C}$ and $\xi$ anymore so that we can resolve the sum over
        them to get that the first summand \cref{eqn:grid_two_summands} gives a
        contribution of
        \begin{equation}
                2^{\kappa\theta_B\log_2(\theta_B)}
                \sum_{y\in Q_{0,0}} \mu\left[c_x^B q_B p
                        {(\nabla_y^{(B)}f)}^2\right].
        \end{equation}
        For the second summand in \cref{eqn:grid_two_summands} note that we
        have not yet specified the subset $V\subset X(\mathcal{C})$.
        $X(\mathcal{C})$ is isomorphic to an equilateral box in $\Z^2$ and
        the dynamics with the constraints $c_y^{(\mathcal{G})}$ are equivalent
        to a two-dimensional East process on that box. Thus
        by~\cite{couzinie2022front}*{Proposition~3.5}(i)
        we find a subset $\{x_{0,0},x_{N,N-1}\}\subset V\subset X(\mathcal{C})$ such that
        \begin{equation}
                \var_V^{(B)}(f)
                \le 2^{\theta_B^2(1+\varepsilon/2)/4}
                \sum_{y\in V}\mu_V^{(B)}\left[c_y^{(\mathcal{G})}q_B
                p{(\nabla_y^{(B)}f)}^2\right]
        \end{equation}
        for $q_B$ small enough. Given the events
        $\{\mathcal{G}=\mathcal{C}\}\cap \{\xi=x\}$ we can again extend to
        $B$-traversable paths this time between points on $X(\mathcal{C})$ and
        using completely analogous calculations to the first summand we get.
        \begin{align}
                \mu\left[\mathds{1}_{\mathcal{G}=\mathcal{C},\xi=x,
                        \mathcal{E}^{(B,2)}}\var^{(B)}_{V}(f)\right]
                &\le 2^{\kappa \theta_B\log_2(\theta_B)}
                \sum_{y\in \mathcal{C}} 
                \mu\left[\mathds{1}_{\mathcal{G}=\mathcal{C},\xi=x,
                                \mathcal{E}^{(B,2)}}c_y^B q_B p {(\nabla^{(B)}_y
                f)}^2\right]\\
                &\le 2^{\kappa \theta_B\log_2(\theta_B)}
                \sum_{y\in \mathcal{Q}^{(B)}} 
                \mu\left[\mathds{1}_{\mathcal{G}=\mathcal{C},\xi=x,
                                \mathcal{E}^{(B,2)}}c_y^B q_B p {(\nabla^{(B)}_y
                f)}^2\right].
        \end{align}
        Resolve the sum over $\mathcal{C}$ and $\xi$ again and note that
        $Q_{0,0}$ is counted twice leading to an additional term of the order
        $O(\ell^2)$ that we absorb into $\kappa$ to get the claim. 
\end{proof}
\begin{remark}\label{rem:grid_generalisation}
        For simplicity we limited the discussion in this section to $B$-grids. The
        results generalise to $h$-grids with $q_h$ going to $0$ and the
        conditions for $q_A+q_C$ are substituted with conditions for $1-q_h-p$.
\end{remark}

\subsection{Low vacancy density: Proof of \texorpdfstring{\cref{thm:abc_relaxation}}{Theorem
6}(3.i)}\label{sec:proof_few_vacancies}
Fix an $\varepsilon>0$, let $\mathbf{q}$ be a parameter set such that
$\min_{h\in G} q_h = q_B$ and $(q_A + q_C)\theta_B^3\rightarrow 0$. By
\cref{lemma:monotonicity_in_g} we have $\gamma(G,\mathbf{q})\le \gamma_2(q_B)$,
so using \cref{thm:east_relax} we need to show that there is a $\delta>0$ so
that for $q_B<\delta$ we have
\begin{equation}
        \gamma(G,\mathbf{q}) \ge 2^{-\theta_B^2(1+\varepsilon)/4}.
\end{equation}
For $h\in G$ let $\mathcal{E}^{(h)} = \mathcal{E}^{(h,1)}\cap
\mathcal{E}^{(h,2)}$ be the events from \cref{sec:hgrid}, let
$\mathcal{E}_x^{(h)}$ be the correspondingly translated event and let $z_h$ be
the analogous vertices to $z_B$. Using the
results from \cref{sec:hgrid} we can get h-vacancies to $z_h$. We thus need an event
that allows us to bring the vacancies back to the origin.

Let $\mathcal{E}_x^{(0)}$
be the event that there is no vacancy on $\{x+i\mathbf{e}_1\colon i\in
[3]\}\cup \{x-i\mathbf{e}_2\colon i\in [3]\}$ and $\mathcal{E}_x:=
\mathcal{E}_x^{(0)}\cap \bigcap_{h\in G}\mathcal{E}_x^{(h)}$. By construction
the family ${\{\mathcal{E}_x\}}_{x\in\Z^2}$ satisfies the exterior condition
with respect to the exhausting and increasing family of sets ${\{V_n\}}_{n\in
\Z}$ given in \cref{lemma:grid_exterior}. By assumption on the parameter set
and \cref{cor:grid_failing_1,lemma:grid_failing_2},
\cref{eqn:ext_cond_thm_assumption_original} holds for the family
${\{\mathcal{E}_x\}}_{x\in \Z^2}$ for $q_B$ small enough. Thus, we can apply
the exterior condition theorem, \cref{thm:exterior_thm}, and \cref{lemma:var_as_trans} to get
\begin{equation}\label{eqn:low_freq_first_ineq}
        \var(f)
        \le 4\sum_{x\in \Z^2}\mu(\mathds{1}_{\mathcal{E}_x}\var_x(f))
        \le C\sum_{x\in \Z^2}\sum_{h\in G}
        \mu(\mathds{1}_{\mathcal{E}_x^{(0)}\cap \mathcal{E}^{(h)}_x}\var^{(h)}_x(f)).
\end{equation}
Let us consider w.l.o.g.\ only the term for $h=B$ and $x=0$ and leave away the
subscript $x$. Recall that we write $z_B=\mathbf{e}_1+3\mathbf{e}_2$ which
by $\mathcal{E}^{(B,2)}$ is $B$-traversable so that we can extend the variance,
\cref{lemma:extend_variance}, and apply the block relaxation Lemma,
\cref{lemma:block_relax}:
\begin{align}
        \mu(\mathds{1}_{\mathcal{E}^{(0)}\cap \mathcal{E}^{(B)}}\var^{(B)}_0(f))
        &\le \mu(\mathds{1}_{\mathcal{E}^{(0)}\cap
        \mathcal{E}^{(B)}}\var^{(B)}_{\{0,z_B\}}(f))\\
        &\le \frac{C}{q_B} \mu\left[\mathds{1}_{\mathcal{E}^{(0)}\cap
        \mathcal{E}^{(B)}}\mu_0^{(B)}\left(
        \mathds{1}_{\omega_{z_B}=B}\var^{(B)}_{0}(f)+\var_{z_B}^{(B)}(f)
        \right)\right]\\
        &= \frac{C}{q_B} \mu\left[\mathds{1}_{\mathcal{E}^{(0)}\cap
        \mathcal{E}^{(B)}}\left(
        \mathds{1}_{\omega_{z_B}=B}\var^{(B)}_{0}(f)+\var_{z_B}^{(B)}(f)
        \right)\right],
\end{align}
where in the last equality we used that $\mathrm{Supp}(\mathcal{E})\cap
\{0\}=\emptyset$ and the tower property. By \cref{lemma:grid_relaxation} we can
upper bound the second summand by
\begin{equation}
        \mu\left(\mathds{1}_{\mathcal{E}^{(0)}\cap
        \mathcal{E}^{(B)}} \var^{(B)}_{z_B}(f)\right)
        \le
        2^{\theta_B^2(1+\varepsilon)/4}\mu\left(\mathcal{D}_{\mathcal{Q}^{(B)}}(f)\right),
\end{equation}
where we added the missing $A$- and $C$-transition terms to get a contribution
to the Dirichlet
form. Using an analogous estimate to the bounding of the one-dimensional terms
in the proof of \cref{lemma:grid_relaxation} we can show that the first summand
is of lower order than the contribution of the second summand. By translation invariance we get analogous terms for
any $x\in \Z^2$. When taking the sum over $x$ we need to account for the
overcounting, which we recall is how many times a single vertex $y\in \Z^2$
appears in the various Dirichlet forms that we get in the above way for the
different $x\in \Z^2$. In this case, for any $y\in \Z^2$ there are
$O(|\mathcal{Q}_B|)$ different $x\in \Z^2$ such that $y\in \mathcal{Q}^{(B)}_x$
we can absorb\footnote{We often use the term absorb in this context, where we either
        mean make the constant larger/smaller or here specifically, where
        $\varepsilon$ is fixed, do the whole proof for $\varepsilon/2$ and only
        in the final step write $\varepsilon$ upper bounding any lower order
        term by $2^{\theta_B^2 \varepsilon/4}$.}
the overcounting into $\varepsilon$ for $q_B$ small enough. Thus for $h=B$ the
r.h.s.\ in \cref{eqn:low_freq_first_ineq} is upper bounded by
\begin{equation}
        \sum_{x\in \Z^2}\mu(\mathds{1}_{\mathcal{E}^{(0)}\cap
        \mathcal{E}^{(B)}}\var^{(B)}_0(f))
        \le 2^{\theta_B^2(1+\varepsilon)/4}\mathcal{D}(f),
\end{equation}
for $q_B$ small enough. The calculation works analogously for each $h\in
\mathcal{G}$ and we get the claim for the chosen $\mathbf{q}$ by arbitrariness
of $\varepsilon$. Further, the proof also works analogously for any
$\mathbf{q}$ such that $q_{\max}\theta_{q_{\min}}^3\rightarrow 0$ as
$q_{\min}\rightarrow 0$ giving part (3.i) of
\cref{thm:abc_relaxation}. \qed

\subsection{Single frequent vacancy type: Proof of
\texorpdfstring{\cref{thm:abc_relaxation}}{Theorem 3(ii)}(3.ii)}%
\label{sec:theorem3ii}
Throughout this section assume that $\mathbf{q}$ is a parameter set such that
$q_{\min}=q_B$, $q_{\max}\theta_{\qmed}^3/\log_2(\theta_{B})\rightarrow \infty$
and $\qmed\theta_{B}^6\rightarrow 0$ as $q_{B}\rightarrow 0$ where we recall
that $\qmed$ is the remaining element of
$\mathbf{q}\setminus\{q_{\max},q_{\min}\}$. 

In this case the assumptions on $\mathbf{q}$ in \cref{cor:grid_failing_1} do
not hold anymore. We resolve this problem by
working on boxes and defining traversable configurations on them that do not
exclude the frequent vacancy type. We then show that on this coarse grained
lattice we can apply the results from \cref{sec:hgrid} again and conclude the
proof by using auxiliary models and the path method to go from the coarse
grained lattice back to $\Z^2$.

We start with the proofs for the case where $q_{\min}=q_B$. We will see later
that this is sufficient as the proofs for $q_{\min}\in \{q_A,q_C\}$ are
analogous.
\subsubsection{The case \texorpdfstring{$q_{\max}=q_A$}{qmax=qA}}\label{sub:qaqmax}
Assume for this subsection that $q_{\max}=q_A$ and $q_{\mathrm{med}}=q_C$. We
start by defining the coarse graining and the states on the coarse-grained
lattice.
\begin{definition}\label{def:btraversable}
        For $\mathbf{j}\in \Z^2$ and $L=\lfloor \theta_B^3\rfloor$ let
        $\Lambda_{\mathbf{j}}=(L+1)\mathbf{j}+{\{0,\ldots, L-1\}}^2$ be an equilateral box of side length $L-1$ and
        origin $(L+1)\mathbf{j}$ and let $W_{\mathbf{j}}$ be the outline of it,
        i.e.\ the shortest cycle containing $(L+1)\mathbf{j}+\{0,(L-1)\mathbf{e}_1,
        (L-1)\mathbf{e}_2,(L-1)(\mathbf{e}_1+\mathbf{e}_2)\}$. Let the
                \emph{enlargement $EW_{\mathbf{j}}$ of $W_{\mathbf{j}}$} be the union of $W_{\mathbf{j}}$ with the set
        $\tilde{\Lambda}\setminus\Lambda_{\mathbf{j}}$ where $\tilde{\Lambda}$
        is an equilateral box of side length $L$, origin $(L+1)\mathbf{j}$ and
        denote the top right corner of $EW_{\mathbf{j}}$ by $x_{\mathbf{j}}
        =(L+1)\mathbf{j}+L(\mathbf{e}_1+\mathbf{e}_2)$. For an $\omega\in
        \Omega$ we call $EW_{\mathbf{j}}$
        \begin{itemize}
                \item \emph{$B$-traversable} (see \cref{fig:btraversable}) if
                        \begin{itemize}
                                \item $\omega_{x}\in\{\star, A\}$ for any $x\in W_{\mathbf{j}}$,
                                \item $\omega_x\in \{\star, A, B\}$ for any
                                        $x\in EW_{\mathbf{j}}\setminus
                                        W_{\mathbf{j}}$ and
                                \item for any $i\in [2]$ there is at least one
                                        $x \in
                                        (L+1)\mathbf{j}+\{\mathbf{e}_i,2\mathbf{e}_i,\ldots,
                                        (L-1)\mathbf{e}_i\}$ such that
                                        $\omega_x=A$.
                        \end{itemize}
                        Let
                        $q_{BT}:=\mu_{EW_{\mathbf{j}}}(\text{$B$-traversable})$.
                \item \emph{$B$-super} if $EW_{\mathbf{j}}$ is $B$-traversable and
                        $\omega_{x_{\mathbf{j}}}=B$. 
                        Let
                        $q_{BS}:=\mu_{EW_{\mathbf{j}}}(\text{$B$-super})$.
                \item \emph{$B$-evil}\footnotemark{} if it is not $B$-traversable.
                        Let
                        $q_{BE}:=\mu_{EW_{\mathbf{j}}}(\text{$B$-evil})$.
        \end{itemize}
\end{definition}
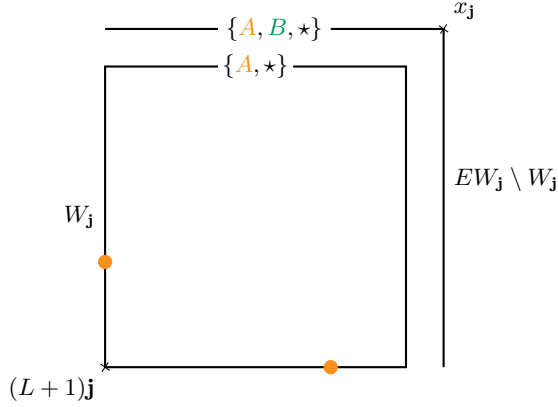
\begin{figure}[t]
        \centering
        \begin{tikzpicture}
                \draw[thick] (0,0) -- ++(0,4) -- ++(1.5,0);
                \draw[thick] (2.5,4) -- (4,4) -- (4,0) -- (0,0);
                \node[anchor=east] at (0,2) {\footnotesize $W_{\mathbf{j}}$};
                \node[] at (2,4) {\footnotesize $\{{\color{BurntOrange} A}, \star\}$};

                \draw[thick] (4.5,0) -- ++(0,4.5) -- ++(-1.5,0);
                \draw[thick] (1.5,4.5) -- ++(-1.5,0);
                \node[anchor=west] at (4.5,2.5) {\footnotesize $EW_{\mathbf{j}}\setminus W_{\mathbf{j}}$};
                \node[] at (2.25,4.5) {\footnotesize $\{{\color{BurntOrange}A},{\color{ForestGreen} B},\star\}$};

                \filldraw[BurntOrange](0,1.4)circle[radius=2.5pt] {};
                \filldraw[BurntOrange](3,0)circle[radius=2.5pt] {};

                \node[anchor=north east] at (0,0) {\footnotesize $(L+1)\mathbf{j}$};
                \draw (0,0) node[cross,rotate=10,minimum size=3] {};

                \node[anchor=south west] at (4.5,4.5) {\footnotesize $x_{\mathbf{j}}$};
                \draw (4.5,4.5) node[cross,rotate=10,minimum size=3] {};
        \end{tikzpicture}
        \caption{Illustration of a $B$-traversable $EW_{\mathbf{j}}$. Note that
        the $A$-vacancies on the bottom and left boundary can be anywhere on
        that boundary.\label{fig:btraversable}}
\end{figure}
\footnotetext{We use super and evil instead of the more common good and bad to
        avoid confusion in the notation with $G\subset H_d$ and $B$-vacancies.}
\textbf{Attention:} Previously, if we
said that $EW_{\mathbf{j}}$ was $B$-traversable, we meant that $\omega_x\in
\{\star, B\}$ for any $x\in EW_{\mathbf{j}}$ instead of the above definition.
In the context in which $q_{A}\gg q_B$ this notion of $B$-traversability has a
very small equilibrium probability so it is not useful for the proof of part
(3.ii). We justify the recycling of the name since the two notions of
traversability play analogous roles. In \cref{sec:hgrid} we
looked for grids of paths with vertices only in $\{\star, B\}$. In this
section we look for grids where each vertex is a $EW_{\mathbf{j}}$ that is
$B$-traversable in the above sense.

The next result whos that we can use the results from \cref{sec:hgrid} on the
coarse-grained lattice if the $B$-super boxes play the role of
$B$-vacancies and $B$-evil boxes the role of $A$ and $C$ vacancies. The proof
is a simple union bound.
\begin{lemma}\label{lemma:probability_limit}
        For $EW_{\mathbf{j}}$ as in \cref{def:btraversable} we have
        \begin{equation}
                \frac{\theta_{q_{BS}}}{\theta_B}\rightarrow 1, \qquad \qquad
                \theta_{q_{BS}}^3 q_{BE} \rightarrow 0,
        \end{equation}
        as $q_B\rightarrow 0$.
\end{lemma}
\begin{remark}
        While $q_A\theta_B^{3}\rightarrow \infty$ as $q_B\rightarrow 0$ we thus
        find a renormalisation such that we again have the equivalent of
        $(q_A+q_C)\theta_B^{3}\rightarrow 0$ from \cref{sec:hgrid} on the
        renormalised lattice.
\end{remark}
For $\mathbf{j}\in \Z^2$ let $\Omega_{\mathbf{j}}^*={S(G)}^{EW_{\mathbf{j}}}$,
$\mu^*_{\mathbf{j}}(\cdot)=\mu_{EW_{\mathbf{j}}}(\cdot\tc
\text{$B$-traversable})$ and let $\var_{\mathbf{j}}^*(f)$ be the associated
variance. We will only use the letter $\mathbf{j}$ in bold font to refer to
indices of $EW_{\mathbf{j}}$ and thus say interchangeably that $\mathbf{j}$ or
$EW_{\mathbf{j}}$ is $B$-traversable, $B$-super or $B$-evil. Let us come to the
analogue statement of \cref{lemma:grid_relaxation}, for which we need to define
the analogue of the events $\mathcal{E}^{(B,i)}$ for the lattice of boxes. We
define $\mathcal{Q}^{(B,*)}=\{EW_{\mathbf{j}}\colon \mathbf{j}\in
\mathcal{Q}^{(B)}\}$ for $\mathcal{Q}^{(B)}$ with side length $\ell=\lceil
\theta_{q_{BS}}^{3/2}\rceil$ and square side length $N=2^{\lceil
\theta_{q_{BS}}/2+\log_2(\theta_{q_{BS}})\rceil}$. The vector
$z_B=\mathbf{e}_1+3\mathbf{e}_2$ we now write as $\mathbf{j}_B$.

Let $\mathcal{E}^{(B,1,*)}$
be the event that we find a $B$-grid $\mathcal{C}$ in $\mathcal{Q}^{(B)}$ such
that $EW_{\mathbf{j}}$ is $B$-traversable for any $\mathbf{j}\in \mathcal{C}$
and such that there is an intersection point $\mathbf{j}_{i,j}\in
X(\mathcal{C})$ with $i,j>N/2$ and $\mathbf{e}\in \mathcal{B}$ such that
$EW_{\mathbf{j}_{(i,j)+\mathbf{e}}}$ is $B$-super.

Let $\mathcal{E}^{(B,2,*)}$
be event that for each $\mathbf{j}$ on the boundary $D_{0,0}^{(1)}$,
$EW_{\mathbf{j}}$ is $B$-traversable. We write
$\mathcal{E}^{(B,*)}=\mathcal{E}^{(B,1,*)}\cap \mathcal{E}^{(B,2,*)}$. The
support is included in $\mathcal{Q}^{(B,*)}$, i.e.\
$\mathrm{Supp}(\mathcal{E}^{(B,*)})\subset \mathcal{Q}^{(B,*)}$ and
$\mathcal{E}^{(B,*)}$ satisfies the exterior condition with respect to the same
${\{V_n\}}_{n\in \Z}$ as in \cref{lemma:grid_exterior}.

The auxiliary model for which we state the analogue of
\cref{lemma:grid_relaxation} is given by the constraints $c_{\mathbf{j}}^{*,B}$
defined as the indicator over the event that there exists an $\mathbf{e}\in
\mathcal{B}$ such that $\mathbf{j}+\mathbf{e}$ is $B$-super. Analogous to the
$\mu^{(h)}$ notation we write $\mu^{(AB)}(\cdot)=\mu(\cdot\tc \{\star, A, B\})$
and $\var^{(AB)}(f):=\var(f\mid \{\star, A, B\})$.
\begin{corollary}\label{cor:block_grid_relax}
        For any $\varepsilon>0$ we find a $\delta>0$ such that
        \begin{equation}\label{eqn:block_eqn}
                \mu(\mathds{1}_{\mathcal{E}^{(B,*)}}
                \var^{(AB)}_{x_{\mathbf{j}_B}}(f))
                \le 2^{\theta_B^2(1+\varepsilon)/4}
                \sum_{\mathbf{j}\in \mathcal{Q}^{(B)}}
                \mu\left[\mathds{1}_{\text{$\mathbf{j}$ $B$-traversable}}
                        c_{\mathbf{j}}^{*,B}\var_{x_{\mathbf{j}}}^{(AB)}(f)
                \right]
        \end{equation}
        for $q_B<\delta$.
\end{corollary}
\begin{remark}
        Notice that on the right hand side we only take the variance over the
        top-right corner points of each box instead of the variance over
        $EW_{\mathbf{j}}$. This is because with our conditions we cannot relax
        all of $EW_{\mathbf{j}}$, since the $A$-vacancies in $B$-traversable
        boxes may not reach all vertices in $EW_{\mathbf{j}}$.
\end{remark}
\begin{proof}
        As in the proof of \cref{lemma:grid_relaxation} let $\mathcal{G}$ be
        the smallest $B$-grid with $B$-traversable crossings and $\xi$ the
        vertex with the highest coordinate in the $\prec$-partial order such
        that if $\mathbf{j}_{i,j}=\xi$ then there is an $\mathbf{e}\in
        \mathcal{B}$ with a $B$-vacancy on $x_{\mathbf{j}_{(i,j)+\mathbf{e}}}$.
        Then,
        \begin{equation}
                \mu(\mathds{1}_{\mathcal{E}^{(B,*)}}
                \var^{(AB)}_{x_{\mathbf{j}_B}}(f))
                = \sum_{\mathcal{C}\ B\text{-grid}}
                \sum_{\mathbf{j}'\in X(\mathcal{C})} 
                \mu\left[\mathds{1}_{\mathcal{G}=\mathcal{C},\xi=\mathbf{j}',
                        \mathcal{E}^{(B,2,*)}}
                \var^{(AB)}_{x_{\mathbf{j}_B}}(f)\right].
        \end{equation}
        To save some space let us write $\tilde{\mathcal{E}}:=
        \{\mathcal{G}=\mathcal{C}\}\cap\{\xi=\mathbf{j}'\}\cap
        \mathcal{E}^{(B,2,*)}$. We upper bound a generic summand so fix a
        $\mathcal{C}$ and a $\mathbf{j}'$. Consider the subset
        $\mathcal{C}^{(TR,*)}:=\{x_{\mathbf{j}}\colon \mathbf{j}\in
        \mathcal{C}\}$ of top right corners of $EW_{\mathbf{j}}$. The event $
        \{\mathcal{G}=\mathcal{C}\}$ reduces to $\omega_{x_{\mathbf{j}}}\in
        \{A,B,\star\}$ on any $x_{\mathbf{j}}\in \mathcal{C}^{(TR,*)}$. In an
        analogous proof to \cref{lemma:grid_relaxation} we find
        \begin{align}
                \mu\left[\mathds{1}_{\tilde{\mathcal{E}}}
                \var^{(AB)}_{x_{\mathbf{j}_B}}(f)\right]
                   &\le 2^{\theta_B^2(1+\varepsilon)/4}
                \sum_{y\in \mathcal{C}^{(TR,*)}}
                \mu
                \left[\mathds{1}_{\tilde{\mathcal{E}}}
                c_y^{(\mathcal{G})}\var_y^{(AB)}(f)\right]\label{eqn:relax_this}
        \end{align}
        for $q_B$ small enough where $c_{x_{\mathbf{j}}}^{(\mathcal{G})}$ is
        the constraint that there is an $\mathbf{e}\in \mathcal{B}$ such that
        $\mathbf{j}+\mathbf{e}\in \mathcal{G}$ and $x_{\mathbf{j}+\mathbf{e}}$
        has a $B$-vacancy. Given $\tilde{\mathcal{E}}$ any $EW_{\mathbf{j}}$ for $\mathbf{j}\in
        \mathcal{G}$ is $B$-traversable so that
        $\mathds{1}_{\tilde{\mathcal{E}}}c_y^{(\mathcal{G})}\le
        \mathds{1}_{\tilde{\mathcal{E}}}c_{\mathbf{j}}^{*,B}$. We can thus
        upper bound the sum in the r.h.s.\ by
        \begin{align}
                \sum_{y\in \mathcal{C}^{(TR,*)}}
                \mu
                \left[\mathds{1}_{\tilde{\mathcal{E}}}
                c_y^{(\mathcal{G})}\var_y^{(AB)}(f)\right]\label{eqn:relax_this}
                \le 
                \sum_{\mathbf{j}\in \mathcal{Q}^{(B,*)}}
                \mu
                \left[\mathds{1}_{\tilde{\mathcal{E}}, \mathbf{j}\ 
                \text{$B$-traversable}}c_{\mathbf{j}}^{*,B}\var_{x_{\mathbf{j}}}^{(AB)}(f)\right].
        \end{align}
        We get the claim after resolving the sum over $\mathcal{G}$ and $\xi$
        and taking into account the overcounting which we can absorb into the
        $\varepsilon$.
\end{proof}
Given a $B$-traversable box with a neighbouring $B$-super box we want to
recover from a generic term in the r.h.s.\ in \cref{cor:block_grid_relax} a
Dirichlet form of the $ABC$-model. To that end, let us isolate two generic
situations first. The first explains how to use the $A$-vacancies on
$W_{\mathbf{j}}$ to relax $EW_{\mathbf{j}}$.
\begin{figure}[t]
        \centering
        \begin{tikzpicture}
                \draw[thick] (0,0) -- (1.8,0);
                \draw[thick] (3.2,0) -- (5,0);
                \node[] at (2.5,0) {\footnotesize $\{{\color{BurntOrange} A}, {\color{ForestGreen} B}, \star\}$};

                \draw[thick] (-4,-1) -- (-0.5,-1);
                \draw[thick] (0.5,-1) -- (5,-1);
                \node[] at (0,-1) {\footnotesize $\{{\color{BurntOrange} A}, \star\}$};

                \filldraw[BurntOrange](-4.25,-1)circle[radius=2.5pt] {};
                \filldraw[ForestGreen](5.25,0)circle[radius=2.5pt] {};
        \end{tikzpicture}
        \caption{Starting situation from
                \cref{lemma:line_relax}.\label{fig:line_relax}}
\end{figure}
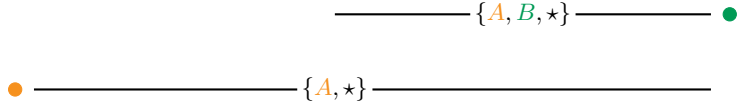
\begin{lemma}\label{lemma:line_relax}
        Let $C_2,C_1=O(\theta_B^{3})$ be two constants and consider two paths
        \begin{align}
                \Gamma_1&=\{0,\mathbf{e}_1,\ldots, C_1\mathbf{e}_1\},\\
                \Gamma_2&=\{-C_2\mathbf{e}_1-\mathbf{e}_2, -(C_2-1) \mathbf{e}_1 -
                \mathbf{e}_2, \ldots, C_1\mathbf{e}_1-\mathbf{e}_2\}.
        \end{align}
        On these paths define the event $\mathcal{A}$ that on $\Gamma_1$ we
        find no $C$-vacancies, on $\Gamma_2$ no $B$- or $C$-vacancies,
        there is an $A$ vacancy on  $\Gamma_2\setminus (\Gamma_1-\mathbf{e}_2)$ and
        $\omega_{(C_1+1)\mathbf{e}_1}=B$ (see \cref{fig:line_relax}). Then we
        find a constant $\kappa>0$ such that for any $y\in \Gamma_1$
        \begin{equation}
                \mu(\mathds{1}_{\mathcal{A}}\var_{y}^{(AB)}(f))
                \le 2^{\kappa \theta_B\log_2(\theta_B)}
                \mu(\mathds{1}_{\mathcal{A}}\mathcal{D}_{\Gamma_{1}\cup
                        \Gamma_2}(f)).
        \end{equation}
\end{lemma}
\begin{proof}
        Consider the auxiliary model on $\Gamma_1$ with constraints $c_x^B$
        that samples from $\mu^{(AB)}$ at a legal ring. If the starting state
        is in $\mathcal{A}$ then any later state is as well and the spectral
        gap of the auxiliary process agrees with that of a one-dimensional East
        model with good boundary conditions. Thus, we can extend the variance
        (\cref{lemma:extend_variance}) and find a constant $\kappa$ with
        \cite{chleboun2016relaxation}*{Theorem~2} such that
        \begin{equation}
                \mu(\mathds{1}_{\mathcal{A}}\var^{(AB)}_{y}(f))
                \le \mu(\mathds{1}_{\mathcal{A}}\var^{(AB)}_{\Gamma_1}(f))
                \le 2^{\kappa\theta_B\log_2(\theta_B)}
                \sum_{x\in \Gamma_1}
                \mu(\mathds{1}_{\mathcal{A}} c^B_x \var^{(AB)}_{x}(f)),%
                        \label{eqn:toym1}
        \end{equation}
        for $q_B$ small enough. For each $x$ we can extend the variance and use
        block relaxation, \cref{lemma:block_relax}, to get
        \begin{equation}
                \mu(\mathds{1}_{\mathcal{A}} c^B_x \var^{(AB)}_{x}(f))
                \le \frac{C}{q_A}
                \mu\left[\mathds{1}_{\mathcal{A}} c^B_x
                (\mathds{1}_{\omega_{x-\mathbf{e}_2}=A}\var^{(AB)}_x(f)
                +\var^{(A)}_{x-\mathbf{e}_2}(f))\right],
        \end{equation}
        for $q_B$ small enough. For the first summand we can write the variance
        as transition terms (\cref{lemma:var_as_trans}) and use that
        $\mathds{1}_{\omega_{x-\mathbf{e}_2}=A}\le c_x^A$ to recover a
        term of the Dirichlet form. For the second summand we can use the
        enlargement trick (\cite{chleboun2016relaxation}*{Lemma~3.6}) so that
        \begin{align}
                \mu(\mathds{1}_{\mathcal{A}} \var_y^{(AB)}(f))
                \le 2^{\kappa\theta_B\log_2(\theta_B)}
                \sum_{x\in \Gamma_1} \mu(\mathcal{D}_{\{x\}\cup \Gamma_2}).
        \end{align}
        The overcounting is of order $O(\theta_B^3)$ and can thus be absorbed into
        the $\kappa$ and we recover the claim.
\end{proof}
Being able to relax $EW_{\mathbf{j}}$ means that we can move the $B$-vacancy
freely on it using the block relaxation Lemma. The second of our isolated
results moves the $B$-vacancy from a neighbouring $B$-super box to a
$B$-traversable box. This requires the path method since $x_{\mathbf{j}}$ does
not neighbour a vertex in $W_{\mathbf{j}}$ and so we can not use
\cref{lemma:line_relax} together with the block relaxation Lemma to move a
$B$-vacancy here.
\begin{figure}[t]
        \centering
        \begin{tikzpicture}[scale=0.695]
                \draw[thick] (0,0) -- ++(0,4) -- ++(1.25,0);
                \draw[thick] (2.75,4) -- (4,4) -- (4,0) -- (0,0);
                \node[] at (2,4) {\footnotesize $\{{\color{BurntOrange} A}, \star\}$};

                \draw[thick] (4.5,0) -- ++(0,4.5) -- ++(-1.25,0);
                \draw[thick] (1.25,4.5) -- ++(-1.25,0);
                \node[] at (2.25,4.5) {\footnotesize $\{{\color{BurntOrange}A},{\color{ForestGreen} B},\star\}$};

                \filldraw[BurntOrange](0,1.4)circle[radius=2.5pt] {};
                \filldraw[BurntOrange](3,0)circle[radius=2.5pt] {};

                \draw[->,thick,BurntOrange] (3,0) -- (4,0) -- (4,3.9);

                \begin{scope}[shift={(5,0)}]
                        \draw[thick] (0,0) -- ++(0,4) -- ++(1.25,0);
                        \draw[thick] (2.75,4) -- (4,4) -- (4,0) -- (0,0);
                        \node[] at (2,4) {\footnotesize $\{{\color{BurntOrange} A}, \star\}$};

                        \draw[thick] (4.5,0) -- ++(0,4.5) -- ++(-1.25,0);
                        \draw[thick] (1.25,4.5) -- ++(-1.25,0);
                        \draw[->,thick,ForestGreen] (4.5,4.5) -- (0.1,4.5);
                        \node[] at (2.25,4.5) {\footnotesize $\{{\color{BurntOrange}A},{\color{ForestGreen} B},\star\}$};

                        \filldraw[BurntOrange](0,3.6)circle[radius=2.5pt] {};
                        \filldraw[BurntOrange](2.5,0)circle[radius=2.5pt] {};

                        \filldraw[ForestGreen](4.5,4.5)circle[radius=2.5pt] {};
                \end{scope}

                \draw[thin] (4.5,4.5) circle (1cm);

                \draw[->, thick] (5, 5) to[bend left] (11,5);

                \begin{scope}[shift={(11,2.5)}]
                        \draw[thin] (0,0) -- (0,2) -- (2,2) -- (2,0) -- cycle;
                        \node[font=\footnotesize] at (0.25,1.75) {$1$};
                        \begin{scope}[shift={(1,0.75)}]
                                \draw[thin] (-1,0) -- (-0.5,0) -- (-0.5,-0.75);
                                \draw[thin] (-1,0.5) -- (0,0.5) -- (0,-0.75);
                                \draw[thin] (0.5,0.5) -- (1,0.5);
                                \draw[thin] (0.5,-0.75) -- (0.5,0) -- (1,0);
                                \filldraw[BurntOrange](-0.5,0)circle[radius=2.5pt] {};
                                \filldraw[](0,0)circle[radius=2.5pt] {};
                                \node[font=\footnotesize] at (0.5,0) {$\star$};
                                \filldraw[](0,0.5)circle[radius=2.5pt] {};
                                \filldraw[ForestGreen](0.5,0.5)circle[radius=2.5pt] {};
                        \end{scope}
                \end{scope}

                \begin{scope}[shift={(13.5,2.5)}]
                        \draw[thin] (0,0) -- (0,2) -- (2,2) -- (2,0) -- cycle;
                        \node[font=\footnotesize] at (0.25,1.75) {$2$};
                        \begin{scope}[shift={(1,0.75)}]
                                \filldraw[BurntOrange](-0.5,0)circle[radius=2.5pt] {};
                                \filldraw[](0,0)circle[radius=2.5pt] {};
                                \filldraw[ForestGreen](0.5,0)circle[radius=2.5pt] {};
                                \filldraw[](0,0.5)circle[radius=2.5pt] {};
                                \filldraw[ForestGreen](0.5,0.5)circle[radius=2.5pt] {};
                        \end{scope}
                \end{scope}

                \begin{scope}[shift={(16,2.5)}]
                        \draw[thin] (0,0) -- (0,2) -- (2,2) -- (2,0) -- cycle;
                        \node[font=\footnotesize] at (0.25,1.75) {$3$};
                        \begin{scope}[shift={(1,0.75)}]
                                \filldraw[BurntOrange](-0.5,0)circle[radius=2.5pt] {};
                                \node[font=\footnotesize] at (0,0) {$\star$};
                                \filldraw[ForestGreen](0.5,0)circle[radius=2.5pt] {};
                                \filldraw[](0,0.5)circle[radius=2.5pt] {};
                                \filldraw[ForestGreen](0.5,0.5)circle[radius=2.5pt] {};
                        \end{scope}
                \end{scope}

                \begin{scope}[shift={(16,0)}]
                        \draw[thin] (0,0) -- (0,2) -- (2,2) -- (2,0) -- cycle;
                        \node[font=\footnotesize] at (0.25,1.75) {$4$};
                        \begin{scope}[shift={(1,0.75)}]
                                \filldraw[BurntOrange](-0.5,0)circle[radius=2.5pt] {};
                                \filldraw[BurntOrange](-0.5,0)circle[radius=2.5pt] {};
                                \filldraw[BurntOrange](0,0)circle[radius=2.5pt] {};
                                \node[font=\footnotesize] at (0.5,0) {$\star$};
                                \filldraw[](0,0.5)circle[radius=2.5pt] {};
                                \filldraw[ForestGreen](0.5,0.5)circle[radius=2.5pt] {};
                        \end{scope}
                \end{scope}

                \begin{scope}[shift={(13.5, 0)}]
                        \draw[thin] (0,0) -- (0,2) -- (2,2) -- (2,0) -- cycle;
                        \node[font=\footnotesize] at (0.25,1.75) {$5$};
                        \begin{scope}[shift={(1,0.75)}]
                                \filldraw[BurntOrange](-0.5,0)circle[radius=2.5pt] {};
                                \filldraw[BurntOrange](0,0)circle[radius=2.5pt] {};
                                \node[font=\footnotesize] at (0.5,0) {$\star$};
                                \node[font=\footnotesize] at (0,0.5) {$\star$};
                                \filldraw[ForestGreen](0.5,0.5)circle[radius=2.5pt] {};
                        \end{scope}
                \end{scope}

                \begin{scope}[shift={(11,0)}]
                        \draw[thin] (0,0) -- (0,2) -- (2,2) -- (2,0) -- cycle;
                        \node[font=\footnotesize] at (0.25,1.75) {$6$};
                        \begin{scope}[shift={(1,0.75)}]
                                \filldraw[BurntOrange](-0.5,0)circle[radius=2.5pt] {};
                                \filldraw[](0,0)circle[radius=2.5pt] {};
                                \node[font=\footnotesize] at (0.5,0) {$\star$};
                                \filldraw[BurntOrange](0,0.5)circle[radius=2.5pt] {};
                                \begin{scope}
                                        \clip (1,1) rectangle (0,0);
                                        \filldraw[ForestGreen](0,0.5)circle[radius=2.5pt] {};
                                \end{scope}
                                \filldraw[ForestGreen](0.5,0.5)circle[radius=2.5pt] {};
                        \end{scope}
                \end{scope}

        \end{tikzpicture}
        \caption{Left: Two enlarged boxes next to each other, the left is
        $B$-traversable, the right $B$-super. Right: The sequence of legal
        moves that brings $B$ from a neighbouring $B$-super box to a $B$-traversable
        box, or alternatively removes it, indicated by the two-coloured
        node.\label{fig:tr_corner_pm}}
\end{figure}
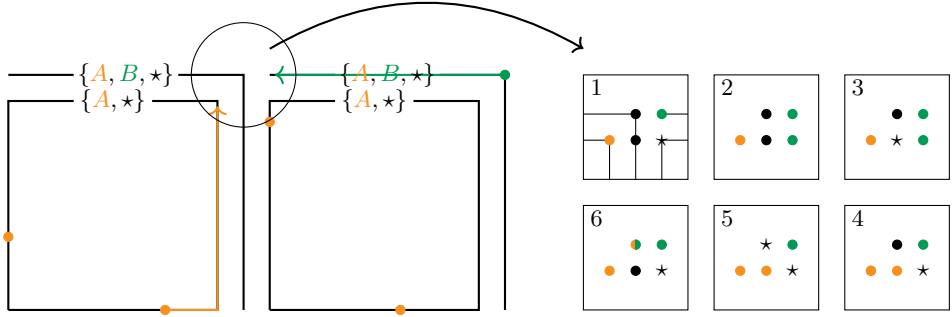
\begin{lemma}\label{lemma:grid_local_path_method}
        Consider the set $V=\{0,-\mathbf{e}_2,-\mathbf{e}_1-\mathbf{e}_2\}$ and
        the event $\mathcal{A}$ given by the $\omega\in \Omega$ such that
        $\omega_{-\mathbf{e}_1-\mathbf{e}_2}=A$,
        $\omega_{\mathbf{e}_1-\mathbf{e}_2}=\star$ and
        $\omega_{\mathbf{e}_1}=B$. Further define the event $\mathcal{A}'$
        given by the configurations $\omega$ such that $\omega_x\in \{\star,
        A,B\}$ for $x\in \{-\mathbf{e}_2,0,\mathbf{e}_1\}$ and $\omega_x\in
        \{\star, A\}$ for $x\in
        \{-\mathbf{e}_1-\mathbf{e}_2,\mathbf{e}_1-\mathbf{e}_2\}$. Then, for
        $q_B$ small enough we find a constant $\kappa$ such that
        \begin{equation}
                \mu_V\left[\mathds{1}_{\mathcal{A}}
                        \var^{(AB)}_{0}(f)\tc \mathcal{A}'\right]
                \le 2^{\kappa \theta_B} \sum_{x\in V} \mathcal{D}_{V}(f).
        \end{equation}
\end{lemma}
\begin{proof}
        The path to use the path method with, and thus the proof, is apparent
        from \cref{fig:tr_corner_pm}.
\end{proof}
Armed with these results we can upper bound the right hand side in
\cref{cor:block_grid_relax}. For this we introduce the notation
$EW(V)=\cup_{\mathbf{j}\in V} EW_{\mathbf{j}}$ for any subset $V\subset
\Z^2$.
\begin{lemma}\label{lemma:b_block_relax}
        Let $\mathbf{j}\in \Z^2$ and $V=\{\mathbf{j},\mathbf{j} +\mathbf{e}_1,
        \mathbf{j}+\mathbf{e}_2\}$. We find a constant $\kappa>0$ such that
        \begin{equation}
                \mu\left[\mathds{1}_{\text{$\mathbf{j}$ $B$-traversable}}
                        c_{\mathbf{j}}^{*,B}\var_{x_{\mathbf{j}}}^{(AB)}(f)
                \right]
                \le 2^{\kappa\theta_B\log_2(\theta_B)}
                \mathcal{D}_{EW(V)}(f)
        \end{equation}
        for $q_B$ small enough.
\end{lemma}
\begin{proof}
        W.l.o.g.\ consider only the case $\mathbf{j}=\mathbf{0}$ and where the
        constraint on the l.h.s.\ of the claim is replaced by
        $\tilde{c}=\mathds{1}_{EW_{\mathbf{e}_1}\ \text{is $B$-super}}$. Let
        $U=x_{\mathbf{0}}+\{0,\mathbf{e}_1,-\mathbf{e}_2-\mathbf{e}_1,
        -\mathbf{e}_2, -\mathbf{e}_2+\mathbf{e}_1\}$ and let $\mathcal{A}$ be
        the event from \cref{lemma:grid_local_path_method} translated by
        $x_{\mathbf{0}}$, so that $\mathrm{Supp}(\mathcal{A})\subset U$. Analogously
        define $\mathcal{A}'$ as the translated version of $\mathcal{A}'$ from
        \cref{lemma:grid_local_path_method}. The event that 
        $EW_{\mathbf{0}}$ and $EW_{\mathbf{e}_1}$ are $B$-traversable on $U$
        reduces to $\mathcal{A}'$. Thus, we can extend the variance,
        \cref{lemma:extend_variance}, and use the block relaxation
        \cref{lemma:block_relax},
        \begin{align}
                \mu&\left[\tilde{c} \mathds{1}_{\text{$EW_{\mathbf{0}}$ $B$-traversable}}
                        \var_{x_{\mathbf{0}}}^{(AB)}(f)
                \right] \\
                   &\le \mu
                \left[\tilde{c}\mathds{1}_{\text{$EW_\mathbf{0}$ $B$-traversable}}
                        \var_{U}(f\tc \mathcal{A}')
                \right]\\
                        &\le 2^{\kappa\theta_B} \mu
                \left[\tilde{c}\mathds{1}_{\text{$EW_\mathbf{0}$ $B$-traversable}}
                        \left(\mathds{1}_{\mathcal{A}}
                                        \var^{(AB)}_{x_{\mathbf{0}}}(f)
                        +\var_{U\setminus \{x_{\mathbf{0}}\}}(f\tc \mathcal{A}')\right)
                \right]
        \end{align}
        for $q_B$ small enough. The first summand can be upper bounded using
        \cref{lemma:grid_local_path_method}. For the second summand write
        \begin{equation}
                \mu\left[ \tilde{c}\mathds{1}_{\text{$EW_0$ $B$-traversable}}
                        \var_{U\setminus \{x_{\mathbf{0}}\}}(f\tc
                        \mathcal{A}')\right]
                \le\mu\left[ \tilde{c}\mathds{1}_{\text{$EW_0$ $B$-traversable}}
                 \sum_{y\in U\setminus \{x_{\mathbf{0}}\}}
                \var_y(f\tc\mathcal{A}')\right].
        \end{equation}
        For $y\in (EW_{\mathbf{0}}\setminus W_{\mathbf{0}})\cup
        (EW_{\mathbf{e}_1}\setminus W_{\mathbf{e}_1})$ we can use
        \cref{lemma:line_relax} and for the others we can use the enlargement
        trick (\cite{chleboun2016relaxation}*{Lemma~3.6}) to get the claim.
\end{proof}

Combining the previous results we thus have
\begin{equation}\label{eqn:B_block_relax}
        \mu(\mathds{1}_{\mathcal{E}^{(B,*)}}
        \var^{(AB)}_{x_{\mathbf{j}_B}}(f))
        \le 2^{\theta_B^2(1+\varepsilon)/4}
        \mu\left[\mathcal{D}_{\mathcal{Q}^{(B,*)}}(f)
        \right],
\end{equation}
for $q_B$ small enough. For $C$-vacancies we can use the same construction of
$EW_{\mathbf{j}}$ and $W_{\mathbf{j}}$ with length parameter $\lfloor
\theta^{3}_C\rfloor$ and define $C$-traversable, -super, and -evil by replacing
the $B$-vacancies with $C$-vacancies. Recall that we assume
$q_A\theta_{q_C}^3/\log_2(\theta_B)\rightarrow \infty$ as $q_B\rightarrow 0$ so
that the results follow analogously for $C$-vacancies with minor adjustments.
We omit details here that lead to the result that
\begin{equation}\label{eqn:C_block_relax}
        \mu(\mathds{1}_{\mathcal{E}^{(C,*)}}
        \var^{(AC)}_{x_{\mathbf{j}_C}}(f))
        \le 2^{\theta_B^2(1+\varepsilon)/4}
        \mu\left[\mathcal{D}_{\mathcal{Q}^{(C,*)}}(f)
        \right],
\end{equation}
for $\mathbf{j}_C=-3\mathbf{e}_1+\mathbf{e}_2$ and $q_B$ small enough.

As in the low vacancy density case we need a final event that brings the $B$-
resp.\ $C$-vacancy from $\mathcal{Q}^{(h,*)}(x)$ to $x$. To that end, let us
define some paths (see \cref{fig:paths}).
\begin{itemize}
        \item Let $\Gamma^{(B)}$ be a shortest path
                starting at $\mathbf{e}_2$ and ending at
                the first vertex neighbouring $EW_{\mathbf{j}_B}\setminus
                W_{\mathbf{j}_B}$ that first goes straight up and
                then right.
        \item Let $\Gamma^{(B,left)}$ be the path that starts at
                $-\lfloor \theta_C^3\rfloor \mathbf{e}_1+\mathbf{e}_2$
                and is straight until $-\mathbf{e}_1+\mathbf{e}_2$ and then equal to
                $(\Gamma^{(B)}-\mathbf{e}_1)\setminus \Gamma^{(B)}$. Let $\Gamma^{(B,right)}$ be the
                path starting at $\mathbf{e}_1+\lfloor
                \theta_B^3\rfloor\mathbf{e}_2$ that goes straight up until it
                hits $\Gamma^{(B)}-\mathbf{e}_2$, which it then follows to the
                right.
        \item Let $\Gamma^{(C)}$ be the shortest path that starts at $-\mathbf{e}_1$ and goes
                straight left and then up that ends up at a vertex neighbouring
                $EW_{\mathbf{j}_{C}}\setminus
                W_{\mathbf{j}_C}$. Let $x^{(C)}$ be the point
                where the path switches from going left to going up.
        \item Let $\Gamma^{(C,left)}$ be the union of
                $\Gamma^{(C)}-\mathbf{e}_1-\mathbf{e}_2$ and $\{x^{(C)}-\lfloor
                \theta_C^3\rfloor \mathbf{e}_1,\ldots, x^{(C)}-\mathbf{e}_1\}$.
\end{itemize}
\begin{figure}[t]
        \centering
        \begin{tikzpicture}

                \fill[opacity=0.2, fill] (2,4) -- (3.5,4) -- (4.5,6) -- (4.5,7.5) --
                        (2,7.5) -- (2,4);
                \node at (3.5, 7.2) {$\mathcal{Q}^{(B,*)}$};

                \fill[opacity=0.2, fill] (-2,4) -- (-3.3,4) -- (-4.3,6) -- (-4.3,7.3) --
                        (-2,7.3) -- (-2,4);
                \node at (-3, 6.9) {$\mathcal{Q}^{(C,*)}$};

                \filldraw[](0,0)circle[radius=2.5pt] {};
                \draw (-1,-1) -- (5,5);

                \draw[thick] (0,0.25) -- (0,4.5) -- (1.9,4.5);

                \draw[thick] (-1.5,0.25) -- (-0.25,0.25) -- (-0.25,4.5);
                \filldraw[BurntOrange](-1.2,0.25)circle[radius=2.5pt] {};

                \draw[thick] (0.25,1.5) -- (0.25,4.25) -- (1.9,4.25);
                \filldraw[BurntOrange](0.25,2)circle[radius=2.5pt] {};

                \draw[thick] (2,4) rectangle (2.5,4.5);
                \draw[thick] (3,4) rectangle (3.5,4.5);
                \draw[thick] (2,5) rectangle (2.5,5.5);
                \draw[thick] (3,5) rectangle (3.5,5.5);
                \draw[thick] (2,6) rectangle (2.5,6.5);
                \draw[thick] (3,6) rectangle (3.5,6.5);
                \draw[thick] (4,6) rectangle (4.5,6.5);

                \draw[thick] (-0.25,0) -- (-2,0) -- (-2,3.9);

                \draw[thick] (-3,-0.25) -- (-0.5,-0.25);
                \filldraw[BurntOrange](-2.9,-0.25)circle[radius=2.5pt] {};

                \draw[thick] (-2.25,-0.25) -- (-2.25,3.9);

                \draw[thick] (-2,4) rectangle (-2.3,4.3);
                \draw[thick] (-3,4) rectangle (-3.3,4.3);
                \draw[thick] (-2,5) rectangle (-2.3,5.3);
                \draw[thick] (-3,5) rectangle (-3.3,5.3);
                \draw[thick] (-2,6) rectangle (-2.3,6.3);
                \draw[thick] (-3,6) rectangle (-3.3,6.3);
                \draw[thick] (-4,6) rectangle (-4.3,6.3);

                \node[anchor=south, font=\footnotesize] at (1,4.5) {$\Gamma^{(B)}$};
                \node[anchor=north, font=\footnotesize] at (1,4) {$\Gamma^{(B,right)}$};
                \node[anchor=east, font=\footnotesize] at (-0.2,3) {$\Gamma^{(B,left)}$};
                \node[anchor=west, font=\footnotesize] at (-2,2) {$\Gamma^{(C)}$};
                \node[anchor=north, font=\footnotesize] at (-2,-0.25) {$\Gamma^{(C,left)}$};





        \end{tikzpicture}
        \caption{Image of the various $\Gamma$ paths and the exemplary
        $A$-vacancies where $\mathcal{E}^{(0)}$ requires them. Note the
        different sizes of $\mathcal{Q}^{(B,*)}$ and
        $\mathcal{Q}^{(C,*)}$.\label{fig:paths}}
\end{figure}
Notice that since $\Gamma^{(B,left)}$ starts at $-\lfloor \theta_C^3\rfloor
\mathbf{e}_1+\mathbf{e}_2$ the various paths do not intersect. We define
$\mathcal{E}^{(0)}$ as the $\omega\in \Omega$ such that
\begin{itemize}
        \item $\omega_x\in \{\star,
                A,B\}$ for any $x\in \Gamma^{(B)}$.
        \item  $\omega_x\in \{\star, A\}$ for any $x\in
                \Gamma^{(B,left)}\cup \Gamma^{(B,right)}$ and there is
                at least one $A$-vacancy on $\Gamma^{(B,left)}\setminus
                \{\Gamma^{(B)}-\mathbf{e}_1\}$ and on
                $\Gamma^{(B,right)}\setminus \{\Gamma^{(B)}-\mathbf{e}_2\}$.
        \item $\omega_x\in \{\star, A,C\}$ for any $x\in \Gamma^{(C)}$.
        \item $\omega_x\in \{\star, A\}$ for any $x\in
                \Gamma^{(C,left)}$ at least one $A$-vacancy on
                $\Gamma^{(C,left)}\setminus
                (\Gamma^{(C)}-\mathbf{e}_1-\mathbf{e}_2)$.
\end{itemize}
The support of $\mathcal{E}^{(0)}$ by construction has no intersection with
$\mathcal{Q}^{(B,*)}$ and $\mathcal{Q}^{(C,*)}$. Let
$\mathcal{E}^{(*)}:=\mathcal{E}^{(0)}\cap \mathcal{E}^{(B,*)}\cap
\mathcal{E}^{(C,*)}$ and let $\mathcal{E}_x^{(*)}$ be the translated version by
$x\in \Z^2$. Then the family ${\{\mathcal{E}_x^{(*)}\}}_{x\in \Z^2}$ satisfies
the exterior condition w.r.t.\ to the same family of sets as given in
\cref{lemma:grid_exterior}. Using the assumptions on $\mathbf{q}$ it is
straightforward to check that
\begin{align}
        \lim_{q_B\rightarrow 0}
        \mathrm{Supp}(\mathcal{E}^{(0)})\mu(1-\mathcal{E}^{(0)})=0.
\end{align}
Combining this with \cref{lemma:probability_limit} and the results from
\cref{sub:renorm_probability} we can apply the exterior condition theorem,
\cref{thm:exterior_thm}.
Further, $\mathcal{E}^{(0)}$ fulfills that analogous role to the eponymous
event in \cref{sec:proof_few_vacancies} as we see in the next Lemma.
\begin{lemma}\label{lemma:origin_relax_box}
        Let $\mathcal{A}$ be the event defined
        by the intersection
        \begin{equation}
            \mathcal{A}:=\mathcal{E}^{(0)}\cap \{EW_{\mathbf{j}_B}\
            \text{$B$-super}\}\cap
            \{EW_{\mathbf{j}_{C}}\
            \text{$C$-super}\}.
        \end{equation}
        Then,
        \begin{equation}
                \mu(\mathds{1}_{\mathcal{A}}
                \var_0(f))
                \le 2^{\kappa\theta_B\log_2(\theta_B)}
                \mathcal{D}_{\mathrm{Supp}(\mathcal{A})}(f).
        \end{equation}
        for $q_B$ small enough.
\end{lemma}
\begin{proof}[Sketch of the proof]
        We only give a sketch since the employed techniques are always the
        same.
        Extending the variance, \cref{lemma:extend_variance} and using
        block relaxation, \cref{lemma:block_relax}, gives
        \begin{equation}
                \mu(\mathds{1}_{\mathcal{A}}
                 \var_0(f))
                \le \frac{C}{q_B}
                \mu\left[\mathds{1}_{\mathcal{A}}
                \left(\var^{(AB)}_{\mathbf{e}_2}(f)
                +\mathds{1}_{\omega_{\mathbf{e}_2}=B}
                \var_0(f)\right)\right].
        \end{equation}
        For the first summand, given $\mathcal{A}$, we can use a combination
        of the block relaxation Lemma (\cref{lemma:block_relax}) and
        \cref{lemma:line_relax} to get an appropriate upper bound. For the
        second summand we can repeat the calculation for the $C$-vacancy side
        to get a term
        \begin{equation}
                \frac{C}{q_A q_B q_C}
                \mu(\mathds{1}_{\mathcal{A},\omega_{\mathbf{e}_2}=B,
                \omega_{-2\mathbf{e}_1}=C,
                \omega_{-2\mathbf{e}_1-\mathbf{e}_2}=A}
                \var_0(f)).
        \end{equation}
        Write $\var_0(f)$ as a sum of transition terms using
        \cref{lemma:var_as_trans}, for the $B$-transition use that
        $\mathds{1}_{\omega_{\mathbf{e}_2}=B}\le c_0^{B}$ and for the $A$ and
        $C$ transition terms we can use the path method recalling that for
        $\omega\in \mathcal{A}$ we have $\omega_{-\mathbf{e}_1}\in \{\star,
        A,C\}$ (analogously to, for example, the situation in
        \cref{fig:tr_corner_pm}). The claim follows.
\end{proof}
Since the intersection of the various grids is negligible this gives us the
proof for the case of $q_A=q_{\max}$. We omit the proof as it is a
straightforward implication of the above Lemma with the block relaxation Lemma.
\begin{proposition}
        For parameter sets as fixed in the beginning of the section with
        $q_{\max}=q_A$ we have
        \begin{equation}
                \lim_{q_B\rightarrow 0}
                \frac{\gamma(G,\mathbf{q})}{\gamma_2(q_B)}
                \ge 1.
        \end{equation}
\end{proposition}
We never explicitly used that $q_C>q_B$ so the same result also holds in the case
$q_B=q_{\mathrm{med}}$ and $q_C=q_{\min}$. Further, by symmetry this also
covers the case $q_B=q_{\max}$. The case $q_C=q_{\max}$ is analogous.
Indeed, above $EW_{\mathbf{j}}$ extended $W_{\mathbf{j}}$ to the north and
east (i.e. the boxes shared their origin). In the case $q_C=q_{\max}$ we do the
completely analogous construction only that $W_{\mathbf{j}}$ and
$EW_{\mathbf{j}}$ share the north-west corner. As everything else works
analogously we omit details here.\qed

\subsection{Single low density vacancy type: Proof of
\texorpdfstring{\cref{thm:abc_relaxation}}{Theorem 2}(3.iii)}%
\label{sec:theorem3iii}
In this section consider again the $G$-MCEM with $G=\{A,B,C\}$ this time with a
parameter set $\mathbf{q}$ such that $q_{\min}=q_B$ and
$\liminf_{q_{B}\rightarrow 0} \qmed >0$, i.e.\ there is a constant $\lambda>0$
with $q_A,q_C>\lambda$ for $q_B$ small enough. This covers case (3.iii) since
both $A$- and $C$-vacancies share the direction $\mathbf{e}_1$, the other case
in which $q_{\min}=q_A$, is equivalent to the present case by symmetry.

Using that both $A$- and $C$-vacancies have a high equilibrium density, we find
configurations that clear any non-$B$-vacancy in the $\mathbf{e}_1$-direction.
As in previous proofs we work with block lattices. In this section we let
${\{W_{\mathbf{j}}\}}_{\mathbf{j}\in \Z^2}$ be the block lattice given by boxes
of side lengths $(0,2)$ so that
\begin{equation}
        W_{\mathbf{j}}=(\mathbf{j}_1,3 \mathbf{j}_2)+\{0, \mathbf{e}_2,
                2\mathbf{e}_2\}.
\end{equation}
We call $(\mathbf{j}_1,3 \mathbf{j}_2)$ the \emph{lower vertex of
$W_{\mathbf{j}}$}, $(\mathbf{j}_1,3 \mathbf{j}_2+2)$ the \emph{upper vertex},
the set of lower and upper vertices we then call the \emph{outer vertices} and
$(\mathbf{j}_1,3 \mathbf{j}_2+1)$ the \emph{central vertex}. The
associated local state space is $\Omega_{\mathbf{j}}^{*} :=
{\{0,1\}}^{W_{\mathbf{j}}}$, the equilibrium measure is
$\mu^*_{\mathbf{j}}=\mu^*_{W_{\mathbf{j}}}$ and the variance is
$\var^*_{\mathbf{j}}(f)=\var_{W_{\mathbf{j}}}(f)$. For $\omega\in \Omega^*$ we
say that $W_{\mathbf{j}}$
\begin{itemize}
        \item is $B$-traversable, if there is no $B$ on the outer vertices.

        \item is $B$-super, if it is $B$-traversable and the central vertex is
                $B$.
        \item is $AC$-traversable, if there is no $B$ on $W_{\mathbf{j}}$.
        \item is $AC$-super, if it is $AC$-traversable, the
                lower vertex is $A$ and the upper vertex is $C$.
\end{itemize}
\begin{remark}
        To justify the above definitions and the recycling of the traversable
        and super names let us give a high level overview of what we do with
        these states to prepare the reader for the detailed calculations.
        Recall that $A$-vacancies propagate north and east, while $C$-vacancies
        propagate south and east. In an $AC$-super box the central vertex is
        always facilitated for any transition from $A$ or $C$ to the neutral
        state and vice versa. By $A$- and $C$-vacancies sharing the east
        propagation direction this extends to any vertex in an $AC$-traversable
        box to the east of an $AC$-super box (see \cref{lemma:ac_super_prop}).

        Further, if there are any $B$-traversable or $B$-super boxes to the
        East of an $AC$-super box, following at least one $AC$-traversable box,
        we can also remove any non $B$-vacancy from the central vertex. This is
        what allows us in \cref{lemma:many_ac_ver,lemma:many_ac_hor} to
        propagate the $B$-vertices from $B$-super boxes on paths of
        $B$-traversable boxes given an appropriate configuration of $AC$-super
        and traversable boxes.
\end{remark}

As in the previous proofs our goal is to define
a set of events ${\{\mathcal{E}_{\mathbf{j}}\}}_{\mathbf{j}\in \Z^2}$ on which
we can use the exterior condition theorem and where $\mathcal{E}_x$ allows us
to recover a Dirichlet form of the $ABC$-model starting from a term like
$\mu(\mathds{1}_{\mathcal{E}_{\mathbf{j}}}\var_{W_{\mathbf{j}}}(f))$ at a cost
$2^{\theta_B^2(1+\varepsilon)/4}$ for $q_B$ small enough.

For this we cannot use $\mathcal{Q}^{(B)}$ anymore since there is no obvious
relaxation scheme that allows us to transport $B$-vacancies on coarse-grained
$B$-paths (as in the proof of part (3.ii)). Since the $A$-vacancies have a high
frequency we also do not have to make a construction that stays above the
diagonal as in \cref{lemma:grid_exterior} to satisfy the exterior condition. We
can work with the set $V_0$ given by the vertices `below' the line that goes
through the origin and $2^{\theta_B^2}\mathbf{e_1}+\mathbf{e}_2$ and define
$V_n=V_0+n\mathbf{e}_2$ for any $n\in \Z^2$ so that ${\{V_n\}}_{n\in \N}$ is an
increasing and exhausting set. This allows us to construct a lattice of
straight lines of side length at most $2^{\theta_B^2}$ in the positive quadrant
and still put a condition on the line going in the $-\mathbf{e}_1$ direction
from the origin while satisfying the exterior condition.

Let $\ell=\lceil \theta_B^{3/2}\rceil$ and $N=2^{\lceil
\theta_{q_B}/2+\log_2(\theta_B)\rceil}$. For $i\in [0, N]$ we call the box of side
lengths $(\ell-1, N\ell-1)$ with origin at $i\ell\mathbf{e}_1+\mathbf{e}_2$ the
$i$-th vertical strip $Q^{(v)}_i$. For $j\in [0,N]$ we call the box with side
lengths $(N\ell-1, \ell-1)$ and origin at $(j\ell+1)\mathbf{e}_2$ the $j$-th
horizontal strip $Q^{(h)}_j$. We denote by $Q_{i,j}$ the equilateral box of
side length $\ell-1$ given by $Q^{(v)}_i\cap Q^{(h)}_j$. The union
$\mathcal{Q}^{(B)}$ of all strips is an equilateral box of side length $N\ell-1$
and origin $\mathbf{e}_2$.

The dynamics to propagate $B$-vacancies on horizontal and vertical paths is
different. $A$- and $C$-vacancies only share the $\mathbf{e}_1$ direction so
that $AC$-super boxes can only propagate in an $\mathbf{e}_1$ directions, which
means that for each row we want to move a $B$-super box vertically, we need an
$AC$-super box somewhere that removes any $A$- or $C$-vacancies. To propagate
$B$-super boxes horizontally a single $AC$-super box suffices. Thus
vertically we need boxes that guarantee us the $AC$-super vertices.
\begin{figure}[t]
        \centering
        \begin{tikzpicture}[scale=1.5]
                 \tikzset{ground/.style={%
                        postaction={draw,line width=0.45mm,BurntOrange,
                                    dash pattern=on 5pt off 13pt,
                                    dash phase=6pt},
                          postaction= {draw,line width=0.45mm,Cyan,
                                    dash pattern=on 5pt off 13pt,
                                    dash phase=12pt},
                        Black,dash pattern= on 5pt off 13pt,
                                    line width=0.45mm}}
         \draw[ground] (0,0.5) -- (6,0.5);
         \draw[ground] (0,0) -- (5.9,0);
         \draw[ground] (0,-0.5) -- (6,-0.5);
        \filldraw[Cyan](1,0.5)circle[radius=2.5pt] {};
        \filldraw[BurntOrange](1,-0.5)circle[radius=2.5pt] {};
        \filldraw[Black](5.9,0)circle[radius=2.5pt] {};
        \draw [thick,decorate,decoration={brace,amplitude=5pt,mirror},
               yshift=-1pt]
               (0,-0.6) -- (6,-0.6)
               node[black,midway,yshift=-0.6cm] {\footnotesize $\theta_B^{5/4}$};

        \draw[thick, ->] (3, 0.75) -- (3,1.25) node[anchor=west, midway] {\footnotesize $N$};

        \draw[gray] (5.8, -0.6) rectangle (6.0,0.6);
        \draw[gray] (0.9, -0.6) rectangle (1.1,0.6);
        \end{tikzpicture}
        \caption{One of the $N$ horizontal sections of a vertical crossing (see
                \cref{def:vert_cross}). Three vertically aligned vertices
                (e.g.\ the rectangles drawn) are one box $W_{\mathbf{j}}$. The
                right box is on the right boundary and thus by assumption
                $B$-traversable, so that there is no condition on the central
        vertex (black dot). All other vertices have no $B$-vacancy. The left
box is the $AC$-super box implied by the definition of vertical
crossings.\label{fig:ac_vert_crossing}}
\end{figure}
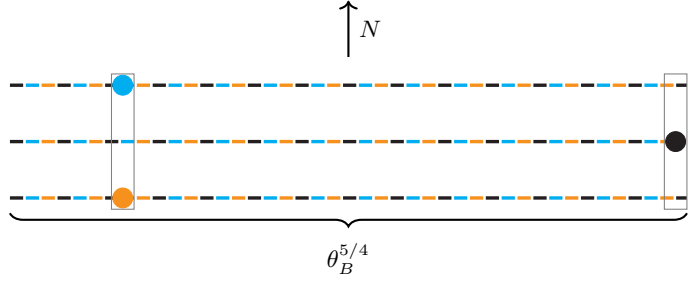
\begin{definition}[Vertical crossing]\label{def:vert_cross}
        Consider a box $\Lambda\subset Q^{(v)}_i$ of side lengths $(\lfloor
        \theta_B^{5/4}\rfloor -1,N\ell-1)$ with origin $\mathbf{j}_0$. Let
        $\partial^{(r)}\Lambda$ be the right boundary of $\Lambda$,
        i.e.\ the $\mathbf{j}\in \Lambda$ such that $\mathbf{j}\cdot
        \mathbf{e}_1=\mathbf{j}_0\cdot \mathbf{e}_1+\lfloor \theta_B^{5/4}\rfloor -1$. For
        $\omega\in \Omega^*_{\mathcal{Q}^{(B)}}$, $\Lambda$ is a \emph{vertical
        crossing of $Q^{(v)}_i$} if
        \begin{itemize}
                \item $W_{\mathbf{j}}$ is $B$-traversable for any
                        $\mathbf{j}\in \partial^{(r)}\Lambda$.
                \item $W_{\mathbf{j}}$ is $AC$-traversable for any
                        $\mathbf{j}\in \Lambda\setminus \partial^{(r)}\Lambda$.
                \item There is at least one $\mathbf{j}$ per row in
                        $\Lambda\setminus (\partial^{(r)}\Lambda\cup
                        (\partial^{(r)}\Lambda-\mathbf{e}_2))$ such that
                        $W_{\mathbf{j}}$ is $AC$-super.
        \end{itemize}
\end{definition}
The main idea behind this definition is the following: To propagate a $B$-super
box on the right boundary, we use that on each row there is an $AC$-super box
on a line of $AC$-traversable boxes. This $AC$-super box can remove any $A$- or
$C$-vacancy from the $AC$-traversable part and then in particular also from the
$B$-traversable part on the right boundary of $\Lambda$, which then allows the
$B$-vacancy in the $B$-super box to move down. Let us isolate this horizontal
motion of $AC$-super boxes. Recall for this, that we write $\mu^{(AC)}$ and $\var^{(AC)}$
to denote the measure resp.\ variance conditioned on there only being $A$ and $C$
vacancies and that by definition
\begin{equation}
        \var^*_{\mathbf{j}}(f\mid \text{$AC$-traversable})
        =\var_{W_{\mathbf{j}}}^{(AC)}(f).
\end{equation}
Using the path method the next lemma is straight forward.
\begin{lemma}\label{lemma:ac_super_prop}
        We find a constant such that
        \begin{equation}
                \mu(\mathds{1}_{\text{$W_{-\mathbf{e}_1}$ $AC$-super}}
                \var^*_{\mathbf{0}}(f\tc \text{$AC$-traversable})
                \le C \mu(\mathcal{D}_{W_{-\mathbf{e}_1}\cup W_{\mathbf{0}}}(f)).
        \end{equation}
\end{lemma}
With this we can show how $B$-super boxes propagate vertically on vertical
crossings.
\begin{lemma}[Vertical propagation]\label{lemma:many_ac_ver}
        Let $\Lambda\subset Q^{(v)}_i$ as in \cref{def:vert_cross},
        and let $\mathbf{j}^{(1)}\in \partial^{(r)}\Lambda\cap Q_{i,0}$. Let
        $\mathcal{A}^{(v)}$ be the event given by the $\omega$ such that
        $\Lambda$ is a vertical crossing of $Q^{(v)}_i$ and
        there is a $\mathbf{j}^{(2)}\in \partial^{(r)}\Lambda\cap Q_{i,1}$
        such that $W_{\mathbf{j}^{(2)}}$ is $B$-super. Then,
        \begin{equation}
                \mu\left(\mathds{1}_{\mathcal{A}^{(v)}}
                        \var^*_{\mathbf{j}^{(1)}}(f\tc \mathcal{A}^{(v)})\right) \le
                        2^{\kappa\theta_B\log_2(\theta_B)}
                        \mu(\mathcal{D}_{Q_{i,0}\cup Q_{i,1}}(f)).
        \end{equation}
\end{lemma}
\begin{proof}
        W.l.o.g.\ assume that the right boundary $\partial^{(r)}\Lambda$ of
        $\Lambda$ is on the vertical axis such that
        $\mathbf{j}^{(1)}=\mathbf{0}$ and assume also w.l.o.g.\ that the
        $B$-super $\mathbf{j}^{(2)}$ implied by $\mathcal{A}^{(v)}$ is on the
        furthest vertex in $\partial^{(r)}\Lambda\cap Q_{i,1}$ from the origin,
        i.e.\ $\mathbf{j}^{(2)}=(2\ell-1)\mathbf{e}_2$. Let
        $\Gamma=\{\mathbf{0},\mathbf{e}_2,\ldots,
        \mathbf{j}^{(2)}-\mathbf{e}_2\}$ be the part of the right boundary
        starting at $\mathbf{j}^{(1)}$ and stopping right before
        $\mathbf{j}^{(2)}$. Let $c_{\mathbf{j}}^{(v)}$ be the constraint given
        by the indicator over the event that $W_{\mathbf{j}+\mathbf{e}_2}$ is
        $B$-super if $\mathbf{j}\neq \mathbf{j}^{(2)}-\mathbf{e}_2$ and $1$ if
        $\mathbf{j}=\mathbf{j}^{(2)}-\mathbf{e}_2$.

        Consider the auxiliary process on $\Gamma$ with the constraints
        $c_{\mathbf{j}}^{(v)}$ that, if $W_{\mathbf{j}}$ is unconstrained,
        samples from all $B$-traversable states on $W_{\mathbf{j}}$. The
        equilibrium measure of this process is given by
        $\mu_{\Gamma}^{(*,BT)}:=\otimes_{\mathbf{j}\in
        \Gamma}\mu_{\mathbf{j}}^*(\cdot\tc \text{$B$-traversable})$. Since
        $\mu^{(*,BT)}_{\mathbf{j}}(\text{$B$-super})=\mu^*_{\mathbf{j}}(\text{$B$-super}\tc
        \text{$B$-traversable})=q_B$ the spectral gap of this process is equal
        to the spectral gap of the one-dimensional East model with vacancy
        density $q_B$ on $\Gamma$ with good boundary conditions.

        Hence, we can extend the variance (\cref{lemma:extend_variance}) and
        use \cite{chleboun2016relaxation}*{Theorem~2} to get
        \begin{align}
                \mu\left(\mathds{1}_{\mathcal{A}^{(v)}}
                        \var^*_{\mathbf{j}^{(1)}}(f\tc \mathcal{A}^{(v)})\right)
                &\le \mu\left(\mathds{1}_{\mathcal{A}^{(v)}}
                        \var_{\mu_{\Gamma}^{(*,BT)}}(f)\right)\\
                &\le2^{\kappa\theta_B\log_2(\theta_B)}
                \sum_{\mathbf{j}\in \Gamma}
                \mu\left(\mathds{1}_{\mathcal{A}^{(v)}}
                        c_{\mathbf{j}}^{(v)}\var_{\mu_{\mathbf{j}}^{(*,BT)}}(f)
                \right).
        \end{align}
        Consider the summand for $\mathbf{j}=\mathbf{0}$ and let $\omega\in
        \mathcal{A}^{(v)}$. Let $V^{(0)}=\cup_{j=0}^{3}j\mathbf{e}_2$ be the
        union of vertices in $W_{\mathbf{0}}$ together with the lower vertex of
        $W_{\mathbf{e}_2}$ and recall that by $c_{\mathbf{0}}^{(v)}$ the vertex
        $4\mathbf{e}_2$ has a $B$-vacancy. Further, let
        $V^{(i)}=\cup_{j=\{0,1\}} W_{-i\mathbf{e}_1+j\mathbf{e}_2}$ for $i\in
        [2]$ and let $V=V^{(0)}\cup V^{(2)}$. Recall that by
        $\mathcal{A}^{(v)}$ the boxes in $V^{(i)}$ are $AC$-traversable and
        define further $\tilde{\mathcal{A}}$ as the event that $W_{\mathbf{j}}$ is $AC$-super
        for $\mathbf{j}\in V^{(2)}$. We can extend the variance to $V$ and
        use the block relaxation Lemma (\cref{lemma:block_relax}) to get
        \begin{align}
                \mu\left(\mathds{1}_{\mathcal{A}^{(v)}}
                        c_{\mathbf{0}}^{(v)}\var_{\mu_{\mathbf{0}}^{(*,BT)}}(f)\right)
                &\le \mu\left(\mathds{1}_{\mathcal{A}^{(v)}}
                        c_{\mathbf{0}}^{(v)}\var_{V}(f\tc
                        \mathcal{A}^{(v)})\right)\\
                &\le C\mu\left[\mathds{1}_{\mathcal{A}^{(v)}}
                        c_{\mathbf{0}}^{(v)}\left(\mathds{1}_{\tilde{\mathcal{A}}}
                                \var_{V^{(0)}}(f\tc
                                \mathcal{A}^{(v)})+\var_{{V^{(2)}}}(f\tc
                \mathcal{A}^{(v)})\right)\right].\label{eqn:ver_move_twosummands}
        \end{align}
        We upper bound the two summands separately. For the first term we get
        \begin{equation}\label{eqn:B_vert_path_method}
                \mu\left(
                        \mathds{1}_{\mathcal{A}^{(v)},\tilde{\mathcal{A}}}
                        c_{\mathbf{0}}^{(v)}
                                \var_{V^{(0)}}(f\tc \mathcal{A}^{(v)})
                        \right)
                \le 2^{\kappa \theta_B}
                \mu\left(\mathcal{D}_{\cup_{i\in [0,2]}
                        V^{(j)}}(f)\right).\label{eqn:vert_prop_1}
        \end{equation}
        This is done through the path method analogous to
        \cref{lemma:ac_super_prop} with the additional step of defining the
        paths for the $B$-vacancy on $4\mathbf{e}_2$ to move downwards after
        clearing any $A$ or $C$ vacancy on $V^{(0)}$ using the $AC$-super
        states in $V^{(2)}$.

        For the second summand first split up the variance
        \begin{equation}
                \var_{V^{(2)}}(f\mid\mathcal{A}^{(v)})
                \le \mu_{V^{(2)}}(
                \var^*_{-2\mathbf{e}_2}(f\mid\mathcal{A}^{(v)})
                +\var^*_{-\mathbf{e}_2+\mathbf{e}_1}(f\mid\mathcal{A}^{(v)})
                \mid\mathcal{A}^{(v)})
        \end{equation}
        and consider the variance over $W_{-2\mathbf{e}_2}$. The upper bound for
        the second term follows analogously.

        Consider an auxiliary process with the constraints
        $c^{(h)}_{\mathbf{j}}$ given by the indicator over the event that
        $W_{\mathbf{j}-\mathbf{e}_1}$ is $AC$-super. If $W_{\mathbf{j}}$ is
        unconstrained in this process, sample it from all $AC$-traversable
        states at a legal ring. This process has the same spectral gap as the
        East model with vacancy density $\frac{q_A q_C}{(q_A+q_C)^2}$. Using
        that $\mathcal{A}^{(v)}$ implies that there is an $AC$-super box to the
        left of $W_{\mathbf{e}_2}$ we can use the enlargement
        trick (\cite{chleboun2016relaxation}*{Lemma~3.6}, which immediately generalises to
        this case), to get
        \begin{align}
                \mu \left(\mathds{1}_{\mathcal{A}^{(v)}}
                        \var^*_{-2\mathbf{e}_2}(f\tc
                \mathcal{A}^{(v)})\right)
                        &\le C
                        \sum_{j=2}^{\lfloor \theta_B^{5/4}\rfloor-2}
                        \mu(c_{-j\mathbf{e}_1}^{(h)}\var^{*}_{-j\mathbf{e}_1}(f\tc
                        \text{$AC$-traversable}))\\
                        &\le C
                        \sum_{j=2}^{\lfloor \theta_B^{5/4}\rfloor-2}
                        \mu(\mathcal{D}_{W_{-j\mathbf{e}_1}\cup
                                W_{-(j+1)\mathbf{e}_1}}(f))
        \end{align}
        where in the second inequality we used \cref{lemma:ac_super_prop}.
        Combining the two estimates gives the claim after taking into account
        that the vertices in $V^{(2)}$ are counted twice which we absorb into
        $\kappa$.
\end{proof}
\begin{remark}
        Notice that here we lose the indicator over $\mathcal{A}^{(v)}$ since
        it requires there to be no $B$-vacancy between the central vertices but
        the path method adds these transitions. This will be important later,
        as keeping the indicators was important for taking the sum over the
        possible grids $\mathcal{C}$.
\end{remark}

The horizontal paths will consist of $B$-traversable $W_{\mathbf{j}}$ that
connect the vertical crossings. We isolate here the result that allows us to
propagate a central $B$ on these horizontal paths.

The basic situation is as follows. Let $\Gamma=\Gamma^{(l)}\cup \Gamma^{(r)}$
with $\Gamma^{(l)}= [-\lfloor \theta_B^{5/4}\rfloor
\mathbf{e}_1,\ldots,-\mathbf{e}_1]$ and $\Gamma^{(r)}=[\mathbf{0},\ldots,
\ell\mathbf{e}_1-1]$. Let $\mathcal{A}^{(h)}$ be the event that
$w_{\mathbf{j}}$ for $\mathbf{j}\in \Gamma^{(l)}$ is $AC$-traversable, that
there is an $AC$-super $W_{-i\mathbf{e}_1}$ for $i\le -3$, that
$W_{\mathbf{j}}$ for $\mathbf{j}\in\Gamma^{(r)}$ is $B$-traversable and that
$W_{\ell\mathbf{e}_1}$ is $B$-super (see \cref{fig:many_ac_hor}).
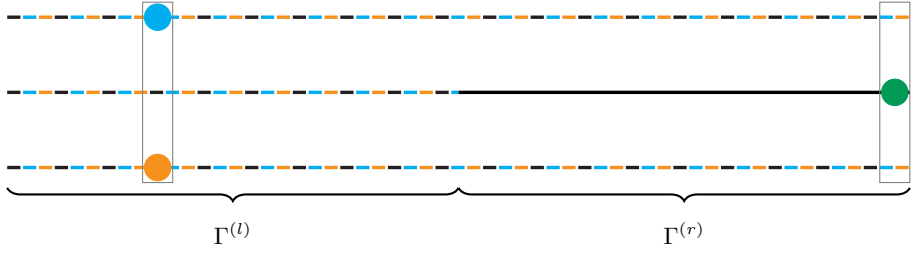
\begin{figure}
        \centering
        \begin{tikzpicture}[scale=2]
                 \tikzset{ground/.style={%
                        postaction={draw,line width=0.45mm,BurntOrange,
                                    dash pattern=on 5pt off 13pt,
                                    dash phase=6pt},
                          postaction= {draw,line width=0.45mm,Cyan,
                                    dash pattern=on 5pt off 13pt,
                                    dash phase=12pt},
                        Black,dash pattern= on 5pt off 13pt,
                                    line width=0.45mm}}
         \draw[ground] (0,0.5) -- (6,0.5);
         \draw[ground] (0,0) -- (3,0);
         \draw[line width=0.45mm] (3,0) -- (6,0);
         \draw[ground] (0,-0.5) -- (6,-0.5);
        \filldraw[Cyan](1,0.5)circle[radius=2.5pt] {};
        \filldraw[BurntOrange](1,-0.5)circle[radius=2.5pt] {};
        \filldraw[ForestGreen](5.9,0)circle[radius=2.5pt] {};
        \draw [thick,decorate,decoration={brace,amplitude=5pt,mirror},
               yshift=-1pt]
               (0,-0.6) -- (3,-0.6)
               node[black,midway,yshift=-0.6cm] {\footnotesize $\Gamma^{(l)}$};
        \draw [thick,decorate,decoration={brace,amplitude=5pt,mirror},
               yshift=-1pt]
               (3,-0.6) -- (6,-0.6)
               node[black,midway,yshift=-0.6cm] {\footnotesize $\Gamma^{(r)}$};

        \draw[gray] (5.8, -0.6) rectangle (6.0,0.6);
        \draw[gray] (0.9, -0.6) rectangle (1.1,0.6);
        \end{tikzpicture}
        \caption{Path $\Gamma$ in the context of 
        \cref{lemma:many_ac_hor}. The right box is the $B$-super box and the
        left box the $AC$-super box. The black path emphasizes that there is no
        condition on the central vertices in $B$-traversable boxes.\label{fig:many_ac_hor}}
\end{figure}
\begin{lemma}[Horizontal propagation]\label{lemma:many_ac_hor}
        For $\Gamma$ and $\mathcal{A}^{(h)}$ as above we find a constant
        $\kappa$ such that
        \begin{equation}
                \mu(\mathds{1}_{\mathcal{A}^{(h)}}
                \var^*_{\mathbf{0}}(f\tc \mathcal{A}^{(h)}))
                \le 2^{\kappa \theta_B^{3/2}}
                \mu(\mathcal{D}_{W(\Gamma)}(f)),
        \end{equation}
        where $W(\Gamma)=\cup_{\mathbf{j}\in \Gamma}W_{\mathbf{j}}$
\end{lemma}
\begin{proof}
        Split $W(\Gamma^{(r)})$ into $W^{(ro)}\cup W^{(rc)}$ of respectively the set
        of outer and central vertices. Define the event $\tilde{\mathcal{A}}$
        that there are only $C$-vacancies on the upper vertices of $W^{(ro)}$
        and only $A$-vacancies the lower vertices. Then, we can extend the
        variance (\cref{lemma:extend_variance}) and use the block relaxation Lemma (\cref{lemma:block_relax})
        to find a constant $\kappa$ such that 
        \begin{align}
                \mu(\mathds{1}_{\mathcal{A}^{(h)}}
                \var_{\mathbf{0}}^*(f\mid\mathcal{A}^{(h)}))
                &\le \mu(\mathds{1}_{\mathcal{A}^{(h)}}
                \var_{\Gamma^{(r)}}^*(f\mid\mathcal{A}^{(h)}))\\
                &\le 2^{\kappa\theta_B^{3/2}}
                \mu\left[\mathds{1}_{\mathcal{A}^{(h)}}
                \left(\mathds{1}_{\tilde{\mathcal{A}}}
                                \var_{W^{(rc)}}(f)
                        +\var^{(AC)}_{W^{(ro)}}(f)\right)\right].
                        \label{eqn:hor_move_two_sums}
        \end{align}
        Consider the first summand.
        On $\mathcal{A}^{(h)}$ there is a $B$-vacancy to the right of
        $W^{(rc)}$, so consider the auxiliary model with the standard
        $B$-vacancy constraints $c_x^B$ that samples from $\mu_x$ at a legal
        ring on $x\in W^{(rc)}$. Given $\mathds{1}_{\mathcal{A}^{(h)}}$ this
        auxiliary model on $W^{(rc)}$ has good boundary conditions and the same spectral gap as the
        East model with vacancy density $q_B$ so that
        by~\cite{chleboun2016relaxation}*{Theorem~2}
        \begin{equation}
                \mu(\mathds{1}_{\mathcal{A}^{(h)},\tilde{A}}\var_{W^{(rc)}}(f))
                \le 2^{\kappa\theta_B\log_2(\theta_B)}
                \sum_{x\in W^{(rc)}}
                \mu(\mathds{1}_{\tilde{A}}c_x^B\var_{x}(f)).
        \end{equation}
        Now write the variances as transition terms using
        \cref{lemma:var_as_trans} and use that with $\tilde{A}$ and $c_x^B$
        every $x\in W^{(rc)}$ is unconstrained for every transition so that
        \begin{equation}
                \sum_{x\in W^{(rc)}}
                \mu(\mathds{1}_{\tilde{A}}c_x^B\var_{x}(f))
                \le C \mu(\mathcal{D}_{W^{(rc)}}(f)).
        \end{equation}
        For the second summand in \cref{eqn:hor_move_two_sums} write
        $\var^{(AC)}_{W^{(ro)}}(f)$ as a sum of transition terms for $A$- and
        $C$-vacancy transitions. We saw in \cref{lemma:ac_super_prop} how an
        $AC$-super state can put any state on an $AC$-traversable state to its
        right. Given an $AC$-super and then an $AC$-traversable state we can
        thus put any state in $\{\star, A, C\}$ onto the upper or lower
        vertices of boxes right to them, if they don't contain $B$-vacancies.
        The legal path dynamic is completely analogous to the one in
        \cref{lemma:ac_super_prop} so we omit the details. The lengths of the paths
        are $O(|W^{(ru)}|)=O(\theta_B^{3/2})$, so the path method gives an
        upper bound of the order $2^{\kappa\theta_B^{3/2}}$
        and the claim follows.
\end{proof}
We now come to the grids we use in this section (see \cref{fig:many_ac_grid}).
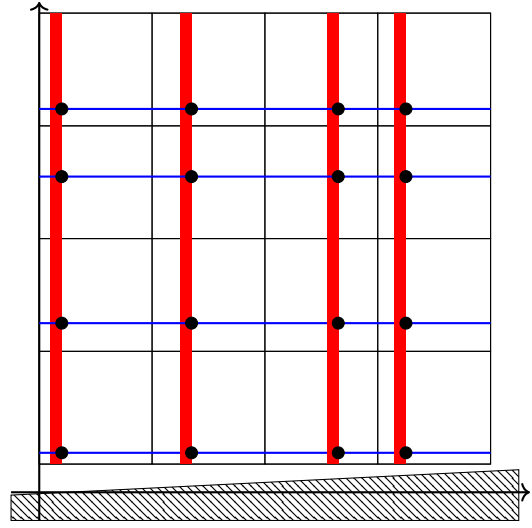
\begin{figure}[t]
        \centering
        \begin{tikzpicture}[scale=0.75]
                \draw[->,thick] (-0.5,0) -- (8.7,0);
                \draw[->,thick] (0,-0.5) -- (0,8.7);
                \draw[pattern=north west lines]  (-0.5,-0.05) -- (8.5,0.4) --
                        (8.5,-0.5) -- (-0.5,-0.5) -- cycle;
                \foreach \x in {0,2,4,6}{
                        \draw[] (0,\x+0.5) rectangle (8,\x+2.5);
                        \draw[] (\x,0.5) rectangle (\x+2,8.5);
                }
                \foreach \x in {0.2, 2.5, 5.1, 6.3}{
                        \fill[red] (\x,0.5) rectangle (\x+0.2,8.5);
                        \draw[blue, thick] (0,\x+0.5) -- (8,\x+0.5);
                }
                \foreach \x in {0.2, 2.5, 5.1, 6.3}{
                        \foreach \y in {0.2, 2.5, 5.1, 6.3}{
                                \filldraw[black](\x+0.2,\y+0.5)circle[radius=3pt] {};
                        }
                }
        \end{tikzpicture}
        \caption{Grid $\mathcal{C}$ as defined in \cref{def:many_ac_grid}. The horizontal
        configurations from \cref{lemma:many_ac_hor} are in blue, the vertical
        crossings from \cref{def:vert_cross} in red, a bit thicker to represent
        the horizontal extension of $\theta_B^{5/4}$. The black circles form the set
        $X(\mathcal{C})$. The striped area indicates the area on which we
        cannot condition by the exterior condition theorem together with the
        exhausting family of sets $\{V_n\}_{n\in \Z}$ defined above
        \cref{def:vert_cross}.\label{fig:many_ac_grid}}
\end{figure}
\begin{definition}[Grid]\label{def:many_ac_grid}
        Call a union of $\mathcal{C}=\cup_{i\in [N]} \mathcal{C}_i^{(h)}\cup
        \mathcal{C}_j^{(v)}$ a \emph{grid} if $\mathcal{C}_i^{(h)}\subset
        Q_i^{(h)}$ is a box of side length $(N\ell-1, 0)$ and
        $\mathcal{C}_i^{(v)}\subset Q_j^{(v)}$ is a box with side lengths
        $(\lfloor \theta_B^{5/4}\rfloor-1, N\ell-1)$. We call the grid \emph{good}
        if $W_{\mathbf{j}}$ is $B$-traversable for any $\mathbf{j}\in
        \cup_i\mathcal{C}_i^{(h)}$ and $\mathcal{C}_j^{(v)}$ is a vertical
        crossing for each $j\in [N-1]$.
\end{definition}
We have that $|\mathcal{C}_i^{(h)}\cap \mathcal{C}_j^{(v)}|=O(\theta_B^{5/4})$
and that on a grid we require this part to be $B$-traversable, $AC$-traversable
and to contain an $AC$-super box at the same time. This is well-defined since
$AC$-super states are a subset of $AC$-traversable states which in turn are
subsets of $B$-traversable states.

For a grid $\mathcal{C}$ let $X(\mathcal{C})$ be the vertices given by
$\mathbf{j}_{i,j}=\partial^{(r)}\mathcal{C}_i^{(v)}\cap \mathcal{C}_j^{(h)}$
for $i,j\in [N-1]$, where we recall that $\partial^{(r)}$ is the right
boundary. We define the event $\mathcal{E}^{(1)}$ as the $\omega\in \Omega^*$ such
that there is a good grid $\mathcal{C}$ and there are $i,j\in [N-1]$ with
$i,j>N/2$ such that $W_{\mathbf{j}_{i,j}}$ is $B$-super. The next lemma is
again a straightforward union bound.
\begin{lemma}\label{lemma:many_ac_failing_prob}
        For any $\varepsilon>0$ we find a $q(\varepsilon)$ such that
        \begin{equation}
                {(N\ell)}^2\mu(1-\mathds{1}_{\mathcal{E}^{(1)}}) \le \varepsilon
        \end{equation}
        if $q_B<q(\varepsilon)$.
\end{lemma}
Combining these events we can bring a $B$-super vertex to $\mathbf{j}_{0,0}$
for the respective good grid given by $\mathcal{E}^{(1)}$. As before, we need
to bring the $B$-super box to a deterministic vertex. Since the grid this time
starts at $\mathbf{e}_2$ we can immediately bring it back to the origin. Let
$\mathcal{E}^{(2)}$ be the event that $\mathbf{W}_{\mathbf{j}}$ is
\begin{itemize}
        \item $AC$-traversable for $\mathbf{j}$ either in
                $\Gamma^{(1)}:=\{-\lfloor
                \theta_B^{5/4}\rfloor\mathbf{e}_1,\ldots, -\mathbf{e}_1\}$ or
                $\Gamma^{(2)}:=\Gamma^{(1)}+\mathbf{e}_2$ and there is at least
                one $\mathbf{j}$ in both $\Gamma^{(1)}$ and $\Gamma^{(2)}$ with
                $W_{\mathbf{j}}$ $AC$-super.
        \item $B$-traversable for $\mathbf{\mathbf{j}}$ in
                $\Gamma^{(3)}:=\{\mathbf{e}_2,\ldots, (\ell-1)\mathbf{e}_2\}$
                (i.e.\ the left boundary of $Q_{0,0}$).
\end{itemize}
In an analogous calculation to \cref{lemma:many_ac_failing_prob} we get.
\begin{lemma}\label{lemma:many_ac_failing_prob_2}
        For any $\varepsilon>0$ we find a $q(\varepsilon)$ such that
        \begin{equation}
                \ell\mu(1-\mathds{1}_{\mathcal{E}^{(2)}}) \le \varepsilon
        \end{equation}
        if $q_B<q(\varepsilon)$.
\end{lemma}

Let $\mathcal{E}:=\mathcal{E}^{(1)}\cap \mathcal{E}^{(2)}$ and let
$\mathcal{E}_x$ be the event translated by $x\in \Z^2$. ${\{\mathcal{E}_x\}}_x$
satisfies the exterior condition w.r.t.\ the ${\{V_n\}}_{n\in \Z}$ defined
above and thus using
\cref{lemma:many_ac_failing_prob,lemma:many_ac_failing_prob_2} we get that we
can apply the exterior condition theorem, \cref{thm:exterior_thm}, with this
family of events. We come to the proof of part (3.iii).
\begin{proof}[Proof of \cref{thm:abc_relaxation}(3.iii)]
        By the exterior condition theorem we have
        \begin{equation}
                \var(f) \le 4\sum_{x\in
                \Z^2}\mu(\mathds{1}_{\mathcal{E}_x}\var_x(f)).
        \end{equation}
        
        Let us upper bound the summand for $x=0$. First use that
        $\mathrm{Supp}(\mathcal{E}_0)\cap W_{\mathbf{0}}=\emptyset$ to extend
        the variance (\cref{lemma:extend_variance})
        \begin{equation}
                \mu(\mathds{1}_{\mathcal{E}_0}\var_0(f))
                \le \mu(\mathds{1}_{\mathcal{E}_0}\var^*_{\mathbf{0}}(f)).
        \end{equation}
        For $\omega\in \mathcal{E}$ let $\mathcal{G}(\omega)$ denote the unique
        good grid in $\omega$ consisting of the lowest horizontal paths and
        vertical crossings in the $\prec$-order that make a good grid. Further
        let $\xi\in X(\mathcal{G})$ be the largest intersection point that is
        $B$-super in the lexicographic order. Let $\mathcal{E}_{\mathcal{C},
        \mathbf{j}_{i,j}}$ be the event $\mathcal{E}$ with
        $\mathcal{G}=\mathcal{C}$ and $\xi=\mathbf{j}_{i,j}$. We have
        \begin{equation}
                \mu(\mathds{1}_{\mathcal{E}_0}\var^*_{\mathbf{0}}(f))
                =\sum_{\text{$\mathcal{C}$ grid}}\sum_{n,m\in [N]}
                \mu(\mathds{1}_{\mathcal{E}_{\mathcal{C},\mathbf{j}_{n,m}}}
                \var^*_{\mathbf{0}}(f)).
        \end{equation}
        Further let $\mathcal{E}^{(i,j)}_{\mathcal{C}, \mathbf{j}_{n,m}}$ for
        $(i,j)\in [0,n]\times [0,m-1]$ be the part of the event
        $\mathcal{E}_{\mathcal{C},\mathbf{j}_{n,m}}$ that depends on the
        vertices outside the $i$-th vertical strip and $j$-th horizontal strip,
        if $i>n$ or $j>m-1$ let
        $\mathcal{E}^{(i,j)}_{\mathcal{C},\mathbf{j}_{n,m}}
        =\mathcal{E}_{\mathcal{C},\mathbf{j}_{n,m}}$. We have
        \begin{equation}
                \sum_{\text{$\mathcal{C}$ grid}}\sum_{n,m\in [N]}
                \mathds{1}_{\mathcal{E}^{(i,j)}_{\mathcal{C},\mathbf{j}_{n,m}}}
                \le 2\ell
        \end{equation}
        since only the grid outside of the $Q_i^{(h)}$ and $Q_j^{(v)}$ is
        fixed and inside these strips there are at most $\ell$ choices of
        straight horizontal paths or boxes that could be vertical crossings
        respectively (in the latter case $\ell$ is a rough estimate of
        $\ell/\lfloor \theta_B^{5/4}\rfloor$).

        Fix a grid $\mathcal{C}$ and $n,m\in [N]$, extend the variance
        (\cref{lemma:extend_variance}) and use the block relaxation Lemma
        (\cref{lemma:block_relax}) to get
        \begin{equation}
                \mu(\mathds{1}_{\mathcal{E}_{\mathcal{C},\mathbf{j}_{n,m}}}
                \var^*_{\mathbf{0}}(f))
                \le
                2^{\kappa\theta_B}
                \mu\left[(\mathds{1}_{\mathcal{E}_{\mathcal{C},\mathbf{j}_{n,m}}}
                \left(\mathds{1}_{\text{$W_{\mathbf{j}_{0,0}}$
                $B$-super}}\var^*_{\mathbf{0}}(f)
                +\var^{(*,BT)}_{\mathbf{j}_{0,0}}(f)\right)
                \right]\label{eqn:many_ac_two_sum_1}.
        \end{equation}
        We extend the variance in the first summand to $\{0,\mathbf{e}_2\}$ and then
        use the block relaxation Lemma again:
        \begin{align}
                \mu&(\mathds{1}_{\mathcal{E}_{\mathcal{C},\mathbf{j}_{n,m}},
                \text{$W_{\mathbf{j}_{0,0}}$
                $B$-super}}\var^*_{\mathbf{0}}(f)))\\
                   &\le 2^{\kappa\theta_B}
                \mu\left[\mathds{1}_{\mathcal{E}_{\mathcal{C},\mathbf{j}_{n,m}},
                \text{$W_{\mathbf{j}_{0,0}}$ $B$-super}}
                (\mathds{1}_{\text{$W_{\mathbf{e}_2}$ $B$-super}}
                \var^*_{\mathbf{0}}(f)+\var^{(*,BT)}_{\mathbf{e}_2}(f))\right].
                \label{eqn:many_ac_two_sum_2}
        \end{align}
        For the second summand in \cref{eqn:many_ac_two_sum_2} there is a
        unique shortest path $\Gamma$ from $\mathbf{e}_2$ to $\mathbf{j}_{0,0}$
        first on the bottom boundary of $D_{0,0}$ and then following the grid
        $\mathcal{C}$. Through a combination of extending the variance, the
        block relaxation Lemma, \cref{lemma:many_ac_hor,lemma:many_ac_ver} we
        get
        \begin{equation}
                \mu\left[\mathds{1}_{\mathcal{E}_{\mathcal{C},\mathbf{j}_{n,m}},
                \text{$W_{\mathbf{j}_{0,0}}$ $B$-super}}
                \var^{(*,BT)}_{\mathbf{e}_2}(f)\right]
                \le 2^{\kappa\theta_B^{3/2}}
                \mu\left[\mathds{1}_{\mathcal{E}^{(0,0)}_{\mathcal{C},\mathbf{j}_{n,m}}}
                \mathcal{D}_{\Gamma\cup\mathrm{Supp}(\mathcal{E}^{(2)})}(f)\right].
        \end{equation}
        Analogously for the first term in \cref{eqn:many_ac_two_sum_2} using
        \cref{lemma:many_ac_ver}. We can then take the sum over $\mathcal{C}$,
        $n$ and $m$ and absorb the overcounting of the vertices in
        $\mathrm{Supp}(\mathcal{E}^{(2)})$ into $\kappa$ above for $q_B$ small enough.

        For the second summand in \cref{eqn:many_ac_two_sum_1} we use
        completely analogous techniques to the proofs of the two-dimensional
        relaxation on the grids in part (i) and (ii), where here the $B$-super
        state corresponds to the vacancy state and the $B$-traversable state to
        the particle state of the auxiliary two-dimensional East model on the
        intersection points. Recovering the spectral gap of the $ABC$-model
        follows the same one-dimensional techniques from the first summand of
        \cref{eqn:many_ac_two_sum_1}.
\end{proof}
\section*{Acknowledgements}
The present work resulted from my Ph.D.\ thesis~\cite{phdthesis}, which also
contains some more details of the proofs. I wish to thank my supervisor Fabio
Martinelli who brought me to the original paper~\cite{garrahan2003coarse} and
has guided me in finding the above results for the multicolour East model.

\bibliographystyle{abbrv}

\begin{bibdiv}
 \begin{biblist}
\bib{cancrini2008kcm}{article}{
      author={Cancrini, Nicoletta},
      author={Martinelli, Fabio},
      author={Roberto, Cyril},
      author={Toninelli, Cristina},
       title={Kinetically constrained spin models},
        date={2008},
     journal={Probab. Theory Rel.},
      volume={140},
      number={3-4},
       pages={459\ndash 504},
  url={http://www.ams.org/mathscinet/search/publications.html?pg1=MR&s1=MR2365481},
}
\bib{chleboun2015mixing}{article}{
      author={Chleboun, Paul},
      author={Faggionato, Alessandra},
      author={Martinelli, Fabio},
      title = {Mixing time and local exponential ergodicity of the {East-like} process in $\mathbb{Z}^d$},
     journal = {Annales de la Facult\'e des sciences de Toulouse : Math\'ematiques},
     pages = {717--743},
     publisher = {Universit\'e Paul Sabatier, Toulouse},
     volume = {Ser. 6, 24},
     number = {4},
     year = {2015},
     doi = {10.5802/afst.1461},
    }
\bib{chleboun2016relaxation}{article}{
title={Relaxation to equilibrium of generalized {East} processes on $\mathbb{Z}^d$
 : Renormalization group analysis and energy-entropy competition},
author={Chleboun, Paul},
author={Faggionato, Alessandra},
author={Martinelli, Fabio},
journal={The Annals of Probability},
volume={44},
number={3},
pages={1817--1863},
year={2016},
publisher={Institute of Mathematical Statistics}
}

\bib{couzinie2022front}{misc}{
  doi = {10.48550/ARXIV.2112.14693},
  url = {https://arxiv.org/abs/2112.14693},
  author = {Couzinié, Yannick},
  author = {Martinelli, Fabio},
  title = {On a front evolution problem for the multidimensional East model},
  publisher = {arXiv},
  year = {2021},
}
\bib{phdthesis}{thesis}{
    title    = {The multidimensional East model: a multicolour model and a front evolution problem},
    school   = {Roma Tre University},
    type = {Ph.D. thesis}
    author   = {Couzini\'{e}, Yannick},
    year     = {2022}, 
    note = {Available on my website
    \url{https://yannick-couzinie.github.io/theses/phd-thesis}},
}

\bib{faggionato2012east}{article}{
  title={The East model: recent results and new progresses},
author={Faggionato, Alessandra},
author={Martinelli, Fabio},
author={Roberto, Cyril},
author={Toninelli, Cristina},
  journal={arXiv preprint arXiv:1205.1607},
  year={2012}
}


\bib{garrahan2003coarse}{article}{
  title={Coarse-grained microscopic model of glass formers},
  author={Garrahan, Juan P.},
  author={Chandler, David},
  journal={Proceedings of the National Academy of Sciences},
  volume={100},
  number={17},
  pages={9710--9714},
  year={2003},
  publisher={National Acad Sciences}
}
\bib{garrahan2011kinetically}{article}{
  title={Kinetically constrained models},
  author={Garrahan, Juan P},
  author={Sollich, Peter},
  author={Toninelli, Cristina},
  journal={Dynamical heterogeneities in glasses, colloids, and granular media},
  volume={150},
  pages={111--137},
  year={2011},
  publisher={International Series of Monographs on Physics}
}

\bib{gine2006lectures}{book}{
  title={Lectures on Probability Theory and Statistics: Ecole D'Et{\'e} de Probabilit{\'e}s de Saint-Flour XXVI-1996},
author={Gin{\'e}, Evarist},
author={Grimmett, Geoffrey R},
author={Saloff-Coste, Laurent},
  year={2006},
  publisher={Springer}
}
\bib{hartarsky2021universality}{article}{
  title={Universality for critical KCM: finite number of stable directions},
  author={Hartarsky, Ivailo and Martinelli, Fabio and Toninelli, Cristina},
  journal={The Annals of Probability},
  volume={49},
  number={5},
  pages={2141--2174},
  year={2021},
  publisher={Institute of Mathematical Statistics}
}
\bib{kordzakhia2006ergodicity}{article}{
  title={Ergodicity and mixing properties of the Northeast model},
  author={Kordzakhia, George},
  author={Lalley, Steven P.},
  journal={Journal of applied probability},
  volume={43},
  number={3},
  pages={782--792},
  year={2006},
  publisher={Cambridge University Press}
}
\bib{liggett1985interacting}{book}{
  title={Interacting particle systems},
  author={Liggett, Thomas Milton},
  volume={2},
  year={1985},
  publisher={Springer}
}

\bib{liggett2010continuous}{book}{
  title={Continuous time Markov processes: an introduction},
  author={Liggett, Thomas Milton},
  volume={113},
  year={2010},
  publisher={American Mathematical Soc.}
}
\bib{mareche2019exponential}{article}{
  title={Exponential convergence to equilibrium for the $ d $-dimensional East model},
  author={Mar{\^e}ch{\'e}, Laure},
  journal={Electronic Communications in Probability},
  volume={24},
  pages={1--10},
  year={2019},
  publisher={Institute of Mathematical Statistics and Bernoulli Society}
}

\bib{martinelli2019towards}{article}{
  title={Towards a universality picture for the relaxation to equilibrium of kinetically constrained models},
  author={Martinelli, Fabio},
  author={Toninelli, Cristina},
  journal={The Annals of Probability},
  volume={47},
  number={1},
  pages={324--361},
  year={2019},
  publisher={Institute of Mathematical Statistics}
}
\bib{martinelli2020diffusive}{inproceedings}{
  title={Diffusive scaling of the {Kob}-{Andersen} model in $\mathbb{Z}^d$},
  author={Martinelli, Fabio},
  author={Shapira, Assaf},
  author={Toninelli, Cristina},
  booktitle={Annales de l'Institut Henri Poincar{\'e}, Probabilit{\'e}s et Statistiques},
  volume={56},
  number={3},
  pages={2189--2210},
  year={2020},
  organization={Institut Henri Poincar{\'e}}
}
\bib{shapira2020kinetically}{article}{
title={Kinetically constrained models with random constraints},
author={Shapira, Assaf},
journal={The Annals of Applied Probability},
volume={30},
number={2},
pages={987--1006},
year={2020},
publisher={Institute of Mathematical Statistics}
}
\end{biblist}
\end{bibdiv}

\end{document}